\documentclass[11pt,a4paper,leqno]{amsproc}

\usepackage{anysize}
\marginsize{3.2cm}{3.5cm}{3cm}{3cm}

\usepackage[]{hyperref}
\hypersetup{
    colorlinks=true,       
    linkcolor=red,          
    citecolor=blue,        
    filecolor=magenta,      
    urlcolor=cyan           
}
\usepackage{tikz}
\usepackage{amsmath}
\usepackage{amsfonts,amssymb}
\usepackage{enumerate}
\usepackage{mathrsfs}
\usepackage{amsthm}

\theoremstyle{plain}
\newtheorem{theo}{Theorem}[section]
\newtheorem{prop}[theo]{Proposition}
\newtheorem{lemm}[theo]{Lemma}
\newtheorem{coro}[theo]{Corollary}

\newtheorem{defi}[theo]{Definition}
\theoremstyle{definition}
\newtheorem{rema}[theo]{Remark}

\DeclareMathOperator{\cnx}{div}
\DeclareMathOperator{\cn}{div}

\DeclareMathOperator{\supp}{supp}

\DeclareSymbolFont{pletters}{OT1}{cmr}{m}{sl}
\DeclareMathSymbol{s}{\mathalpha}{pletters}{`s}
\def\even#1{ #1^{\text{ev}}}
\def\odd#1{#1^{\text{od}}}

\def\B{B }
\def\cnxy{\cn_{x,y}}

\def\defn{\mathrel{:=}}
\def\deta{\eta}

\def\dzeta{\zeta}
\def\dB{B}
\def\dV{V}

\def\eps{\varepsilon}

\def\la{\left\lvert}
\def\lA{\left\lVert}
\def\le{\leq}

\def\ma{a}
\def\mez{\frac{1}{2}}
\def\partialx{\nabla}
\def\partialyx{\nabla_{x,y}}
\def\ra{\right\rvert}
\def\rA{\right\rVert}
\def\s{\sigma}
\def\tdm{\frac{3}{2}}

\def\xC{\mathbf{C}}
\def\xN{\mathbf{N}}
\def\xR{\mathbf{R}}
\def\xS{\mathbf{S}}
\def\xT{\mathbf{T}}
\def\xZ{\mathbf{Z}}

\numberwithin{equation}{section}

\title{Cauchy theory for the gravity water waves system with non localized initial data}
\author{
T. Alazard, 
\address{T. Alazard. D\'epartement de Math\'ematiques et Applications, UMR 8553 du CNRS \\ {\'E}cole Normale Sup\'erieure, 45, rue d'Ulm 75005 Paris Cedex, France}
N. Burq, 
\address{N. Burq. Laboratoire de Math\'ematiques d'Orsay, UMR 8628 du CNRS, Universit\'e Paris-Sud, 91405 Orsay Cedex, France 
et D\'epartement de Math\'ematiques et Applications, UMR 8553 du CNRS, {\'E}cole Normale Sup\'erieure, 45, rue d'Ulm 75005 Paris Cedex, France}
C. Zuily
\address{C. Zuily. Laboratoire de Math\'ematiques d'Orsay, UMR 8628 du CNRS, Universit\'e Paris-Sud, 91405 Orsay Cedex, France}
}
\thanks{
T.A. was supported by the French Agence Nationale de la Recherche, projects ANR-08-JCJC-0132-01 and ANR-08-JCJC-0124-01. 
}

\date{\empty}
\begin{document}
\begin{abstract}
In this article, we develop the local Cauchy theory for the gravity water waves system, 
for rough initial data which do not decay at infinity. We work in the context of 
$L^2$-based uniformly local Sobolev spaces introduced by Kato~(\cite{Kato}). 
We prove a classical well-posedness result (without loss of derivatives). 
Our result implies also a local well-posedness 
result in H\"older spaces (with loss of $d/2$ derivatives). 
As an illustration, we solve a question raised by Boussinesq in~\cite{Boussinesq} on the  water waves problem in a canal. 
We take benefit of an elementary 
observation to show that the strategy suggested in~\cite{Boussinesq} does indeed apply to this setting.
\end{abstract} 

\maketitle
\section{Introduction}
We are interested in this paper in the free boundary problem describing the motion of an 
incompressible, irrotational fluid flow moving under the force of gravitation,  without surface tension, 
in case where the initial data are neither localized nor periodic. 
There are indeed two cases where the mathematical analysis is rather well understood: 
firstly for periodic initial data (in the classical Sobolev spaces $H^s(\xT^d)$) and secondly when 
they are decaying to zero at infinity (for instance for data in $H^s(\xR^d)$ with $s$ large enough). 
With regards to the analysis of the Cauchy problem, we refer to the recent papers 
of Lannes~\cite{LannesKelvin}, Wu~\cite{Wu09,Wu10} and Germain, Masmoudi and Shatah~\cite{GMS}. 
We also refer to the introduction of \cite{ABZ3} or \cite{BL,LannesLivre} for more references. 
However, one can think to the moving 
surface of a lake or a canal where the waves are 
neither periodic nor decaying to zero (see also~\cite{Favre}). 

A most natural  strategy would be to solve the Cauchy problem in the classical 
H\"older spaces $W^{k,\infty}(\xR^d)$. 
However even the linearized system at the origin (the fluid at rest) 
is ill-posed in these spaces (see Remark~\ref{re.illposed} below), 
and this strategy leads consequently to loss of derivatives. 
Having this loss of derivatives in mind, the other natural approach is to work in the framework 
of $L^2$ based uniformly local Sobolev spaces, denoted by $H^{s}_{ul}(\xR^d)$. 
These spaces were introduced by Kato (see~\cite{Kato}) in the analysis of hyperbolic systems. 
Notice however, that 
compared to general hyperbolic systems, the water waves system appears to be non local,   
which induces new difficulties. 
This framework appears to be quite 
natural in our context. Indeed, the uniformly local Sobolev spaces $H^s_{ul}(\xR^d)$ contain, 
in particular, the usual Sobolev spaces $H^s(\xR^d)$, the periodic 
Sobolev spaces $H^s(\xT^d)$ (spaces of periodic functions on $\xR^d$), 
the sum $H^s(\xR^d)+H^{s}(\xT^d)$ and also the H\"older spaces $W^{s,p}(\xR^d)$ 
(and as a by-product of our analysis, we get well-posedness in H\"older spaces, with a loss of derivatives).

The aim of this paper is precisely to prove that the water waves system is locally (in time) 
well posed in the framework of uniformly local Sobolev spaces. 
Moreover, following our previous paper \cite{ABZ3}, the data for which we solve the 
Cauchy problem are allowed to be quite rough. Indeed we shall assume, for instance, 
that the initial free surface is the graph of a function which belongs to the 
space $H^{s+\mez}_{ul}(\xR^d)$ for $s>1+ \frac{d}{2}$. In particular, in term of Sobolev 
embedding, the initial free surface is merely $W^{\frac{3}{2}, \infty}(\xR^d)$ thus may have 
unbounded curvature. On the other hand this threshold should be compared 
with the scaling of the problem. Indeed  it is known that the water wave system 
has a scaling invariance for which the critical space for the initial free surface 
is the space $\dot{H}^{1+\frac{d}{2}}(\xR^d)$ (or $W^{1,\infty}(\xR^d))$. 
This shows that we  solve here  the Cauchy problem for data $\mez$ above the scaling. 
(Notice that in \cite{ABZ4} we prove well-posedness, 
in the classical Sobolev spaces, $\mez-\frac{1}{12}$ above the scaling when $d =2$).

As an illustration of the relevance of this low regularity Cauchy theory in the context of {\em local} spaces, we solve a question raised by Boussinesq in 1910~\cite{Boussinesq} on the  water waves problem in a canal. In~\cite{Boussinesq}, Boussinesq suggested to reduce the water-waves system in a canal to the same system on $\xR^3$ with periodic conditions with respect to one variable, by a simple reflection/periodization procedure (with respect to the normal variable to the boundary of the canal). However, this idea remained inapplicable for 
the simple reason that the even extension of a smooth function 
on the half line is in general merely Lipschitz continuous (due to the singularity at the origin). As a consequence, even if one starts with a smooth initial domain, the reflected/periodized domain will only be Lipschitz continuous. Here, we are able to take benefit of an elementary (though seemingly previously unnoticed) observation which shows that actually, as soon as we are looking for {\em reasonably smooth} solutions, the angle between the free surface and the vertical boundary of the canal is a right angle. Consequently, the reflected/periodized domain enjoy additional smoothness (namely up to $C^{3}$), which is enough to apply our {\em rough data} Cauchy theory and to show that the strategy suggested in~\cite{Boussinesq} does indeed apply.  This appears to be the first result on Cauchy theory for the water-wave system in a domain with boundary. 

The present paper relies on the strategies developed in our previous paper \cite{ABZ3} 
and we follow the same scheme of proof. In Section~\ref{sec.3}, we develop the machinery of para-differential calculus in the framework of uniformly local spaces that we need later. We think that this section could be useful for further studies in this framework. In Section~\ref{sec.4} we prove that the Dirichlet-Neumann operator is well defined in this framework (notice that this fact is not straightforward, see~\cite{GVMa, DaPr} for related works), and we give a precise description (including sharp elliptic estimates in very rough domains) 
on these  spaces. In Section~\ref{sec.5}, we symmetrize the system and prove {\em a priori} estimates. In section~\ref{sec.6}, we prove contraction estimates and well posedness. In section~\ref{sec.7}, we give the application to the canal (and swimming pools). Finally, in an appendix, we prove that in the context of H\"older spaces, the water-waves system linearized on the trivial solution (rest) is {\em ill posed}.

\section{The problem and the result }
In this paper  we shall denote by $t\in \xR$   the time variable and by $x \in \xR^d $ 
(where $  d\geq 1),$ $ y\in \xR,$   the horizontal and vertical space variables. 
We work in a fluid domain with free boundary and fixed bottom on the form
\begin{align*}
\Omega &= \{(t,x,y) \in[0,T] \times \xR^d \times \xR: (x,y) \in \Omega(t)\} \text{ where } \\
\Omega(t) &= \{(x,y): \eta_* (x)< y < \eta(t,x)\}. 
\end{align*}
Here the free surface is described by $\eta$, an unknown of the problem,  
and the bottom by a given function $\eta_* $. We shall only assume that $\eta_*$ is bounded and continuous. 
We assume that the bottom is the graph of a function for the sake of simplicity: our analysis applies 
whenever one has the Poincar\'e inequality given by Lemma~\ref{poinc} below. 

We shall denote by $\Sigma$ the free surface and by $\Gamma$ the bottom,
 \begin{align*}
 \Sigma &= \{(t,x,y) \in [0,T] \times \xR^d \times \xR: (x,y) \in \Sigma(t)\} \quad\text{where} \\
 \Sigma(t)  &= \{(x,y) \in\xR^d \times \xR: y= \eta(t,x)\},\\
      \Gamma &=  \{(x,y) \in \xR^d \times \xR: y= \eta_*(x)\}.
 \end{align*}
We shall use the following notations
$$
\nabla_x = (\partial_{x_i})_{1 \leq i \leq d}, \quad \nabla_{x,y} 
= (\nabla_x, \partial_y), \quad \Delta_x = \sum_{1 \leq i \leq d}\partial_{x_i}^2, \quad \Delta_{x,y} = \Delta_x + \partial_y^2.
$$
\subsection{The equations}
The Eulerian velocity $v: \Omega \to \xR^{d+1}$ solves the incompressible 
and irrotational Euler equation
$$
 \partial_t v + (v \cdot \nabla_{x,y})v + \nabla_{x,y}P = -g e_y, \quad \text{div}_{x,y} \, v =0, 
 \quad  \quad \text{curl}_{x,y}\, v =0 \quad \text{in } \Omega
 $$
where $g>0$ is the acceleration of the gravity, $e_y$ the  vector $(x=0,y=1)$ and $P$   the pressure. 
The problem is then given by three boundary conditions:
\begin{itemize}
\item a kinematic condition (which states that the free surface moves with the fluid)
$$
\partial_t \eta = \sqrt{1+ \vert \nabla_x \eta \vert^2} (v \cdot n) \quad \text{on} \quad \Sigma,
$$
where $n$ denotes the unit normal vector  to $\Sigma,$
\item a dynamic condition (that expresses a balance of forces across the free surface)
$$ P = 0 \quad \text{ on }   \Sigma,$$
\item the ``solid wall" boundary condition on the bottom $\Gamma$
$$
v \cdot \nu =0 \quad \text{on } \Gamma,
$$
where $\nu$ denotes  the normal vector to $\Gamma$ whenever it exists.
\end{itemize}

Since the motion is incompressible and irrotational there 
exists a velocity potential $\phi: \Omega \to \xR$ 
such that $v = \nabla_{x,y} \phi, $ thus $\Delta_{x,y} \phi = 0$ in $\Omega$. 
We shall work with the Zakharov/Craig--Sulem formulation of the water waves equations. 
We introduce
$$
\psi(t,x)= \phi(t,x,\eta(t,x))
$$
and the Dirichlet-Neumann operator 
\begin{align*}
G(\eta) \psi &= \sqrt{1 + \vert \nabla_x \eta \vert ^2}
\Big( \frac{\partial \phi}{\partial n} \Big \arrowvert_{\Sigma}\Big)\\
&= (\partial_y \phi)(t,x,\eta(t,x)) - \nabla_x \eta(t,x) \cdot(\nabla_x \phi)(t,x,\eta(t,x)).
\end{align*}
Then (see \cite{CrSuSu} or \cite{CL}) the water waves system can be written in term of the unknown $\eta,\psi$ as 
\begin{equation}\label{ww}
\left\{
\begin{aligned}
\partial_t \eta &= G(\eta) \psi,\\
\partial_t \psi &= - \mez \vert \nabla_x \psi \vert^2 + \mez \frac{(\nabla_x \eta \cdot \nabla_x \psi + G(\eta)\psi)^2}{1+ \vert \nabla_x \eta \vert^2} - g \eta.
\end{aligned}
\right.
\end{equation}
It is useful to introduce the  vertical and horizontal components of the velocity. We set 
\begin{equation}\label{BV}
\left\{
\begin{aligned}
B &= (v_y)\arrowvert_\Sigma = \frac{ \nabla_x \eta \cdot \nabla_x \psi + G(\eta)\psi} {1+ \vert \nabla_x \eta \vert^2},\\
V&= (v_x)\arrowvert_\Sigma  =\nabla_x \psi - B \nabla_x \eta.
\end{aligned}
\right.
 \end{equation}
 We recall also that the Taylor coefficient defined by $a = -\frac{\partial P}{\partial y}\big\arrowvert_\Sigma$ can be defined in terms of $\eta,\psi,B,V$ only (see \S 7.2 below  and \S 4.3.1 in \cite{LannesLivre}).
\subsection{The uniformly local Sobolev spaces}
We recall here the definition of the uniformly local Sobolev spaces   introduced by Kato in \cite{Kato}.

Recall that there exists $\chi \in C ^\infty(\xR^d)$ with
$\supp \chi \subset [-1,1]^d, \chi = 1$ near $ [-\mez,\mez]^d$   such that 
\begin{equation}\label{kiq}
\sum_{q \in \xZ^d}\chi_q(x) = 1,\quad \forall   x\in \xR^d
\end{equation}
where $$
\chi_q(x) = \chi(x-q).
$$ 

\begin{defi}
For $s\in \xR$   the space $H^s_{ul}(\xR^d)$ is the space of distributions $u\in H^s_{loc}(\xR^d)$  such that 
$$\Vert u\Vert_{H^s_{ul}(\xR^d)}:= \sup_{q\in \xZ^d}\Vert \chi_q u\Vert_{H^s (\xR^d)} < +\infty.$$
\end{defi}
The definition of the space $H^s_{ul}(\xR^d)$ is independent 
of the choice of the function $\chi$ in $C_0^\infty(\xR^d)$ satisfying~\eqref{kiq} (see Lemma \ref{invariance} below).

\begin{prop}One has the following embeddings:
\begin{enumerate}
\item  If $s> \frac{d}{2}$ and $s-\frac{d}{2} \notin \xN,$  $H^s_{ul}(\xR^d)$ is  continuously embedded in 
$W^{s- \frac{d}{2} ,\infty}(\xR^d)$.
\item If $m\in \xN,$ $W^{m,\infty}(\xR^d)$ is continuously embedded in $H^{m}_{ul}(\xR^d)$.
 \item If $s \geq 0, $  $W^{s+\eps, \infty}(\xR^d)$ is continuously embedded in $H^s_{ul}(\xR^d) $  for $\eps>0 $.
\end{enumerate}
\end{prop}

\subsection{The main result}
The goal of this article  is to prove the following result.
\begin{theo}\label{theo:princ}Let $d\geq 1, s>1+ \frac{d}{2}$. Assume that $\eta_*$ 
is a bounded continuous function on~$\xR^d$. 
Consider an initial data $(\eta_0, \psi_0)$ satisfying the following conditions

\begin{enumerate}[(i)]
\item 
$ \eta_0 \in H^{s+\mez}_{ul}(\xR^d),  \psi_0 \in  H^{s+\mez}_{ul}(\xR^d),  V_0 \in H^{s}_{ul}(\xR^d), B_0 \in   H^{s}_{ul}(\xR^d),  $

\item there exists $h>0$ such that $\eta_0(x)  -\eta_*(x) \geq 2 h, \quad \forall x \in \xR^d,$

\item there exists  $c>0$ such that $a_0(x) \geq c, \quad \forall x \in \xR^d$,
\end{enumerate}
where $a_0$ denotes the Taylor coefficient at time $t=0$.

Then there exists $T>0$ such that the Cauchy problem for the system \eqref{ww} with initial data $(\eta_0, \psi_0)$ at $t=0$ has a unique solution  
$$(\eta, \psi) \in L^\infty\big([0,T], H^{s+\mez}_{ul}(\xR^d) \times H^{s+\mez}_{ul}(\xR^d)\big)
$$
such that 

\begin{enumerate}[1.]

\item $(V,B) \in L^\infty\big([0,T], H^{s }_{ul}(\xR^d) \times H^{s }_{ul}(\xR^d)\big)$,

\item $\eta(t,x) - \eta_*(x) \geq  h, \quad \forall (t,x) \in [0,T]\times \xR^d,$

\item $a(t,x) \geq \mez c, \quad \forall (t,x) \in [0,T]\times \xR^d$.

\item For any $s'<s$, 
$$
(\eta, \psi, V, B) \in C^0\big([0,T], H^{s'+\mez}_{ul}(\xR^d) \times H^{s'+\mez}_{ul}(\xR^d) \times H^{s' }_{ul}(\xR^d) \times H^{s' }_{ul}(\xR^d)\big).
$$
\end{enumerate}
\end{theo}
\begin{rema}\label{re.illposed}
$\bullet$ 
Theorem~\ref{theo:princ}  implies local well posedness  in H\" older spaces: indeed, assuming that 
\begin{multline*}
(\eta_0, \psi_0, V_0, B_0) \in W^{\sigma+ \mez+ \eps, \infty}(\xR^d)\times W^{\sigma+\mez+ \eps, \infty}(\xR^d)\times W^{\sigma+ \eps, \infty}(\xR^d)\times W^{\sigma+ \eps, \infty}(\xR^d) \\
\subset H^{\sigma+ \mez+\frac{\eps}{2}}_{ul}( \xR^d)\times H^{\sigma+\mez+\frac{\eps}{2}}_{ul}(\xR^d)\times H^{\sigma+\frac{\eps}{2}}_{ul}(\xR^d)\times  H^{\sigma+\frac{\eps}{2}}_{ul}(\xR^d)
\end{multline*}
for some $\sigma>1+d/2$ and $\eps>0$, then we get 
 a solution 
 \begin{multline*}
 (\eta, \psi, V, B)  \in C^0([0,T], H^{\sigma+ \mez}_{ul}( \xR^d)\times H^{\sigma+\mez}_{ul}(\xR^d)\times H^{\sigma}_{ul}(\xR^d)\times  H^{\sigma}_{ul}(\xR^d))\\
 \subset C^0([0,T], W^{\sigma+ \mez- \frac d2, \infty}(\xR^d)\times W^{\sigma+\mez- \frac d2, \infty}(\xR^d)\times W^{\sigma- \frac d2, \infty}(\xR^d)\times W^{\sigma- \frac d2, \infty}(\xR^d)).
 \end{multline*}

$\bullet$ It is very likely that this loss of $d/2$ derivatives cannot be completely avoided.  Indeed the linearized water waves equation around the zero solution can be  written as 
 $$\partial_t u +i  \la D_x\ra^{\mez} u=0.$$
The solution of this equation, with initial data 
$u_0$, is given by
$$
u(t)=S(t)u_0, \qquad 
S(t)=\exp(-i t \la D_x\ra^{\mez}).
$$
Proposition \ref{perte} shows that for $t\neq 0$ the operator $S(t)$ 
 is not bounded from the Zygmund space $C _*^{\sigma}(\xR^d)$ to $C^{s}_*(\xR^d)$ 
 if $s > \sigma-\frac{d}{4} $, remembering that $C^\sigma_*(\xR^d)=W^{\sigma,\infty}(\xR^d)$ if 
 $\sigma \ge 0$, $\sigma \not\in\xN$. (For positive results see Fefferman and Stein  \cite[page 160]{FS}). 
 Thus even in the linear case we have  a loss of $\frac{d}{4}$ derivative. 
 
$\bullet$ The result in the appendix also shows that, in the presence of surface tension, 
a similar well posedness result in the framework of uniformly local Sobolev space is rather unlikely to hold. 
Indeed, in the presence of surface tension,  the linearized operator 
around the solution $(\eta, \psi) = (0,0)$ can be written 
(see \cite{ABZ1}) with  $ u = \vert D \vert ^\mez \eta + i \psi$ 
as
$$
\partial_t u + i \vert D_x \vert^{\frac{3}{2}} u = 0, \quad u\arrowvert_{t=0} = u_0.
$$
According to Proposition \ref{perte} the loss of derivatives in $x$ from  $u_0$ to the solution $u(t,\cdot), t \neq0,$ is at least $\frac{3d}{4}$ whereas an analogue of the above theorem would give a loss of at most $\frac{d}{2}.$
\end{rema}


\section{The Dirichlet-Neumann operator}\label{sec.4}

\subsection{Definition of the Dirichlet-Neumann operator}
 
For $d \geq 1$ we set 
 \begin{equation}\label{omega}
\left\{
\begin{aligned} 
   \Omega &=\{(x,y)\in \xR^{d+1}: \eta_*(x) <y < \eta(x)\},\\
   \Sigma&= \{(x,y)\in \xR^{d+1}: y=\eta(x)\} 
   \end{aligned}
   \right.
   \end{equation}
where $\eta_*$ is a fixed bounded continuous function on $\xR^d$ and $\eta \in W^{1,\infty}(\xR^d)$.
We shall assume that there exists $h>0$ such that 
 \begin{equation}\label{condition}
       \{(x,y) \in \xR^{d+1} : \eta(x) -h \leq y <\eta(x) \} \subset \Omega.
  \end{equation}
  In \cite{ABZ3} the Dirichlet-Neumann operator $G(\eta)$ associated to $\Omega$ has been defined as a continuous operator from $H^{ \mez}(\xR^d)$ to $H^{- \mez}(\xR^d)$. Our aim here is to prove that it has a unique extension to the space $H_{ul}^{ \mez}(\xR^d)$ (see Theorem \ref{def:DN} below).  Define first the space $H^1_{ul}(\Omega)$ by 
 $$ u \in H^1_{ul}(\Omega) \Leftrightarrow \Vert u \Vert_{H^1_{ul}(\Omega)}:= \sup_{q\in \xZ^d} \Vert  \chi_q u \Vert_{H^1(\Omega)} < + \infty.$$
 Each element  $u \in H_{ul}^1(\Omega)$ have a trace  on $\Sigma $ (see below) which will be denoted by  $\gamma_0u$. We introduce   the subspace $H_{ul}^{1,0}(\Omega) \subset H_{ul}^{1}(\Omega)$  defined by 
    $$H_{ul}^{1,0}(\Omega) = \{u \in H_{ul}^{1 }(\Omega): \gamma_0 u =0\}.$$
 
 Then we have the following Poincar\'e inequality.
 \begin{lemm}\label{poinc}
 There exists  $C>0$ depending on  $\Vert \eta\Vert_{L^\infty (\xR^d)} + \Vert \eta_*\Vert_{ L^\infty(\xR^d)}  $  such that for    $\alpha \in C_0^\infty(\xR^d)$   non negative and   $u\in H_{ul}^{1,0}(\Omega)$ we have
  $$\iint_{\Omega} \alpha(x) \vert u(x,y) \vert^2 \, dx dy \leq C \iint_{\Omega} \alpha(x) \vert \partial_y u(x,y) \vert^2 \, dx dy. $$
\end{lemm}
\begin{proof}
 Let ${u} \in  {H}^{1,0}_{ul}({\Omega})$.  
It is easy to see that  there exists  a sequence $({u}_n)$ of functions which are $C^1$ in ${\Omega}$ and vanish near the top boundary $y = \eta(x)$ such that
$$
\lim_{n\to + \infty}\| u_n - u \|_{H^1( \Omega\cap \{ |x| \leq K\})}=0.
$$
As a consequence, it is enough to prove the result for such functions. Let $\alpha \in C_0^\infty(\xR^d), \alpha \geq0$.
We can write
$$ u(x,y) =\int_{  \eta(x)} ^y \partial_s u(x,s) ds$$
from which we deduce 
$$\alpha(x) \vert u(x,y)\vert^2 \leq  \| \eta - \eta_*\|_{L^\infty} \alpha(x)  \int_{\eta_*(x)}^{\eta(x)}\vert \partial_s u(x,s)\vert^2 ds.$$
Integrating this inequality on ${\Omega}$  we obtain, 
$$\iint_{\Omega}\alpha(x) \vert u(x,y)\vert^2 dx\,dy\leq  \| \eta - \eta_*\|_{L^\infty} \iint_{{\Omega} }\alpha(x) \vert \partial_y{u}(x,y)\vert^2 \, dx dy.$$
\end{proof}
\begin{rema}
Let
$$
{H}^{1,0}({\Omega}) = \bigl\{{u} \in L^2({\Omega}) \,:\, 
\nabla_{x,y} {u} \in L^2({\Omega}),\, \text{ and }  {u}\arrowvert_{y = \eta(x) }=0\bigr\},
$$
then we also have the Poincar\'e inequality
\begin{equation}\label{encore:poinc}
 \iint_{{\Omega}}  \alpha(x)\vert {u} (x,y)\vert^2 \leq  C  \iint_{{\Omega} } \alpha(x) \vert \partial_y {u} (x,y)\vert^2 \, dx dy
\end{equation}
for all ${u} \in {H}^{1,0} ({\Omega})$, $\alpha \in C^\infty_b( \xR^d)$, $\alpha\ge 0$, with a constant $C$ independent of $\alpha$. Indeed, 
this follows from the same computation as above using 
the fact that any ${u} \in {H}^{1,0} ({\Omega})$ 
can be approximated by a sequence of functions which are $C^\infty$ in ${\Omega}$ 
and vanish near $y=\eta(x)$.
 \end{rema}
\begin{prop}\label{existe:phi}
For every $\psi \in H_{ul}^\mez(\xR^d)$ the problem
\begin{equation}\label{eqPhi}
 \Delta_{x,y} \Phi = 0 \text{ in } \Omega, \quad  \Phi\arrowvert_{\Sigma} = \psi, \quad \frac{\partial \Phi}{\partial \nu}\arrowvert_ \Gamma = 0,
 \end{equation}
has a unique solution $\Phi \in H^1_{ul}(\Omega)$ 
and there exists a function $\mathcal{F} :\xR^+ \to \xR^+$ 
independent of $(\psi,\eta)$ such that 
$$\Vert \Phi \Vert_{H^1_{ul}(\Omega)} \leq \mathcal{F}(\Vert \eta \Vert_{W^{1,\infty}(\xR^d)}) \Vert \psi\Vert_{H_{ul}^\mez(\xR^d)}.$$
 \end{prop}
\begin{proof}
Before giving the proof we have to precise the meaning of the 
boundary condition $\frac{\partial \Phi}{\partial \nu}\arrowvert_ \Gamma = 0$ 
since $\Gamma $ is only $C^0$. This condition means that 
\begin{equation}\label{green1}
  \iint_\Omega \nabla_{x,y}\Phi(x,y) \cdot\nabla_{x,y}( \alpha(x)\theta(x,y)) \, dx\, dy = 0
 \end{equation}
for every $\theta \in H^1(\Omega)$ (the usual Sobolev space)   with $\supp \theta \subset   \{(x,y):  \eta_*(x) \leq y \leq \eta_*(x)+ \eps \}$ 
for a small $\eps>0 $ and every $\alpha \in C_0^\infty(\xR^d)$.

Notice that if $\eta_*\in W^{2,\infty}(\xR^d)$  the Green formula (see \cite{GRI} p.62) 
shows that \eqref{green1}  is equivalent to $\frac{\partial \Phi}{\partial \nu}\arrowvert_ \Gamma = 0$.

\begin{lemm}\label{GREEN}
 We have 
\begin{equation}\label{green2}
 \iint_\Omega \nabla_{x,y}\Phi(x,y) \cdot \nabla_{x,y}(\alpha(x)\theta(x,y)) \, dx\, dy = 0 
 \end{equation}     
for every $\theta \in H^1(\Omega)$ with $\gamma_0 \theta = 0$ and every $\alpha \in C_0^\infty(\xR^d)$.
\end{lemm}
\begin{proof}
If $\theta$  has   support  in a neighborhood of $\Gamma, V_\Gamma = \{(x,y): x\in \xR^d, \eta_*(x) < y <\eta_*(x) + \eps\},$ this follows from \eqref{green1}. Assume that $\theta$ vanishes in a neighborhood of $\Gamma$.  Let $\Omega_0 = \{(x,y):   \eta_*(x) + \frac{\eps}{2} < y <\eta(x)\}$  for $\eps >0$ small enough. Then $ \theta  \in   H ^1(\Omega_0)$  and $\theta\arrowvert_{\partial \Omega_0} =0$. Thus $\theta \in H^1_0(\Omega_0)$. Since $\Omega_0$ has a Lipschitz upper boundary there exists a sequence $\theta_n \in C_0^\infty(\Omega_0)$ which converges to $\theta$ in $H^1(\Omega_0) $ (see \cite{GRI} Corollary 1.5.1.6). Now by the equation we have
$$0= \langle \Delta_{x,y} \Phi, \alpha \theta_n \rangle =   \iint_\Omega \nabla_{x,y}\Phi(x,y) \cdot \nabla_{x,y}(\alpha(x)\theta_n(x,y)) \, dx\, dy.  $$
Moreover
$$
\la \iint_\Omega \nabla_{x,y}\Phi(x,y) \cdot\nabla_{x,y}\Big( \alpha(x)(\theta_n - \theta)(x,y)\Big) \, dx\, dy\ra 
\leq   C\Vert \Phi \Vert_{H^1_{ul}(\Omega)} \Vert \theta_n - \theta\Vert_{H^1(\Omega_0)}.
$$
Therefore, passing to the limit,  we obtain \eqref{green2} for such $\theta. $  
 \end{proof}

Part 1. Uniqueness: Let us denote by $\Phi$ the difference of two solutions in $H^1_{ul}(\Omega)$ of \eqref{eqPhi}.  Then $ \gamma_0 \Phi =0$.
 Now we take in \eqref{green2}   $\alpha(x) = e^{-\frac{\langle x \rangle}{A}}\zeta(\frac{\langle x \rangle}{B})$  where  $A,B$ are large constants to be chosen,  $\zeta \in C^\infty(\xR)$  $\zeta(t) = 1$ when $\vert t \vert \leq \mez,$   $ \supp \zeta \subset \{ t \in \xR: \vert t \vert \leq 1\}, 0 \leq \zeta \leq 1 $ and $\theta = \alpha_1(x) \Phi$ where $\alpha_1 \in C_0^\infty(\xR^d)$ is equal to one on the support of $\alpha$. Then $\theta \in H^1(\Omega)$ and $\gamma_0 \theta = 0$.  We can therefore use Lemma \ref{GREEN} and  we obtain
\begin{equation}\label{unicite}
\begin{aligned}
 I&:=  \iint_\Omega   \alpha(x) \vert \nabla_{x,y} \Phi(x,y)\vert^2  \, dx \,dy   \\
 &= -\frac{1}{A}\iint_\Omega (\nabla_x \langle x \rangle) \alpha(x) \Phi(x,y)\cdot \nabla_x \Phi(x,y)\, dx\, dy\\
  & \quad + \frac{1}{B}\iint_\Omega e^{-\frac{\langle x \rangle}{A}}\zeta'(\frac{\langle x \rangle}{B}) \Phi(x,y) ( \nabla_x \langle x \rangle)  \cdot \nabla_x \Phi(x,y)\, dx\, dy 
  \\
&=(1) +(2).
\end{aligned}
\end{equation}
By the Cauchy-Schwarz inequality we have
$$\vert (1) \vert \leq\frac{C_1}{A}  \Big(\iint_\Omega \alpha(x) \vert \nabla_x \Phi(x,y)\vert^2\, dx\, dy\Big)^\mez  \Big(\iint_\Omega \alpha(x) \vert \Phi(x,y)\vert^2\, dx\, dy\Big)^\mez.$$
Using Lemma \ref{poinc} we deduce that
$$ \vert (1) \vert \leq\frac{C_2}{ A} \Big(\iint_\Omega \alpha(x) \vert \nabla_x \Phi(x,y)\vert^2\, dx\, dy +  \iint_\Omega \alpha(x) \vert \partial_y\Phi(x,y)\vert^2\, dx\, dy\Big).$$ 
Taking $A$ large enough we see that the term $(1)$ can be absorbed by the left hand side of \eqref{unicite}. We then fix $A$. It follows that 
$$  I   \leq \frac{C_3}{B} \sum_{q\in \xZ^d} \sum_{k\in \xZ^d} \iint_\Omega  e^{-\frac{\langle x \rangle}{A}} \vert \zeta'(\frac{\langle x \rangle}{B})\vert 
\vert \chi_q(x) \Phi(x,y)\vert 
\big\vert  \nabla_x \big(\chi_k(x)\Phi(x,y)\big)\big\vert \, dx\, dy.$$
If $\vert k-q \vert \geq 2$ we have $\supp \chi_q \cap \supp \chi_k =\emptyset$. Therefore we have $\vert k-q \vert \leq 1$ (essentially $k=q$.) Moreover in the integral in the right hand side we have $\vert x-q\vert \leq 1$ and $1 \leq \frac{\langle x \rangle}{B} \leq 2$. If $B$ is large enough we have therefore $\frac{B}{3} \leq \vert q \vert \leq 3B$ and $\vert x \vert \geq \mez \vert q \vert$. It follows that 
$$
I  \leq \frac{C_4}{B}Ê\sum_{\frac{B}{3} \leq \vert q \vert \leq 3B} I_q
$$
where
$$
I_q=
e^{-\frac{ \langle q \rangle} {C_5A}} \Big(\iint_{\Omega} \vert \chi_q(x) \Phi(x,y)\vert^2\, dx\,dy\Big)^\mez 
\Big(\iint_{\Omega} \big\vert \nabla_x \big(\chi_q(x) \Phi(x,y)\big)\big\vert^2\, dx\,dy\Big)^\mez
$$
so using again the Poincar\'e inequality we obtain
\begin{align*}
I   &\leq \frac{C_6}{B}Ê\sum_{\frac{B}{3} \leq \vert q \vert \leq 3B}e^{-\frac{ \langle q \rangle} {C_5A}}    \iint_{\Omega} \vert \nabla_{x,y} \big(\chi_q(x) \Phi(x,y)\big)\vert^2\, dx\,dy \\
&\leq  \frac{C_7}{B}Ê\Big(\sum_{\frac{B}{3} \leq \vert q \vert \leq 3B}e^{-\frac{ \langle q \rangle} {C_5A}}\Big) \Vert \Phi \Vert^2_{H^1_{ul}(\Omega)}. 
\end{align*}
Since the cardinal of the set $\{q\in \xZ^d:  \frac{B}{3}\leq \vert q \vert \leq 3B\}$ is bounded by $C  B^d$ we obtain eventually
$$ \iint_\Omega   e^{-\frac{\langle x \rangle}{A}}\theta(\frac{\langle x \rangle}{B}) \vert  \nabla_{x,y} \Phi(x,y)\vert^2  \, dx \,dy   \leq C_8 B^{d-1} e^{-c \frac{ \langle B \rangle}{A}} \Vert \Phi \Vert^2_{H^1_{ul}(\Omega)}.$$
Letting  $B$ go to $+\infty$ and applying Fatou's Lemma we obtain 
$$
\iint_\Omega   e^{-\frac{\langle x \rangle}{A}}  \vert  \nabla_{x,y} \Phi(x,y)\vert^2  \, dx \,dy =0,
$$
which implies that $ \nabla_{x,y} \Phi(x,y) =0$ in $\Omega$ thus $\Phi =0$ since $\Phi\arrowvert_\Sigma = 0$.

Part 2: Existence. We first recall the situation when   $\psi \in H^\mez(\xR^d)$. In the following lemma, whose proof is given below, in section 4.2,  we construct a suitable extension of $\psi$ to $\Omega$.
\begin{lemm}\label{psisoul}
Let  $\psi \in H^\mez(\xR^d)$. One can find $\underline{ \psi } $ such that
 \begin{enumerate}
    \item $\underline{\psi} \in H^1(\Omega), \quad  
    \supp \underline{\psi} \subset \{(x,y): \eta(x) -h \leq y \leq \eta(x)\},$ 
 \item $\underline{\psi}_{\mid_{y= \eta (x) }} = \psi(x), $
 \item$ \| \underline{\psi} \|_{H^1( \Omega)} \leq  \mathcal{F}(\Vert \eta \Vert_{W^{1,\infty}(\xR^d)}) \| \psi\|_{H^{\mez} ( \xR^d)}$.  
     \end{enumerate}
 \end{lemm}
  Then (see \cite{ABZ3}  for more details) the problem
$$
\Delta_{x,y} u = - \Delta_{x,y} \underline{\psi} 
$$
has a unique solution $u\in H^{1,0}(\Omega)$. This solution, which is the variationnal one,  is characterized by 
\begin{equation}\label{eqvar}
\iint_{\Omega} \nabla_{x,y}u(x,y) \cdot \nabla_{x,y}\theta(x,y)dx dy = - \iint_{\Omega} \nabla_{x,y} \underline{\psi}(x,y) \cdot \nabla_{x,y}\theta(x,y)dx dy  
\end{equation}
for every $\theta \in H^{1,0}(\Omega)$. It satisfies
$$
\Vert \nabla_{x,y}u\Vert_{L^2(\Omega)} \leq C \Vert \psi \Vert_{H^\mez (\xR^d)}.
$$
Then $\Phi = u + \underline{\psi} $ solves the problem \eqref{eqPhi}. 

Let us consider now the case where  $\psi \in H_{ul}^\mez(\xR^d)$.
  If $q \in \xZ^d$ and   $\chi_q$ is defined in~\eqref{kiq} we set 
$$\psi_q = \chi_q \psi \in H^\mez(\xR^d).$$
By Lemma \ref{psisoul} one can find $  \underline{ \psi}_q \in H^1(\Omega)$ such that $\underline{ \psi}_q  \arrowvert_{ y=\eta(x)}  = \psi_q(x)$ and  
 \begin{align*} 
& (i) \quad \supp \underline{ \psi}_q \subset \{(x,y): \vert x-q \vert \leq 2, \eta(x) -h \leq y \leq \eta(x)\}\\
  & (ii) \quad \| \underline{ \psi}_q \|_{H^1( \Omega)} \leq \mathcal{F}(\Vert \eta \Vert_{W^{1,\infty}(\xR^d)}) \| \psi_q\|_{H^{\mez} ( \xR^d)}. 
   \end{align*}
To achieve $(i)$  we  multiply the function constructed in the lemma by $\widetilde{\chi}_q(x),$ where $\supp \widetilde{\chi}$ is contained in $\{x: \vert x \vert \leq 2\} $ and $ \widetilde{\chi} = 1$ on the support of $\chi$.

Let $u_q $ be the variational solution, described above,  of the equation $\Delta_{x,y} u_q = -\Delta_{x,y} \underline{\psi_q}$.   Our aim is to prove that the series $\sum_{q\in \xZ^d} {u}_q$ is   convergent in the space   $ {H}^{1,0}_{ul}({\Omega})$. This will be a consequence of the following lemma.
 \begin{lemm}\label{coro1}
 There exists $\delta >0$ and  $\mathcal{F}: \xR^+ \to \xR^+$   non decreasing  such that for all $q \in \xZ^d$ we have
 \begin{equation}\label{FN}
 \Vert e^{\delta \langle x-q\rangle} \nabla_{x,y}{u}_q \Vert_{L^2({\Omega})}   \leq \mathcal{F} (\Vert \eta \Vert_{W^{1,\infty}(\xR^d)})  \Vert  \psi_q \Vert_{H^\mez(\xR^d)}.
  \end{equation}
 \end{lemm}
Assuming that this lemma has been proved, one can write
\begin{equation}\label{kikuq}
\begin{aligned}
 \Vert \chi_k \nabla_{x,y}{u}_q \Vert_{L^2({\Omega})}  &= \Vert \chi_k e^{-\delta \langle x-q\rangle}e^{\delta \langle x-q\rangle} \nabla_{x,y}{u}_q \Vert_{L^2({\Omega})}\\
 &  \leq  Ce^{-\delta \langle k-q\rangle} \mathcal{F}  (\Vert \eta \Vert_{W^{1,\infty}}) \Vert  \psi_q \Vert_{H^\mez} \\
 &\leq C'e^{-\delta \langle k-q\rangle}   \mathcal{F}(\Vert \eta \Vert_{W^{1,\infty}})   \Vert \psi \Vert_{H^\mez_{ul}}.
 \end{aligned}
 \end{equation}
 Let us set ${S}^Q = \sum_{\vert q \vert \leq Q} {u}_q$. First of all $({S}^Q)$ converges to ${u}= \sum_{q\in \xZ^d} {u}_q$   in $\mathcal{D}'( {\Omega})$.  Indeed if $\varphi \in C_0^\infty({\Omega})$   there exists a finite set $A \subset \xZ^d$ such that $\varphi = \sum_{k\in A} \chi_k \varphi$. Then using Lemma \ref{poinc} and \eqref{kikuq} we can write
 \begin{equation*}
 \begin{aligned}
  \vert \langle {u}_q, \varphi \rangle \vert &\leq  \sum_{k\in A} \vert \langle \chi_k{u}_q, \varphi \rangle\vert \leq  C \sum_{k\in A} \Vert \chi_k \partial_y {u}_q \Vert_{L^2({\Omega})} \Vert \varphi \Vert_{L^2({\Omega})}  \\
  &  \leq {C}{e^{-\delta \langle q\rangle} }  \mathcal{F}(\Vert \eta \Vert_{W^{1,\infty}})   \Vert \psi \Vert_{H^\mez_{ul}} \Vert \varphi \Vert_{L^2({\Omega})}
     \end{aligned}
 \end{equation*}   
 for large $\vert q \vert$.

 On the other hand  \eqref{kikuq} shows that for fixed $k$ the series $\sum_{ q\in \xZ^d}  \chi_k  {u}_q $ is absolutely convergent in ${H}^{1,0}_{ul}({\Omega})$. Therefore $(\chi_k   {S}^Q)$ converges to $(\chi_k  {u} )$ in $ {H}^{1,0}_{ul}({\Omega})$ and we can write using~\eqref{kikuq},  
 \begin{equation*}
 \begin{aligned}
 \Vert \chi_k \nabla_{x,y} {u} \Vert_{L^2({\Omega})}&= \lim_{Q \to + \infty}\Vert \chi_k \nabla {S}^Q \Vert_{L^2({\Omega})}\\
 & \leq   \sum_{q\in \xZ^d}  e^{-\delta \langle k-q\rangle}  \mathcal{F} (\Vert \eta \Vert_{W^{1,\infty}})   \Vert \psi \Vert_{H^\mez_{ul}}.
 \end{aligned}
 \end{equation*}
 Therefore ${u} \in  {H}^{1,0}_{ul}({\Omega})$ and
   \begin{equation}\label{rec0}
 \Vert\nabla_{x,y} {u}  \Vert_{L_{ul}^2(\Omega)} \leq \mathcal{F}  (\Vert \eta \Vert_{W^{1,\infty} })  \Vert \psi \Vert_{H^\mez_{ul} }.
 \end{equation}

Finally  $ {\Phi} = {u}+ {\underline{\psi}}$ solves the problem \eqref{eqPhi} and we have 
$$\Vert \Phi \Vert_{H^1_{ul}(\Omega)} \leq \mathcal{F}(\Vert \eta \Vert_{W^{1,\infty} } ) \Vert \psi\Vert_{H_{ul}^\mez } $$
which completes the proof of Proposition \ref{existe:phi} assuming Lemma \ref{coro1}. 
 \end{proof}

\begin{proof}[Proof of Lemma \ref{coro1}]
We set
$$
w_\eps(x) = \frac{\langle x-q\rangle}{1+ \eps \langle x -q\rangle}\cdot
$$
Let $u_q$ be the variational solution in $H^{1,0} ( \Omega)$ of $\Delta u_q= -\Delta \underline{ \psi} _q$.
According to the variational formulation~\eqref{eqvar}, with $\theta= e^{\delta w_\eps(x)} u_q$, we have 

$$ \iint_{\Omega} \nabla_{x,y}{u_q} \cdot \nabla_{x,y} (e^{2\delta w_\eps(x)} u_q) \,dx\,dy= - \iint_{\Omega}\nabla_{x,y}{\underline{\psi}_q} \cdot \nabla_{x,y} (e^{2\delta w_\eps(x)} u_q) \,dx\,dy.$$
Therefore
\begin{equation}
\begin{aligned}
&\iint_{\Omega}e^{2\delta w_\eps(x)} \nabla_{x,y} {u_q} \cdot \nabla_{x,y}  u_q \, dx\,dy 
= - \iint_{\Omega} e^{2\delta w_\eps(x)} \nabla_{x,y}{\underline{\psi}_q} \cdot    \nabla_{x,y} u_q \,dx\,dy\\
&- 2\delta \iint _\Omega  e^{2\delta w_\eps(x)}  u_q \nabla_{x,y}{\underline{\psi}_q} \cdot \nabla_x   w_{\eps} \,dx\,dy - 2 \delta \iint _\Omega  e^{2\delta w_\eps(x)}  u_q \nabla_{x,y} {u_q} \cdot \nabla_x   w_{\eps}\, dx\,dy.\\
  \end{aligned}
 \end{equation}
Now $\nabla_x w_\eps $ is uniformly bounded in $L^\infty$ with respect to $\eps$ and $x$  and, on the support of $ {\underline{\psi}_q},$ we have $e^{ \delta w_\eps(x)} \leq e^{C\delta}$. Consequently, using the Cauchy-Schwarz inequality, the inequality \eqref{encore:poinc} with $\alpha = e^{ 2\delta w_\eps(x)}$ and taking $\delta$ small enough we obtain
\begin{equation} 
\iint_{\Omega}e^{2\delta w_\eps(x)} |\nabla_{x,y}{u_q}|^2 \,dx\, dy  
\leq C \| \underline{\psi}_q\|^2_{H^1( \Omega)}. 
\end{equation} 
 
We deduce when $\eps $ goes to $0$, using the Fatou Lemma, that
$$
\Vert e^{ \delta \langle x-q \rangle}  \nabla_{x,y}{u_q}\Vert_{L^2(\Omega)}    
\leq C \| \underline{\psi}_q\|_{H^1( \Omega)} \leq \mathcal{F} (\Vert \eta \Vert_{W^{1,\infty}(\xR^d)})\Vert \psi_q \Vert_{H^\mez(\xR^d)}.
$$
This completes the proof.
\end{proof}
\subsection{Straightening the boundary}
Before studying more precisely the properties of the Dirichlet-Neumann operator, we first straighten the  boundaries of
\begin{equation}\label{omegah} 
 \Omega_h = \{(x,y)\in \xR^{d+1}: \eta(x) -h < y< \eta(x)\}.
\end{equation}   
 \begin{lemm}\label{diffeo}
There is an absolute constant $C>0$ such that if we take $\delta>0$ so small that 
$$\delta  \Vert \eta \Vert_{ W^{1,\infty}(\xR^d) }   \leq \frac{h}{2C}$$
then the map $(x,z) \mapsto (x, \rho(x,z))$ where
\begin{equation}\label{rho} 
\rho(x,z) = (1+z)e^{\delta z \langle D_x \rangle}\eta(x) - z \big\{  e^{-\delta (1+z) \langle D_x \rangle}\eta(x) -h\big \}
\end{equation}
is a diffeomorphism from $\widetilde{\Omega} = \{(x,z): x\in \xR^d, -1<z<0\}$ to $\Omega_h$.
 \end{lemm}
\begin{proof}
First of all  we have $\rho(x,0)= \eta(x), \rho(x,-1) = \eta(x) -h$. Moreover we have $\partial_z \rho \geq \frac{h}{2}$. Indeed we have
\begin{align*}
\partial_z \rho  &= \eta    + (e^{\delta z \langle D_x \rangle}\eta  - \eta)  +(1+z)\delta e^{\delta z \langle D_x \rangle}\langle D_x \rangle \eta \\
&-\{\eta -h + (e^{-\delta (1+z) \langle D_x \rangle}\eta  - \eta)\} + z \delta e^{-\delta (1+z) \langle D_x \rangle} \langle D_x \rangle \eta.
\end{align*}
  Now  for any $\lambda<0$ 
the symbol $a(\xi) = e^{\lambda \langle \xi \rangle}$ satisfies 
the estimate $\vert \partial_\xi^\alpha a(\xi) \vert \leq C_\alpha \langle \xi \rangle^{-\vert \alpha \vert}$ 
where $C_\alpha$ is independent on $\lambda$. Therefore its 
Fourier transform in an $L^1(\xR^d)$ function  whose norm is uniformly bounded. 
This implies that $\Vert a(D)f \Vert_{L^\infty(\xR^d)} \leq K \Vert  f \Vert_{L^\infty(\xR^d)} $ 
with $K$ independent of $\lambda$. Since 
$ e^{\delta \lambda \langle D_x \rangle}\eta  - \eta
= \delta \lambda \int_0^1e^{\delta t \lambda \langle D_x \rangle}\langle D_x \rangle \eta \, dt$ 
we can write
\begin{multline*}\Vert e^{\delta z \langle D_x \rangle}\eta  - \eta \Vert_{L^\infty(\xR^d)} 
+ \delta \Vert e^{\delta z \langle D_x \rangle}
\langle D_x \rangle \eta\Vert_ {L^\infty(\xR^d)} + \Vert e^{-\delta (1+z) \langle D_x \rangle}\eta  - \eta \Vert_{L^\infty(\xR^d)}\\ 
+ \delta \Vert e^{-\delta (1+z)\langle D_x \rangle}
\langle D_x \rangle \eta\Vert_ {L^\infty(\xR^d)}
\leq C \delta \Vert \eta \Vert_{W^{1,\infty}(\xR^d)} \leq \frac{h}{2}.
\end{multline*}
This completes the proof.
\end{proof}
From the above computation we deduce that 
\begin{equation}\label{rho>}
  \partial_z \rho(x,z)\geq  \frac{h}{2} \quad  \text{and } \quad 
\Vert \nabla_{x,z} \rho \Vert_{L^\infty(\xR^d)}  \leq C  \Vert \eta \Vert_{W^{1,\infty}(\xR^d)}. 
    \end{equation}
 
We shall denote by $\kappa$ the inverse of $\rho$,  
$$
\rho(x,z) = y \Longleftrightarrow z = \kappa (x,y).
$$

If we set
$$ \widetilde{f}(x,z) = f(x,\rho(x,z))$$
 we have
 \begin{equation}\label{lambda}
 \left\{
 \begin{aligned}
  \frac{\partial f}{\partial y}(x,\rho(x,z)) &= \frac{1}{\partial_z \rho}\partial_z \widetilde{f}(x,z):=\Lambda_1\widetilde{f}(x,z) \\
  \nabla_x f(x,\rho(x,z)) &= \big(\nabla_x \tilde{f} - \frac{\nabla_x \rho}{\partial_z \rho}\partial_z \widetilde{f}\big)(x,z) := \Lambda_2 \widetilde{f}(x,z).
  \end{aligned}
  \right.
\end{equation}
We introduce the space
$$
\mathcal{H}^1_{ul}(\widetilde{\Omega}) = \big\{\widetilde{u} \in L^2_{ul}(\tilde{\Omega}): \Lambda_j \widetilde{u} \in  L^2_{ul}(\widetilde{\Omega}),\, j=1,2 \big\},
$$
endowed with the norm
$$\Vert \widetilde{u}\Vert_{\mathcal{H}^1_{ul}(\widetilde{\Omega})} = \sup_{q\in \xZ^d} \Vert \chi_q\widetilde{u}\Vert_{L^2 (\tilde{\Omega})}+  \sum_{j=1}^2 \sup_{q\in \xZ^d} \Vert \chi_q \Lambda_j\widetilde{u}\Vert_{L^2 (\tilde{\Omega})}.$$
 
Then according to Lemma \ref{lions} we see that the elements of $\mathcal{H}^1_{ul}(\widetilde{\Omega})$ 
have a trace  on $z=0$ 
belonging  to the space $\mathcal{H}^\mez_{ul}(\xR^d)$. 
Then we introduce the subspace $\mathcal{H}^{1,0}_{ul}(\widetilde{\Omega}) \subset \mathcal{H}^{1}_{ul}(\widetilde{\Omega})$ 
defined as follows 
$$
\mathcal{H}^{1,0}_{ul}(\widetilde{\Omega})  = \{ \widetilde{u} \in \mathcal{H}^{1}_{ul}(\widetilde{\Omega}) : \widetilde{u}\arrowvert_{z=0} =0\}.
$$
It follows that we have
$$
u\in H^{1}_{ul}(\Omega) \Leftrightarrow\widetilde{u} 
\in \mathcal{H}^{1}_{ul}(\widetilde{\Omega}), \quad  u\in H^{1,0}_{ul}(\Omega) \Leftrightarrow\widetilde{u} \in \mathcal{H}^{1,0}_{ul}(\widetilde{\Omega}).
$$

\subsection{Definition of the Dirichlet-Neumann operator}
    We can now define the Dirichlet-Neumann operator. Formally we set for $\psi \in H_{ul}^\mez(\xR^d)$
\begin{equation}\label{G=}
\begin{aligned}
 G(\eta) \psi(x) &= (1+\vert \nabla_x \eta\vert^2)^\mez \frac{\partial \Phi}{\partial n}\arrowvert_\Sigma = \big(\frac{\partial \Phi}{\partial y} - \nabla_x \eta \cdot \nabla_x \Phi\big)\arrowvert_\Sigma\\
   &= \big( \Lambda_1\widetilde{\Phi}- \nabla_x \eta \cdot\Lambda_2 \widetilde{\Phi}\big)\arrowvert_{z=0} = \big( \Lambda_1\widetilde{\Phi}- \nabla_x \rho\cdot\Lambda_2 \widetilde{\Phi}\big)\arrowvert_{z=0} 
 \end{aligned}
 \end{equation}
 where $\Phi$ is the solution described in Lemma \ref{existe:phi} and $\Lambda_j$ is as defined by \eqref{lambda}.
Our aim is to prove the following theorem.
\begin{theo}\label{def:DN}
 Let $d\geq 1 $  and $\eta\in W^{1,\infty}(\xR^d)$. Then the Dirichlet-Neumann operator is well defined on $H^\mez_{ul}(\xR^d)$ by \eqref{G=}. Moreover there exists a non decreasing function $\mathcal{F}: \xR^+ \to \xR^+$   such that for all  $\eta\in W^{1,\infty}(\xR^d)$
$$
\Vert G(\eta) \psi \Vert_{H^{-\mez}_{ul}(\xR^d)} \leq 
\mathcal{F}\bigl(\Vert \eta \Vert_{W^{1,\infty}(\xR^d)}\bigr)
\Vert \psi \Vert_{H^{\mez}_{ul}(\xR^d)}.
$$
\end{theo}
\begin{proof}
Set $U =   \Lambda_1\widetilde{\Phi}- \nabla_x \rho\cdot\Lambda_2\widetilde{\Phi}$ and $J=(-1,0)$. We shall prove that for all $q\in \xZ^d$,
\begin{equation}\label{tag:i} 
 \Vert \chi_qU\Vert_{L^2(J, L^2 )} \leq \mathcal{F}\bigl(\Vert \eta \Vert_{W^{1,\infty} }\bigr) \Vert \psi \Vert_{H^{\mez}_{ul}} , 
  \end{equation}
  \begin{equation}\label{tag:ii}
  \Vert \chi_q \partial_zU\Vert_{L^2(J, H^{-1} )} \leq \mathcal{F}\bigl(\Vert \eta \Vert_{W^{1,\infty}}\bigr) \Vert \psi \Vert_{H^{\mez}_{ul}},
  \end{equation}
where $\mathcal{F}: \xR^+ \to \xR^+$  is independent of $q$ and $\eta$.  Then Theorem \ref{def:DN} will follow from \eqref{tag:i}, \eqref{tag:ii} 
and Lemma \ref{lions}. Recall that $\widetilde{\Phi} = \widetilde{u} + \widetilde{\underline{\psi}}$. Now the estimate \eqref{tag:i} follows 
from \eqref{rec0}, \eqref{lambda} and Corollary \ref{e^z} with $\sigma = 0$ and $m=1$.
To prove \eqref{tag:ii} we observe that
\begin{equation}\label{dzU}
\partial_z U = -\nabla_x \big((\partial_z \rho)\Lambda_2\widetilde{\Phi}\big).
\end{equation}
Indeed we have
\begin{equation*}
\begin{aligned}
\partial_z U &= \partial_z \Lambda_1\widetilde{\Phi}- \nabla_x \partial_z \rho \cdot \Lambda_2\widetilde{\Phi}- \nabla_x \rho  \cdot \partial_z \Lambda_2\widetilde{\Phi}\\
&= (\partial_z \rho) \Lambda_1^2\widetilde{\Phi}-  \nabla_x \partial_z \rho \cdot \Lambda_2\widetilde{\Phi}+ (\partial_z \rho) (\Lambda_2 - \nabla_x) \Lambda_2 \widetilde{\Phi}\\
&= (\partial_z \rho)( \Lambda_1^2 +  \Lambda_2^2)\widetilde{\Phi}-\nabla_x \big((\partial_z \rho) \Lambda_2\widetilde{\Phi}\big) 
\end{aligned}
\end{equation*}
so \eqref{dzU} follows from the fact that  $ (\Lambda_1^2 +  \Lambda_2^2)\widetilde{\Phi}=0$. Then \eqref{tag:ii} follows from the 
estimates used to bound \eqref{tag:i} and the Poincar\' e inequality \eqref{encore:poinc}. The proof of Theorem \ref{def:DN} is complete.
 \end{proof}
 
We state now a consequence of the previous estimates which will be used in the sequel. 
Notice first that the equation $(\Lambda_1^2 +  \Lambda_2^2)\widetilde{\Phi}=0$ 
is equivalent to the equation
\begin{equation}\label{equ:modifie}
(\partial_z^2 + \alpha \Delta_x + \beta\cdot \nabla_x \partial_z - \gamma \partial_z) \widetilde{\Phi} = 0,
\end{equation}
where 
\begin{equation}\label{alpha}
\alpha\defn \frac{(\partial_z\rho)^2}{1+|\partialx  \rho |^2},\quad 
\beta\defn  -2 \frac{\partial_z \rho \nabla_x \rho}{1+|\nabla_x  \rho |^2} ,\quad 
\gamma \defn \frac{1}{\partial_z\rho}\bigl(  \partial_z^2 \rho 
+\alpha\Delta_x \rho + \beta \cdot \nabla_x \partial_z \rho\bigr).
\end{equation}
 
 \begin{coro}\label{rec00}
 Let $s_0 >1+Ê\frac{d}{2}$ and $J = (-1,0)$. There exists a non decreasing function $\mathcal{F}: \xR^+ \to \xR^+$ such that
 \begin{equation}
 \Vert \nabla_{x,z}\widetilde{\Phi}\Vert_{X^{-\mez}_{ul}(J)} \leq \mathcal{F}(\Vert \eta \Vert_{H^{s_0+\mez}_{ul}(\xR^d)}) \Vert \psi \Vert_{H^\mez_{ul}(\xR^d)}.
 \end{equation}
   \end{coro}
 \begin{proof}
First of all the estimate
$$ \Vert \nabla_{x,z} \widetilde{\underline{\psi}} \Vert_{X^{-\mez}_{ul}(J)} \leq  C \Vert \psi \Vert_{H^\mez_{ul} } $$
follows from Corollary \ref{e^z} with $\delta =1, m=0,1$ and $\sigma = 0$.

On the other hand we notice that $\partial_z = (\partial_z \rho) \Lambda_1$ and $\nabla_x = \Lambda_2 + (\nabla_x \rho) \Lambda_1$. Let $\widetilde{\chi} \in C_0^\infty(\xR^d), \widetilde{\chi}  = 1$ on the support of $\chi$.  Since $s>1 + \frac{d}{2},$  using Corollary  \ref{e^z}  with $\sigma = s$ we can write
\begin{equation*}
\begin{aligned}
  \Vert \widetilde{\chi}_k \nabla_{x,z} \rho \Vert_{L^\infty(J\times \xR^d)}  
 &\leq C \Vert \widetilde{\chi}_k\nabla_{x,z} \rho \Vert_{L^\infty(J,H^{s-\mez} )}\leq C'\Vert \nabla_{x,z} \rho \Vert_{L^\infty(J,H^{s-\mez} )_{ul}}\\
  \quad &\leq C''(1+ \Vert \eta \Vert_{H^{s+\mez}_{ul} }).
  \end{aligned}
\end{equation*}
It follows from~\eqref{rec0}  that
\begin{equation}\label{-mez1}
  \Vert  \nabla_{x,z} \widetilde{u}  \Vert_{L^2(J, L^2 )_{ul}} \leq \mathcal{F}  (\Vert \eta \Vert_{H^{s+\mez}_{ul} }) \Vert \psi \Vert_{H^\mez_{ul}}.
  \end{equation}
  Now using Lemma~\ref{lions} we have
 $$
\Vert \chi_k \nabla_x \widetilde{u}\Vert_{L^\infty(J, H^{-\mez} )} 
\leq C(\Vert \chi_k \nabla_x \widetilde{u}\Vert_{L^2(J, L^2 )}+\Vert \chi_k \partial_z \nabla_x \widetilde{u}\Vert_{L^2(J, H^{-1} )}).
$$
The first term in the right hand side is estimated using~\eqref{-mez1}. For the second term  using~\eqref{-mez1} we have
\begin{equation}\label{-mez2}
\begin{aligned}
 \Vert \chi_k \partial_z \nabla_x \widetilde{u}\Vert_{L^2(J, H^{-1})}  
  &\leq \Vert (\nabla_{x } \chi_k) \partial_z \widetilde{u}\Vert_{L^2(J, H^{-1} )} + \Vert \chi_k \partial_z  \widetilde{u}\Vert_{L^2(J, L^2 )}\\
&  \leq C  \Vert  \partial_z  \widetilde{u}\Vert_{L^2(J, L^2 )_{ul}} \leq \mathcal{F}  (\Vert \eta \Vert_{H^{s+\mez}_{ul} }) \Vert \psi \Vert_{H^\mez_{ul} }.
\end{aligned}
\end{equation}
Therefore
\begin{equation}\label{-mez3}
\Vert  \nabla_x \widetilde{u}\Vert_{L^\infty(J, H^{-\mez} )_{ul}} \leq  \mathcal{F}  (\Vert \eta \Vert_{H^{s+\mez}_{ul} }) \Vert \psi \Vert_{H^\mez_{ul} }. 
\end{equation}
Eventually 
\begin{equation}\label{-mez4}
\Vert \chi_k \partial_z \widetilde{u}\Vert_{L^\infty(J, H^{-\mez} )} \leq C(\Vert \chi_k \partial_z \widetilde{u}\Vert_{L^2(J, L^2 )} + \Vert \chi_k \partial_z^2  \widetilde{u}\Vert_{L^2(J, H^{-1} )}).
\end{equation}
The first term in the right hand side is estimated using~\eqref{-mez1}. For the second term  using~\eqref{alpha} we have
\begin{equation*}\left\{
\begin{aligned}
 &\Vert \chi_k \partial_z^2  \widetilde{u}\Vert_{L^2(J, H^{-1} )}  \leq A_1 +A_2 +A_3\\
 &A_1  = \Vert \chi_k \alpha \Delta \widetilde{u}\Vert_{L^2(J, H^{-1} )}\\
 &A_2  =  \Vert \chi_k \beta \partial_z \nabla_x  \widetilde{u}\Vert_{L^2(J, H^{-1} )}\\
 &A_3  =  \Vert \chi_k \gamma\partial_z   \widetilde{u}\Vert_{L^2(J, H^{-1} )}.
\end{aligned}
\right.
\end{equation*}
Now using~\eqref{Houl},~\eqref{-mez1} and~\eqref{-mez2} we obtain
   \begin{equation*}
  \begin{aligned}
  A_1 &\leq \Vert \alpha \Vert_{L^\infty(J,H^{s-\mez} )_{ul}} \Vert \Delta \widetilde{u}\Vert_{L^2(J, H^{-1}  )_{ul}} 
  \leq  \mathcal{F}  (\Vert \eta \Vert_{H^{s+\mez}_{ul}} )\Vert \psi \Vert_{H^\mez_{ul}},\\
  A_2 &\leq \Vert \beta \Vert_{L^\infty(J,H^{s-\mez} )_{ul}} \Vert \partial_z \nabla_x  \widetilde{u}\Vert_{L^2(J, H^{-1}  )_{ul}}
   \leq  \mathcal{F}  (\Vert \eta \Vert_{H^{s+\mez}_{ul} })\Vert \psi \Vert_{H^\mez_{ul} },\\
  A_3 &\leq \Vert \gamma \Vert_{L^\infty(J,H^{s-\frac{3}{2}}) _{ul}}\Vert \partial_z \widetilde{u} \Vert_{L^2(J,L^2 )_{ul}} 
    \leq  \mathcal{F} (\Vert \eta \Vert_{H^{s+\mez} _{ul} } \Vert \psi \Vert_{H^\mez_{ul} }.
  \end{aligned}
\end{equation*}
Therefore using~\eqref{-mez4} we obtain
$$\Vert \partial_z \widetilde{u} \Vert_{L^\infty(J,H^{-\mez})_{ul}  } 
\leq \mathcal{F} \big(\Vert \eta \Vert_{H^{s+\mez} _{ul} }\big)\Vert \psi \Vert_{H^\mez_{ul} } $$
which completes the proof of Corollary~\ref{rec00}.
\end{proof}

\subsection{Higher estimates for the Dirichlet--Neumann operator}

In this section we prove the following results.
\begin{theo}\label{princ}
Let $d \geq 1$ and  $s_0>1 + \frac{d}{2}$. 

Case 1. There exists  $\mathcal{F}:\xR^+ \to \xR^+$ non decreasing such that for $-\mez \leq \sigma \leq s_0-1,$ every   $\eta  \in H^{s_0+\mez}_{ul}(\xR^d) $ satisfying~\eqref{condition} and every $\psi \in H^{\sigma+1}_{ul}(\xR^d)$ we have
$$
\Vert G(\eta) \psi \Vert_{H^{\sigma}_{ul} } \leq 
\mathcal{F}\bigl(\Vert  \eta   \Vert_{  H^{s_0+ \mez}_{ul}} \bigr)
  \Vert \psi \Vert_{H^{\sigma+1}_{ul}} .
$$
Case 2.  For every $s \geq s_0,$    there exists  $\mathcal{F}:\xR^+ \to \xR^+$ non decreasing such that for every $\eta  \in H^{s +\mez}_{ul}(\xR^d) $ satisfying~\eqref{condition}  every $s_0-1 \leq \sigma \leq s-\mez$ and every $\psi \in H^{\sigma+1}_{ul}(\xR^d)$ we have
 
$$
\Vert G(\eta) \psi \Vert_{H^{\sigma}_{ul} }  
  \leq \mathcal{F}\bigl(\Vert  (\eta,\psi)  \Vert_{H^{s_0+ \mez}_{ul} \times H^{s_0 }_{ul}} \bigr)
\big\{ \Vert \eta \Vert_{H^{s+\mez}_{ul} } + \Vert \psi\ \Vert_{H^{\sigma+1}_{ul}}+ 1\big\}.
 $$
    \end{theo}
 We  set 
 \begin{equation}\label{reste}
 R(\eta)\psi\ := G(\eta)\psi\ -T_\lambda \psi\, 
 \end{equation}
where $$
\lambda =  \big((1+ \vert \nabla_x \eta\vert^2 ) \vert \xi \vert^2 -(\nabla_x \eta \cdot \xi)^2\big)^\mez.
$$
\begin{theo}\label{paralinDN}
Let $d \geq 1$ and $s_0>1 + \frac{d}{2}$. 

Case 1. There exists  $\mathcal{F}:\xR^+ \to \xR^+$ non decreasing such that for $0\leq t \leq s_0-\mez, $  $\eta  \in H^{s_0+\mez}_{ul}(\xR^d) $ satisfying~\eqref{condition} we have 

 $$\Vert R(\eta)\psi\Vert_{H^{t }_{ul} } \leq \mathcal{F}\bigl(\Vert  \eta  \Vert_{  H^{s_0+ \mez}_{ul}} \bigr)
  \Vert \psi\ \Vert_{H^{t+\mez}_{ul}} 
$$
for every $\psi \in H^{t+\mez}_{ul}(\xR^d)$.

  Case 2.  For all $s \geq s_0 $   there exists  $\mathcal{F}:\xR^+ \to \xR^+$ non decreasing such that for every $\eta  \in H^{s +\mez}_{ul}(\xR^d) $ satisfying~\eqref{condition},  every $s_0-\mez \leq t \leq s-\mez$ and every $\psi \in H^{t+\mez}_{ul}(\xR^d)$ we have
   $$\Vert R(\eta)\psi\Vert_{H^{t }_{ul} } \leq  \mathcal{F}\bigl(\Vert  (\eta,\psi)  \Vert_{H^{s_0+ \mez}_{ul} \times H^{s_0 }_{ul}} \bigr)
\big\{ \Vert \eta \Vert_{H^{s+\mez}_{ul} } + \Vert \psi\ \Vert_{H^{t+\mez}_{ul}}+ 1\big\}.
$$

\end{theo}
 The main step in the proof of the above theorems  is the following elliptic regularity result.
 
\begin{theo}\label{regell}
Let $d\geq1 , J=(-1,0), s_0>1+\frac{d}{2}$. Let $\tilde{v}$  be a solution of the problem
\begin{equation}\label{elliptique}
\left\{
\begin{aligned}
&(\partial_z^2 + \alpha \Delta_x + \beta\cdot \nabla_x \partial_z - \gamma \partial_z) \widetilde{v} 
= F  \quad \text{in } \, \xR^d \times  J, \\
&\widetilde{v} \arrowvert_{z=0} = \psi.  
\end{aligned}
\right.
\end{equation}
Case 1. For $-\mez \leq \sigma \leq s_0-1 $ let  $ \eta  \in H^{s_0 + \mez}_{ul}(\xR^d)$ 
satisfying~\eqref{condition},  $\psi \in H^{\sigma+1}_{ul}(\xR^d), F \in Y_{ul}^{\sigma}(J)$  and 
\begin{equation}\label{-mez-i}
\Vert \nabla_{x,z}\widetilde{v}\Vert_{X_{ul}^{-\mez}(J)}  <+\infty.
\end{equation}
Then for every $z_0\in ]-1,0[$ one has
$ \nabla_{x,z}\widetilde{v} \in X^\sigma_{ul}(z_0,0)  $ 
and one can find $\mathcal{F} :\xR^+ \to \xR^+ $ non decreasing, 
depending only on $(s_0, d)$   such that \\
$$
\Vert \nabla_{x,z}\widetilde{v} \Vert_{X^\sigma_{ul}(z_0,0)} 
\leq \mathcal{F} \bigl(\Vert  \eta  \Vert_{H^{s_0+ \mez}_{ul}}\bigr) 
\Big\{ \Vert \psi \Vert_{H^{\sigma+1}_{ul} } +   \Vert F \Vert_{Y_{ul}^{\sigma}(J)}  
+  \Vert \nabla_{x,z}\tilde{v}\Vert_{X_{ul}^{-\mez}(J)} \Big\}.
$$
Case 2.  For   $s\geq s_0,$ and $s_0-1 \leq \sigma \leq s-\mez $ let $ \eta  \in H^{s  + \mez}_{ul}(\xR^d)$ satisfying~\eqref{condition},  $\psi \in H^{\sigma+1}_{ul}(\xR^d), F \in Y_{ul}^{\sigma}(J)$ and 
\begin{equation}\label{-mez}
\Vert \nabla_{x,z}\widetilde{v}\Vert_{X_{ul}^{s_0 -1}(J)}  <+\infty.
\end{equation}
 Then for every $z_0\in ]-1,0[$ one has
   $ \nabla_{x,z}\widetilde{v} \in X^\sigma_{ul}(z_0,0)  $ 
   and one can find $\mathcal{F} :\xR^+ \to \xR^+ $   non decreasing,  depending only on $(s_0, s, d)$   such that
     \begin{multline*}
\Vert \nabla_{x,z}\widetilde{v} \Vert_{X^\sigma_{ul}(z_0,0)} \\
\leq \mathcal{F} \bigl(\Vert  (\eta,\psi)  \Vert_{H^{s_0+ \mez}_{ul}  \times H^{s_0 }_{ul}}\bigr) \Big\{ \Vert \psi \Vert_{H^{\sigma+1}_{ul} } +   \Vert F \Vert_{Y_{ul}^{\sigma}(J)} +(\Vert \eta \Vert_{H^{s+\mez}_{ul} } +1) \Vert \nabla_{x,z}\tilde{v}\Vert_{X_{ul}^{s_0-1}(J)}\Big\}.
\end{multline*}
   \end{theo}
 \begin{coro}\label{coroetape1}
  Let $s_0>1+ \frac{d}{2}$. Let  $ {\Phi}$  be defined in Proposition \ref{existe:phi}. 
        
Case 1. For   $-\mez \leq\sigma \leq  s_0-1$ assume that   $ \eta  \in H^{s_0 + \mez}_{ul}(\xR^d)$ satisfying~\eqref{condition} and $\psi \in H^{\sigma+1}_{ul}(\xR^d)$.   Then  there exists  $\mathcal{F}:\xR^+ \to \xR^+$    non decreasing  depending only on $(s_0,d)$ such that 
   $$\Vert  \widetilde{\Phi}  \Vert_{X^{\sigma+1}_{ul}(z_0,0)}+  \Vert \nabla_{x,z}\widetilde{\Phi} \Vert_{X^\sigma_{ul}(z_0,0)}  \leq \mathcal{F}\bigl(\Vert   \eta  \Vert_{H^{s_0+ \mez}_{ul}}\bigr)     \Vert \psi \Vert_{H^{\sigma+1}_{ul}}.$$  

Case 2. For $s\geq s_0$, $ s_0-1 \leq \sigma \leq  s-\mez$ assume that $ \widetilde{\eta} \in H^{s  + \mez}_{ul}(\xR^d),$ $ \eta  \in H^{s  + \mez}_{ul}(\xR^d)$ satisfying~\eqref{condition} and $\psi \in H^{\sigma+1}_{ul}(\xR^d)$.   Then  there exists  $\mathcal{F}:\xR^+ \to \xR^+$    non decreasing  depending only on $(s_0, s, d)$ such that 
     \begin{align*} 
   \Vert  \widetilde{\Phi}  \Vert_{X^{\sigma+1}_{ul}(z_0,0)}+& \Vert \nabla_{x,z}\widetilde{\Phi} \Vert_{X^\sigma_{ul}(z_0,0)}\\
    &\leq \mathcal{F}\bigl(\Vert  (\eta,\psi)  \Vert_{H^{s_0+ \mez}_{ul}  \times H^{s_0 }_{ul}}\bigr) \big\{  \Vert \eta \Vert_{H^{s+\mez}_{ul} } +  \Vert \psi \Vert_{H^{\sigma+1}_{ul} } +1 \big\}.
   \end{align*} 
\end{coro}
\begin{proof}
Indeed $\widetilde{\Phi}$ satisfies~\eqref{elliptique} with $F=0$ and it is proved in Corollary~\ref{rec00}  that 
\begin{equation*}
\Vert \nabla_{x,z}\widetilde{\Phi}\Vert_{X^{-\mez}_{ul}(z_0,0)} \leq \mathcal{F}(\Vert \eta \Vert_{ H^{s_0+ \mez}_{ul}}) \Vert \psi \Vert_{H^\mez_{ul} }<+\infty.
\end{equation*}
Moreover the estimate of $\widetilde{\Phi}$ in $X^\sigma_{ul}(z_0,0)$ is obtained by the Poincar\' e inequality from the estimate of $\partial_z\widetilde{\Phi}$.
\end{proof}

\begin{proof}[Proof of Theorem~\ref{princ} given Corollary~\ref{coroetape1}]

  Let us set
\begin{equation}
H= \Lambda_1\widetilde{\Phi}- \nabla_x \rho\cdot\Lambda_2 \widetilde{\Phi}.
\end{equation}
By~\eqref{G=} we have $H\arrowvert_{z=0} = G(\eta)\psi $ and by \eqref{dzU}
  \begin{equation}\label{dzU1}
\partial_z H = -\nabla_x \big((\partial_z \rho)\Lambda_2\widetilde{\Phi}\big) =- \nabla_x \big( (\partial_z \rho)\nabla_x - (\nabla_x \rho) \partial_z\big)\widetilde{\Phi}.
\end{equation}
   Using  Lemma~\ref{lions} with $f= \chi_q H,$ ~\eqref{dzU1} we deduce that
 \begin{equation*}
        \Vert \chi_q G(\eta)\psi\Vert_{H^{\sigma}}   \leq  
        C \Bigl [\Vert \chi_q H \Vert_{L^2(J,H^{\sigma +\mez})} + \Vert \chi_q \partial_z H \Vert_{L^2(J,H^{\sigma -\mez})} \Bigl]. 
        \end{equation*}
Moreover by \eqref{dzU1} we have
$$
\Vert \chi_q \partial_z H \Vert_{L^2(J,H^{\sigma +\mez})} 
\leq C' \Bigl[\Vert  (\partial_z \rho)\nabla_x \widetilde{\Phi} \Vert_{ L^2(J,H^{\sigma+\mez})_{ul}}
+\Vert  (\nabla_x \rho) \partial_z \widetilde{\Phi}\Vert_{L^2(J,H^{\sigma +\mez})_{ul}}\Bigr].
$$

Case 1. If $-\mez \leq \sigma \leq s_0-1$ we use the estimate
$$
\Vert  fg \Vert_{ L^2(J,H^{\sigma +\mez}))_{ul}}  \leq C \Vert f \Vert_{L^\infty(J, H^{s_0-\mez})_{ul}}  \Vert   g \Vert_{ L^2(J,H^{\sigma +\mez})_{ul}}
$$
which follows from Proposition \ref{Ho} with 
$\sigma_0 = \sigma+\mez, \sigma_1=s_0-\mez, \sigma_2= \sigma+\mez, $ 
the estimates on $\rho$ and  Corollary \ref{coroetape1}.

Case 2. If $s_0-1 \leq \sigma \leq s-\mez$  we use the inequality
$$
\Vert fg\Vert_{L^2(J, H^{\sigma+\mez})_{ul}} \leq \Vert f \Vert_{L^\infty(J, H^{s_0-1})_{ul}}\Vert g\Vert_{X_{ul}^\sigma(J)} + \Vert g \Vert_{L^\infty(J, H^{s_0-1})_{ul}}\Vert f \Vert_{X_{ul}^\sigma(J)}$$
  the estimates on $\rho$ and again Corollary \ref{coroetape1} to obtain Theorem \ref{princ}. 
\end{proof}
Theorem~\ref{regell} will be a consequence of the following two results.
\begin{prop}\label{etape0}
Let $s_0>1+ \frac{d}{2}$. There exists   $\mathcal{F}:\xR^+ \to \xR^+$ non   decreasing such that for   $-1<z_0<z_1<0$,   $-\mez \leq\sigma \leq  s_0-1$  and     $k\in \xZ^d$ we have 
\begin{equation}
   \tag {$\mathcal{H}_\sigma$}  
\Vert \nabla_{x,z}\widetilde{v}_k\Vert_{X^\sigma(z_1,0)} 
\leq\mathcal{F}\bigl(\Vert  \eta  \Vert_{  H^{s_0+ \mez}_{ul} }\bigr) \Big\{  \Vert \psi \Vert_{H^{\sigma+1}_{ul}} +  \Vert F\Vert_{Y_{ul}^{\sigma }(J)} + \Vert \nabla_{x,z}\widetilde{v}_k   \Vert_{X ^{-\mez}(z_0,0)} \Big\},
\end{equation}
where $\widetilde{v}_k= \chi_k\widetilde{v}$.
\end{prop}
\begin{prop}\label{etape1}
Let $s_0>1+ \frac{d}{2}, $  and $s \geq s_0. $  Then there exists   $\mathcal{F}:\xR^+ \to \xR^+$ non decreasing such that for     $-1<z_0<z_1<0 $,   $s_0-1 \leq \sigma \leq  s-\mez$  and     $k\in \xZ^d$ we have 
\begin{multline}
     \tag {$\mathcal{K}_\sigma$} 
\Vert \nabla_{x,z}\widetilde{v}_k \Vert_{X^\sigma(z_1,0)}  
\leq\mathcal{F}\bigl(\Vert  \eta  \Vert_{  H^{s_0+ \mez}_{ul} }\bigr) \Big\{ \Vert \psi \Vert_{H^{\sigma+1}_{ul}} + \Vert F\Vert_{Y_{ul}^{\sigma }(J)} \\+ (1+\Vert \eta \Vert_{H^{s+\mez}_{ul} }) \Vert \nabla_{x,z}\widetilde{v}_k\Vert_{X_{ul}^{s_0-1}(z_0,0)} \Big\},
\end{multline}
where $\widetilde{v}_k= \chi_k\widetilde{v}$.
\end{prop}
  We shall prove these two results  by induction on $\sigma$ and by the same method. However we have to distinguish them since we want the right hand side of these estimates to be linear with respect to the higher norms of $(\psi,\eta)$. Since   $(\mathcal{H}_{-\mez})$ and $(\mathcal{K}_{s_0-1})$ are trivially satisfied if $\mathcal{F} \geq 1 $  these propositions will   be a consequence of the following one.   
 \begin{prop}\label{rec1}
 
   Case 1. \quad Let $s_0>1 + \frac{d}{2}$.    If $(\mathcal{H}_\sigma)$ is satisfied for some $-\mez \leq \sigma \leq s_0-1$ then $(\mathcal{H}_{\sigma+\mez})$ is true as long as $\sigma+ \mez \leq s_0-1$.
   
  Case 2.  \quad  Let $s_0>1 + \frac{d}{2} $ and $s\geq s_0$.     If $(\mathcal{K}_\sigma)$ is satisfied for some $ s_0-1  \leq \sigma \leq s -\mez $ then $(\mathcal{K}_{\sigma+\mez})$ is true as long as $\sigma+ \mez \leq s-\mez$.
 \end{prop}
 In the sequel, case 1 will refer to Proposition \ref{etape0} and case 2 to Proposition \ref{etape1}. 
 \subsection{Non linear estimates}
 We begin by estimating the coefficients $\alpha, \beta, \gamma,$ defined in~\eqref{alpha}. We set $J= (z_0,0)$.  
 \begin{lemm}\label{est-alpha}
 Case 1. \quad Let $s_0>1 + \frac{d}{2}$.    Then there exists   $\mathcal{F}:\xR^+ \to \xR^+$ non decreasing such that
  \begin{equation*}
    \Vert \alpha \Vert_{X^{s_0-\mez}_{ul}(J)}  +  \Vert \beta \Vert_{X^{s_0-\mez}_{ul}(J)} + \Vert \gamma \Vert_{X^{s_0-\frac{3}{2}}_{ul}(J)} \\
   \leq \mathcal{F}\big(\Vert  \eta \Vert_{ H^{s_0+ \mez}_{ul}}\big).
        \end{equation*}
        
 Case 2. \quad Let $s_0>1 + \frac{d}{2}. $ Then for   $s\geq s_0$ there exists   $\mathcal{F}_1:\xR^+ \to \xR^+$ non decreasing such that
\begin{equation*}
\Vert \alpha \Vert_{X^{s -\mez}_{ul}(J)}  
+  \Vert \beta \Vert_{X^{s -\mez}_{ul}(J)} + \Vert \gamma \Vert_{X^{s -\frac{3}{2}}_{ul}(J)}  
   \leq \mathcal{F}_1\big(\Vert  \eta \Vert_{ H^{s_0+ \mez}_{ul}}\big) 
   \bigl( \Vert \eta \Vert_{H^{s+\mez}_{ul}} +  1\bigr).
   \end{equation*}
 \end{lemm}
\begin{proof}
 Since 
 \begin{equation*}
    \rho  =   (1+z)e^{\delta z \langle D_x\rangle}  \eta -z( e^{-\delta (1+z) \langle D_x\rangle}  \eta -h)
    \end{equation*}
    Lemma \ref{e^z} shows that that for all $t \in \xR, k \in \xN$ and all $a \in S^m_{1,0}$ we have  
    \begin{equation}\label{gradrh}
   \Vert \partial_z^k a(D) \rho \Vert_{X^t_{ul}(J)} \leq C (1+\Vert  \eta \Vert_{H^{t+m+k}_{ul} \times H^{t+m+k}_{ul}}).
  \end{equation}
Then according to~\eqref{alpha} we write
$$
\alpha = (\partial_z \rho)^2 - (\partial_z \rho)^2 G_1(\nabla_x \rho), \quad G_1(\xi) = \frac{ \vert \xi  \vert^2}{1+ \vert\xi \vert^2}.
$$
Case 1:\quad The estimate for $\alpha$ follows from \eqref{est:prod1} and \eqref{est:prod3} with $\mu = s_0-\mez$. 
The estimate  for $\beta$ is similar. 
Now we can write
$$
\gamma
= \frac{\partial^2_z \rho}{\partial_z \rho}
+\bigr(\partial_z \rho - (\partial_z \rho) G_1(\nabla_x \rho)\bigl) \Delta_x \rho 
+ G_2(\nabla_x \rho) \cdot\nabla_x \partial_z \rho, \quad G_2(\xi)
= \frac{\xi}{1+ \vert \xi \vert^2}.
$$
To estimate $\gamma$ we first use  the embedding
$$
X^{s_0-\mez}_{ul} \times X^{s_0-\frac{3}{2}}_{ul} \subset X^{s_0-\frac{3}{2}}_{ul}
$$
which is a consequence of Lemma~\ref{Houl} with $p=+\infty,\sigma_0 =\sigma_1 = s_0-\frac{3}{2}, \sigma_2 = s_0-\mez$ and $p =2, \sigma_0 = \sigma_2 = s_0-1, \sigma_1 = s_0-\mez$. Then we use \eqref{est:prod3} and \eqref{gradrh}.
   Case 2:\quad The estimates of $\alpha$ and $\beta$ follow from \eqref{est:prod1} and \eqref{est:prod3} with $\mu = s-\mez $ and from \eqref{gradrh}. The estimate of $\gamma$ follow from  \eqref{est:prod2} with  $\mu = s-\mez$  and \eqref{gradrh} with $t= s_0-\frac{3}{2}, m+k=2$.  
      \end{proof}
  According to~\eqref{elliptique} we have
  \begin{equation}\label{eqtildephi}
  \left\{
     \begin{aligned}
  &(\partial_z^2 + \alpha \Delta_x + \beta\cdot \nabla_x \partial_z) (\chi_k \widetilde{ v}) = \chi_kF+F_0 +F_1\\
  & F_0 := \alpha \nabla_x \chi_k \cdot \nabla_x\widetilde{v} + \beta\cdot  \nabla_x \chi_k \cdot \partial_z \widetilde{v}\\
  &F_1:= \gamma \chi_k \partial_z \widetilde{v}.
\end{aligned}
\right.
  \end{equation}
 \begin{lemm}\label{estF0F1}
 Case 1. Let $s_0>1+\frac{d}{2}$.  There exists    $\mathcal{F}:\xR^+ \to \xR^+$  non decreasing  such that  for $-\mez\leq \sigma \leq s_0-1 $ with $\sigma+\mez \leq s_0-1$ 
 $$ \sum_{j=0}^1\Vert F_j \Vert_{Y^{\sigma + \mez}(J)} \leq    \mathcal{F}\big(\Vert  \eta \Vert_{H^{s_0+\mez}_{ul} }\bigr)  \Vert \nabla_{x,z} \widetilde{v}\Vert_{X^\sigma_{ul}(J)}.$$
 
 Case 2. Assume $s_0> 1 + \frac{d}{2}. $  Then for all $s \geq s_0 $   there exists      $\mathcal{F}:\xR^+ \to \xR^+$ non decreasing such that  for $s_0-1 \leq \sigma \leq s-\mez $  with $\sigma+\mez \leq s -\mez$ we have
 \begin{equation*}
 \begin{aligned}
  \sum_{j=0}^1\Vert F_j& \Vert_{Y^{\sigma + \mez}(J)} \\
  &\leq    \mathcal{F}\big(\Vert  \eta \Vert_{   H^{s_0+\mez}_{ul} }\bigr) \Big\{ \Vert \nabla_{x,z} \widetilde{v}\Vert_{X^\sigma_{ul}(J)}+ \bigl(\Vert \eta \Vert_{H^{s+\mez}_{ul}} +1)\Vert \nabla_{x,z} \widetilde{v}\Vert_{X^{s_0-1}_{ul}(J)}\Big\}.
  \end{aligned}
  \end{equation*}
\end{lemm}
\begin{proof}
Case 1: The terms $F_0$ and $F_1$ have the same structure but $F_1$ is worse since, according to Lemma \ref{est-alpha}, $\gamma $ is bounded in a weaker norm.
 

Let us look at the term $F_1$. We can use Proposition~\ref{Houl} with $p=2, \sigma_0 = \sigma, \sigma_1 = s_0-1, \sigma_2 = \sigma. $ Indeed we have $\sigma_1 +\sigma_2 >0$ since $s_0> 1+ \frac{d}{2}, \sigma_0 = \sigma_1$, $\sigma_0 \leq \sigma_2$ due to the definition of $\sigma,$ eventually $\sigma_0< \sigma_1 +\sigma_2 -\frac{d}{2}$  since $s_0>1+\frac{d}{2}$. We obtain
$$ \Vert  \gamma \chi_k \partial_z  \widetilde{v}\Vert_{L^2(J, H^{\sigma } )} \leq  \Vert  \widetilde{\chi}_k \gamma  \Vert_{L^2(J, H^{s_0-1} )}\Vert  \chi_k \partial_z  \widetilde{v}\Vert_{L^\infty(J, H^{\sigma} )}$$
and we use Lemma~\ref{est-alpha} to conclude.

Case 2: Using \eqref{est:prod2} with $\mu = \sigma-\mez$ we obtain
$$ \Vert  \gamma \chi_k \partial_z \widetilde{v}\Vert_{L^2(J, H^\sigma)}  \leq C\bigl( \Vert \gamma \Vert_{X_{ul}^{\sigma-\mez}} \Vert \partial_z \widetilde{v}\Vert_{X^{s_0-1}_{ul}} + \Vert \gamma \Vert_{X_{ul}^{s_0-\frac{3}{2}}} \Vert \partial_z \widetilde{v}\Vert_{X^\sigma_{ul}}\bigr). $$ 
Since $\sigma-\mez \leq s - \frac{3}{2}$ we can use Lemma \ref{est-alpha} to conclude.
\end{proof}
Our next step is to replace the multiplication by $\alpha$  (resp. $\beta$) by the paramultiplication by $T_\alpha$ (resp$.T_\beta)$.  
 Recall that  we see that the equation ~\eqref{eqtildephi} can be written as
  \begin{equation}\label{eqtildephik}
  (\partial_z^2 +  {\alpha}  \Delta_x +  {\beta} \cdot \nabla_x \partial_z)  \widetilde{v}_k = F+ F_0 +F_1.
  \end{equation}
  Then we have the following result.
  \begin{lemm}\label{aTa}
  Let $J =(z_0,0), s_0> 1 + \frac{d}{2} $ and $s\geq s_0$. There exists   $\mathcal{F}:\xR^+ \to \xR^+$ non decreasing such that, for all $I\subset J,$ $\tilde{v}_k$ satisfies the paradifferential equation
  \begin{equation}\label{eqpara}
   (\partial_z^2 +  T_{{\alpha} } \Delta_x +  T_{{\beta} }\cdot \nabla_x \partial_z) \widetilde{v}_k = F+ F_0 +F_1+F_2
\end{equation}
for some remainder $F_2$ satisfying

Case1: if \quad $0\leq \sigma \leq s_0-1$ with $\sigma + \mez \leq s_0-1$  
\begin{equation}\label{F2}
\Vert F_2\Vert_{Y^{\sigma + \mez}(I)} \leq    \mathcal{F}\big(\Vert  \eta \Vert_{ H^{s_0+ \mez}_{ul}}\big) \Vert \nabla_{x,z}\widetilde{v}  \Vert_{X^{\sigma}_{ul}(I)}.
\end{equation}
Case 2. if \quad $s_0-1 \leq \sigma \leq s-\mez$  with $\sigma + \mez \leq s -\mez$    
\begin{equation}\label{F3}
\Vert F_2\Vert_{Y^{\sigma + \mez}(I)} \leq  \mathcal{F}\big(\Vert  \eta \Vert_{   H^{s_0+\mez}_{ul} }\bigr)    \bigl(\Vert \eta \Vert_{H^{s+\mez}_{ul}} +1)\Vert \nabla_{x,z} \widetilde{v}\Vert_{X^{s_0-1}_{ul}(J)}.
\end{equation}

 \begin{proof}
Case 1: Using Proposition \ref{a-Ta} with $\gamma = \sigma, r = s_0-\mez, \mu = \sigma- \mez$ which satisfy all the conditions   we obtain 
 \begin{equation*} 
\begin{aligned}
&\Vert (\alpha -T_\alpha)\Delta_x\widetilde{v}_k  \Vert_{L^2(J,H^{\sigma } )}  
 \leq C\Vert \alpha  \Vert_{L^\infty(J, H^{s_0-\mez} )_{ul}}\Vert \nabla_{x,z}\widetilde{v}_k  \Vert_{L^2(J, H^{\sigma+\mez} )},\\
&\Vert (\beta -T_\beta)\cdot \nabla_x \partial_z \widetilde{v}_k \Vert_{L^2(J,H^{\sigma } )}  
 \leq C \Vert \beta  \Vert_{L^\infty(J, H^{s_0-\mez} )_{ul}} \Vert \nabla_{x,z}\widetilde{v}_k  \Vert_{L^2(J, H^{\sigma+\mez} )}.
\end{aligned}
\end{equation*}
The result follows then from Lemma \ref{est-alpha}.

Case 2. By  Theorem 2.10 in \cite{ABZ3} we have the following estimate for $\sigma >0$
$$\Vert (\alpha -T_\alpha)u\Vert_{H^\sigma} \leq C \Vert u \Vert _{C_*^{-\mez}} \Vert \alpha \Vert_{H^{\sigma+\mez}}.$$
Let $\widetilde{\chi} \in C_0^\infty(\xR^d)$ equal to one on the support of $\chi$. We write
$$(\alpha -T_\alpha)\Delta_x \widetilde{v}_k = (\widetilde{\chi}_k\alpha -T_{\widetilde{\chi}_k\alpha})\Delta_x \widetilde{v}_k +T_{(\widetilde{\chi}_k -1)\alpha}\Delta_x \widetilde{v}_k .$$
It follows from Lemma \eqref{techpara} and the above inequality that 
$$ \Vert (\alpha -T_\alpha)\Delta_x \widetilde{v}_k \Vert_{H^\sigma} \leq C \Vert \Delta_x \widetilde{v}_k \Vert_{C_*^{-\mez}} \Vert \widetilde{\chi}_k\alpha \Vert_{H^{\sigma+\mez}}.$$
Since $H^{s_0-\mez} \subset C_*^{\mez}$ and $\sigma \leq s-1$ we obtain  
$$ \Vert (\alpha -T_\alpha)\Delta_x \widetilde{v}_k \Vert_{L^2(J,H^\sigma)} \leq C \Vert \nabla_x \widetilde{v}  \Vert_{X_{ul}^{s_0-1}(J)} \Vert  \alpha \Vert_{ X_{ul}^{s-\mez}(J)} $$
wich in view of Lemma \ref{est-alpha} Case 2, proves \eqref{F3}.
  \end{proof}
   \end{lemm} 
  Then as in~\cite{ABZ3} we perform a decoupling in a forward and a backward parabolic evolution equations. Recall that $\eta \in H^{s+\mez}_{ul}(\xR^d),$ in particular $\eta \in W^{\tdm,\infty}(\xR^d)$. We can apply Lemma 3.29 in~\cite{ABZ3} to obtain the following result.
\begin{lemm}\label{decouple}
 Let $s_0> 1+ \frac{d}{2}$. There exist two symbols $a , A $ in $\Gamma^1_{\mez}(\xR^d\times J),$  $ \mathcal{F}: \xR^+ \to \xR^+$ non decreasing and a remainder term $F_3$ such that
\begin{equation}\label{decouple1}
(\partial_z - T_{a(z)})(\partial_z-T_{A(z)})\widetilde{v}_k =F+ F_0 +F_1 +F_2 +F_3
\end{equation}
with
\begin{equation}\label{estsymb}
\sup_{z\in (-1,0)} \big(\mathcal{M}^1_{\mez}(a(z)) + \mathcal{M}^1_{\mez}(A(z))\big)  \leq \mathcal{F}(\Vert \eta \Vert_{H^{s_0+\mez}_{ul}(\xR^d)})
\end{equation}
and
\begin{equation}\label{estF3}
\Vert F_3\Vert_{Y^{\sigma +\mez}(J)} \leq \mathcal{F}(\Vert \eta \Vert_{H^{s_0+\mez}_{ul}(\xR^d)})\Vert \nabla_{x,z}\widetilde{v}  \Vert_{X^{\sigma}_{ul}(J)} 
\end{equation}
for all $\sigma \in \xR$.
\end{lemm}
\begin{proof}
We follow closely the proof of Lemma 3.29 in~\cite{ABZ3}. We set 
\begin{equation}\label{a,A}
\begin{aligned}
a &= \mez\Big( -i\beta \cdot \xi - \sqrt{4 \alpha \vert \xi \vert^2 -(\beta \cdot \xi)^2}\Big)\\
A&= \mez\Big( -i\beta \cdot \xi + \sqrt{4 \alpha \vert \xi \vert^2 -(\beta \cdot \xi)^2}\Big).
\end{aligned}
 \end{equation}
We claim the there exists $c>0$ depending only on $\Vert \eta \Vert_{H^{s_0+\mez}_{ul}(\xR^d)}$ such that 
 \begin{equation}
  \sqrt{4 \alpha \vert \xi \vert^2 -(\beta \cdot \xi)^2} \geq c \vert \xi \vert.
\end{equation}
Indeed according to~\eqref{alpha} we see by an elementary computation that
$$ 4 \alpha \vert \xi \vert^2 -(\beta \cdot \xi)^2  \geq \frac{ 4(\partial_z \rho)^2}{(1+ \vert \nabla_x\rho\vert ^2)^2}\vert \xi \vert^2. $$
Then our claim follows from~\eqref{rho>}.

 Since  we have $s_0>1 + \frac{d}{2}  $ we deduce from the paradifferential symbolic calculus that
$$ (\partial_z -T_a)(\partial_z - T_A) = \partial_z^2 + T_\alpha \Delta_x + T_\beta \cdot \nabla_x \partial_z +R_0 +R_1$$
where
\begin{equation*}
\begin{aligned}
 R_0(z) &:= T_{a(z)}T_{A(z)} -T_\alpha \Delta_x  \quad \text{ is of order } \frac{3}{2}\\
R_1(z)&:=  -T_{\partial_zA(z)} \quad \text{ is of order } \frac{3}{2}
\end{aligned}
\end{equation*}
together with the estimates
$$\sup_{z\in (-1,0)} \big( \Vert R_0(z) \Vert_{H^{t+\frac{3}{2}} \to H^t } 
+ \Vert R_1(z) \Vert_{H^{t+\frac{3}{2}} \to H^t }\big) \leq \mathcal{F}( \mathcal{M}^1_{\mez}(A ) + \mathcal{M}^1_{\mez}(a )).$$
Now the seminorms $\mathcal{M}^1_{\mez} (A )$ and $ \mathcal{M}^1_{\mez}(a)$ are bounded by the $W^{\mez, \infty}(\xR^d)$ norms of $\alpha$ and~$\beta$. Since for $f = \alpha, \beta$ we have
$$\Vert f(z) \Vert_{W^{\mez, \infty}(\xR^d)}  \leq C\Vert f(z) \Vert_{H^{s_0-\mez}_{ul}(\xR^d)} $$we deduce from Lemma~\ref{est-alpha} and the fact that the symbols of $R_j$ vanish near the origin that  for $j=0,1$
$$\Vert R_j(z)\widetilde{v}_k \Vert_{H^{\sigma } } \leq \mathcal{F}(\Vert \eta \Vert_{H^{s_0+\mez}_{ul} })\Vert \nabla_x\widetilde{v}_k \Vert_{H^{\sigma+\mez} }.$$
The proof is complete.
\end{proof}
\subsection{Proof of Proposition~\ref{rec1}}

  Case 1. \quad Assume that $(\mathcal{H}_\sigma)$ is satisfied, which means that  there exists $I_0 = (z_0,0)$ such that
  \begin{equation}\label{recphik}
  \Vert \nabla_{x,z}\widetilde{v}_k\Vert_{X^\sigma(I_0)} \\
\leq\mathcal{F}\bigl(\Vert \eta  \Vert_{   H^{s_0+ \mez}_{ul} }\bigr) \Big\{   \Vert \psi \Vert_{H^{\sigma+1}_{ul}} + \Vert F\Vert_{Y_{ul}^{\sigma }(J)} + \Vert \nabla_{x,z}\widetilde{v}   \Vert_{X_{ul}^{-\mez}(J)}  \Big\}.
     \end{equation}
  From this estimate and the Poincar\'e inequality we deduce that 
  \begin{equation}\label{recphi}
   \Vert \nabla_{x,z}\widetilde{v} \Vert_{X^\sigma(I_0)} \\
\leq      \mathcal{F} \bigl(\Vert  \eta \Vert_{  H^{s_0+ \mez}_{ul} }\bigr) \Big\{   \Vert \psi \Vert_{H^{\sigma+1}_{ul}} + \Vert F\Vert_{Y_{ul}^{\sigma }(J)} + \Vert \nabla_{x,z}\widetilde{v}   \Vert_{X_{ul}^{-\mez}(J)}  \Big\}.
 \end{equation}
 We want to prove that 
  \begin{equation}\label{recmu0}
  \Vert \nabla_{x,z}\widetilde{v}_k\Vert_{X^{\sigma+\mez}(I_1)} \\
\leq\mathcal{F}_1\bigl(\Vert  \eta \Vert_{    H^{s_0+ \mez}_{ul} }\bigr) \Big\{   \Vert \psi \Vert_{H^{\sigma+\frac{3}{2}}_{ul}} + \Vert F\Vert_{Y_{ul}^{\sigma+\mez }(J)} + \Vert \nabla_{x,z}\widetilde{v}   \Vert_{X_{ul}^{-\mez}(J)}  \Big\}. \end{equation}
Introduce a cutoff function $\theta$ such that $\theta(z_0) = 0,\quad \theta(z) = 1$ for $z \geq z_1$. Set 
\begin{equation}\label{wk=}
 \widetilde{w}_k(z, \cdot) := \theta(z) (\partial_z - T_A) \widetilde{v}_k(z,\cdot).
 \end{equation}
 It follows from Lemma~\ref{decouple}  for $z\geq z_0$ that
 \begin{equation}\label{eqwk}
  \partial_z \widetilde{w}_k - T_a \widetilde{w}_k = \theta(z)\big(F+\sum_{j =0}^3 F_j \big)+ F_4
  \end{equation}
 where 
 $$F_4 = \theta'(z) ( \partial_z - T_A) \widetilde{v}_k.$$ 
We deduce from Lemma~\ref{estF0F1}, Lemma~\ref{aTa},  Lemma~\ref{decouple} and~\eqref{recphi} that
\begin{equation}\label{estF0-3}
  \sum_{j=0}^3 \Vert \theta F_j \Vert_{Y^{\sigma + \mez}(I_0)}  \leq\mathcal{F} \bigl(\Vert  \eta  \Vert_{  H^{s_0+ \mez}_{ul}  }\bigr) \Big\{   \Vert F\Vert_{Y_{ul}^{\sigma }(J)} + \Vert \nabla_{x,z}\widetilde{v}   \Vert_{X_{ul}^{-\mez}(J)} \Big\} 
\end{equation}
 and we  see easily using~\eqref{recphik} that
\begin{equation}\label{estF4}
 \Vert F_4 \Vert_{Y^{\sigma + \mez}(I_0)} 
 \leq \mathcal{F} \bigl(\Vert  \eta \Vert_{   H^{s_0+ \mez}_{ul} }\bigr) 
 \Big\{  \Vert \psi \Vert_{H^{\sigma+1}_{ul}} +\Vert F\Vert_{Y_{ul}^{\sigma }(J)} + \Vert \nabla_{x,z}\widetilde{v}   \Vert_{X_{ul}^{-\mez}(J)}  \Big\} \cdot
 \end{equation} 
Now using Proposition 2.18 in~\cite{ABZ3} and~\eqref{eqwk},~\eqref{estF0-3} and~\eqref{estF4} we see, since $\widetilde{w}_k\arrowvert_{z=z_0} = 0,$ that 
\begin{equation}\label{estwk}
\Vert \widetilde{w}_k \Vert_{X^{\sigma + \mez}(I_0)} \leq   \mathcal{F} \bigl(\Vert  \eta  \Vert_{  H^{s_0+ \mez}_{ul} }\bigr) \Big\{  \Vert \psi \Vert_{H^{\sigma+1}_{ul}} + \Vert F\Vert_{Y_{ul}^{\sigma }(J)} + \Vert \nabla_{x,z}\widetilde{v}   \Vert_{X_{ul}^{-\mez}(J)}  \Big\} \cdot
 \end{equation}
Now notice that on $I_1 :=(z_1,0)$ we have $\theta(z) =1$ so that
\begin{equation}\label{eqphik}
 (\partial_z  -T_A )\widetilde{v}_k = w_k.
 \end{equation}
 We may apply again Proposition 2.19 in~\cite{ABZ3} and write
 \begin{equation*}
 \begin{aligned}
   \Vert \nabla_x\widetilde{v}_k \Vert_{X^{\sigma +\mez}(I_1)}  &\leq \Vert  \widetilde{v}_k \Vert_{X^{\sigma +1+\mez}(I_1)} \\
  & \leq \mathcal{F} \big(\Vert \eta \Vert_{H^{s_0+\mez}_{ul} }\big)
  \big(\Vert \widetilde{w}_k\Vert_{Y^{\sigma +\frac{3}{2}}(I_1)}  + \Vert \psi \Vert_{H^{\sigma+\frac{3}{2}}_{ul}} \big)\\
  & \leq \mathcal{F} \big(\Vert \eta \Vert_{H^{s_0+\mez}_{ul} }\big) \big(\Vert \widetilde{w}_k\Vert_{X^{\sigma + \mez}(I_1)}  +  \Vert \psi \Vert_{H^{\sigma+\frac{3}{2}}_{ul}} \big)\\
  &  \leq  \mathcal{F} \big(\Vert \eta \Vert_{H^{s_0+\mez}_{ul} }\big)
  \big(\Vert \psi \Vert_{H^{\sigma + \frac{3}{2}}_{ul}}+\Vert F\Vert_{Y_{ul}^{\sigma+\mez}(J)} +  \Vert \nabla_{x,z}\widetilde{v}   \Vert_{X_{ul}^{-\mez}(J)}  \big).
  \end{aligned}
  \end{equation*}
  The same estimate for $\partial_z\widetilde{v}_k$ follows then from~\eqref{eqphik} and~\eqref{estwk}. Thus we have proved~\eqref{recmu0} which completes the induction.

Case 2. Assuming that $(\mathcal{K}_\sigma)$ is true,   the exact same method shows that $(\mathcal{K}_{\sigma + \mez})$ holds as long as $\sigma+\mez \leq s-\mez$. Details are left to the reader.
  
\subsection{Proof of Theorem~\ref{paralinDN}}
Let $0\leq t \leq s-\mez$.  Recall that 
\begin{equation*}
G(\eta) \psi  = g_1 \partial_z\widetilde{\Phi}-g_2\cdot \nabla_x \widetilde{\Phi}  \big \arrowvert_{z=0}, \quad 
g_1  =  \frac{1 + \vert \nabla_x \rho\vert^2}{\partial_z \rho}, \quad g_2 = \nabla_x \rho. 
\end{equation*}
We shall set  
$$
g_j\arrowvert_{z=0} = g_{j }^0, \quad j=1,2, \quad A\arrowvert_{z=0} = A_0, \quad a\arrowvert_{z=0} = a_0.
$$ 
We recall  that we have set $\chi_k \widetilde{\Phi}= \widetilde{\Phi}_k$ where $\widetilde{\Phi}\arrowvert_{z=0} = \psi$  and $\widetilde{w}_k = (\partial_z - T_A)\widetilde{\Phi}_k $ (see \eqref{wk=}) for $z\in I_1$. It follows that we can write
\begin{equation}
\begin{aligned}
\chi_k  G(\eta) \psi  &=  g_{1 }^0(\partial_z \widetilde{\Phi}_k)\arrowvert_{z=0} - \chi_kg_2^0 \nabla_x \psi \\
&= g_1^0\widetilde{w}_k\arrowvert_{z=0} + g_1^0[T_{A_0}, \chi_k]\psi + \chi_k \big(g_1^0 T_{A_0} - g_2^0\cdot \nabla\big)\psi.
\end{aligned}
\end{equation}
We shall set 
\begin{equation}\label{Rj=}
R_1 = (\widetilde{\chi}_kg_1^0)[T_{A_0}, \chi_k]\psi, \quad  R_2 = (\widetilde{\chi}_k g_1^0) \widetilde{w}_k\arrowvert_{z=0}
\end{equation}
where $\widetilde{\chi}\in C_0^\infty(\xR^d)$ is equal to one on the support of $\chi $ so 
\begin{equation}\label{DNul=}
 \chi_k  G(\eta) \psi   =  \chi_k \big(g_1^0 T_{A_0} - g_2^0\cdot \nabla\big)\psi + R_1 +R_2.
\end{equation}
 Let us set $U =  [T_{A_0}, \chi_k]\psi$.  By the symbolic calculus, since $H_{ul}^{s_0+\mez} \subset W^{\mez, \infty}$ we have for all $\sigma\in \xR$
 \begin{equation}\label{est:U-i}
    \Vert U \Vert_{ H^{\sigma} } \leq \mathcal{F}\bigl(\Vert  \eta  \Vert_{H^{s_0+ \mez}_{ul}  }\bigr)  \Vert \psi \Vert_{H^{\sigma+\mez}_{ul} } .
    \end{equation}
  If $0\leq t \leq s_0-\mez$ the product law in Proposition \ref{Ho} gives
   $$\Vert \widetilde{\chi}_kg_1^0 U \Vert_{ H^{t } }\leq \Vert \widetilde{\chi}_kg_1^0 \Vert_{H^{s_0-\mez}}\Vert U \Vert_{H^t} \leq \mathcal{F}\bigl(\Vert  \eta  \Vert_{H^{s_0+ \mez}_{ul}}\bigr)\Vert U \Vert_{H^t}. $$
  If $s_0-\mez \leq t \leq s-\mez$ we use the estimation  
  $$\Vert \widetilde{\chi}_kg_1^0 U \Vert_{ H^{t } }\leq C \big(\Vert \widetilde{\chi}_kg_1^0\Vert_{H^{s_0 -1}}\Vert U \Vert_{H^t} + \Vert \widetilde{\chi}_kg_1^0\Vert_{H^{t}}\Vert U \Vert_{H^{s_0-1}}\big).$$ 
Therefore using \eqref{est:U-i} and \eqref{est:prod3} we obtain
$$
\Vert \widetilde{\chi}_kg_1^0 U \Vert_{ H^{t } }
\leq   \mathcal{F}\bigl(\Vert  (\eta,\psi)  \Vert_{H^{s_0+ \mez}_{ul}  \times H^{s_0+ \mez}_{ul}}\bigr) \big\{  \Vert \eta \Vert_{H^{s+\mez}_{ul} } +  \Vert \psi \Vert_{H^{t+\mez}_{ul} } +1 \big\}.
$$
It follows that we have
\begin{equation}\label{est:R1final}
\Vert R_1 \Vert_{ H^{t } }\leq    \mathcal{F}\bigl(\Vert  (\eta,\psi)  \Vert_{H^{s_0+ \mez}_{ul}  \times H^{s_0+ \mez}_{ul}}\bigr)\big\{  \Vert \eta \Vert_{H^{s+\mez}_{ul} } +  \Vert \psi \Vert_{H^{t+\mez}_{ul} } +1 \big\}. 
   \end{equation}
By the same argument as above we have
\begin{equation}\label{est:R2}
 \Vert R_2 \Vert_{ H^{t } }\leq \mathcal{F}\bigl(\Vert  \eta  \Vert_{H^{s_0+ \mez}_{ul}  }\bigr)\big\{  \Vert\widetilde{w}_k\arrowvert_{z=0} \Vert_{H^{t}_{ul}} + \Vert \eta \Vert_{H^{s+\mez}_{ul}} \Vert\widetilde{w}_k\arrowvert_{z=0} \Vert_{H^{s_0-1}_{ul}}\}. 
 \end{equation}
    Now using \eqref{estwk} with $\sigma = t-\mez$ we obtain in particular
 \begin{equation}\label{estwktilde}
 \Vert \widetilde{w}_k\Vert_{L^2(I_1, H^{t+\mez}) } \leq \mathcal{F}\bigl(\Vert  \eta  \Vert_{H^{s_0+ \mez}_{ul}  }\bigr)(\Vert \psi \Vert_{H^{t+\mez}_{ul} }+1).
\end{equation}
Moreover we deduce from \eqref{eqwk} that 
$$ \Vert \partial_z \widetilde{w}_k\Vert_{L^2(I_1, H^{t-\mez} )} \leq  \Vert T_a \widetilde{w}_k\Vert_{L^2(I_1, H^{t-\mez} )}    + \sum_{j=0}^4 \Vert  F_j \Vert_{L^2(I_1, H^{t-\mez} )} .$$
 It follows from \eqref{estwktilde} and the estimates already obtained on the $F_j's$ that 
 \begin{equation}\label{estdzwk}
 \Vert \partial_z \widetilde{w}_k\Vert_{L^2(I_1, H^{t-\mez} )} \leq  \mathcal{F}\bigl(\Vert  \eta  \Vert_{H^{s_0+ \mez}}\bigr)(\Vert \psi \Vert_{H^{t+\mez} }+1).
 \end{equation}
Applying Lemma \ref{lions} to $\chi_q \widetilde{w}_k$ we obtain, for $0\leq t \leq s-\mez$
 \begin{equation}\label{wkz=0}
 \Vert\widetilde{w}_k\arrowvert_{z=0} \Vert_{H^{t}_{ul}}  \leq \mathcal{F}\bigl(\Vert  \eta  \Vert_{H^{s_0+ \mez}_{ul}  }\bigr)(\Vert \psi \Vert_{H^{t+\mez}_{ul} }+1).
 \end{equation}
Combining with \eqref{est:R2} we obtain
\begin{equation}\label{est:R2final}
\Vert R_2 \Vert_{ H^{t } }\leq    \mathcal{F}\bigl(\Vert  (\eta,\psi)  \Vert_{H^{s_0+ \mez}_{ul}  \times H^{s_0+ \mez}_{ul}}\bigr)\big\{  \Vert \eta \Vert_{H^{s+\mez}_{ul} } +  \Vert \psi \Vert_{H^{t+\mez}_{ul} } +1 \big\}. 
   \end{equation}
Now we have
\begin{equation}\label{R3-R4}
\begin{aligned}
 \chi_k\big(g_1^0 T_{A_0} - g_2^0\cdot \nabla\big)\psi &= \chi_kT_{g_1^0 A_0 - i \xi \cdot g_2^0} \psi + R_3 +R_4,\\
 R_3&= \chi_k\{Ê(g_1^0-T_{g_1^0})T_{A_0}\psi -(g_2^0-T_{g_2^0})\cdot \nabla\psi\}\\
R_4 &= \chi_k\{T_{g_1^0}T_ {A_0} - T_{g_1^0 A_0}\}.
 \end{aligned}
 \end{equation}
If $0 \leq t \leq s_0-\mez$ we use Proposition \ref{a-Ta} with  $\gamma =t, r = s_0-\mez, \mu = t-\mez$ which satisfy the conditions and we obtain
$$\Vert  \chi_k\{Ê(g_1^0-T_{g_1^0})T_{A_0}\psi \Vert_{H^t} \leq \Vert g_1^0\Vert_{H^{s_0-\mez}_{ul}} \Vert T_{A_0} \psi \Vert_{H^{t-\mez}_{ul}} $$
 and an analogue estimate for the term containing $g_2^0$, from which we deduce
 \begin{equation}\label{est:R3}
 \Vert R_3 \Vert_{H^t} \leq \mathcal{F}\bigl(\Vert  \eta  \Vert_{H^{s_0+ \mez}_{ul}  }\bigr) \Vert \psi \Vert_{H^{t+\mez}_{ul}}.
 \end{equation}
From Theorem \ref{calc:symb}, $(ii)$ with $\rho = \mez$ we have
   \begin{equation}\label{est:R4}
 \Vert R_4 \Vert_{H^t} \leq \mathcal{F}\bigl(\Vert  \eta  \Vert_{H^{s_0+ \mez}_{ul}  }\bigr) \Vert \psi \Vert_{H^{t+\mez}_{ul}}.
 \end{equation}
 Summing up, using \eqref{DNul=}, \eqref{est:R1final}, 
 \eqref{est:R2final}, \eqref{est:R3}, and \eqref{est:R4}, we obtain
$$
\chi_k G(\eta)\psi =  \chi_kT_{g_1^0 A_0 - i \xi \cdot g_2^0} \psi +R_5
$$
with
 $$\Vert R_5 \Vert_{H^t} \leq  \mathcal{F}\bigl(\Vert  (\eta,\psi)  \Vert_{H^{s_0+ \mez}_{ul}  \times H^{s_0+ \mez}_{ul}}\bigr)\big\{  \Vert \eta \Vert_{H^{s+\mez}_{ul} } +  \Vert \psi \Vert_{H^{t+\mez}_{ul} } +1 \big\}.$$ 
  So Theorem \ref{paralinDN} follows from the fact that
$$g_ 1^0 A_0 -i\xi\cdot g_ 2 ^0   = \sqrt{(1+ \vert \nabla_x \eta \vert^2) \vert \xi \vert^2 - (\nabla_x \eta \cdot \xi)^2}.$$

\section{A priori estimates in the uniformly local Sobolev space}\label{sec.5}
\subsection{Reformulation of the equations}
 We introduce the following unknowns
 \begin{equation}\label{new:inc}
 \zeta = \nabla_x \eta, \quad B= (\partial_y \Phi)\arrowvert_{y = \eta}, \quad V = (\nabla_x \Phi)\arrowvert_{y=\eta}, \quad a = - (\partial_y P)\arrowvert_{y= \eta}
\end{equation}
where $\Phi$ is the velocity potential and the pressure $P$ is given by 
$$ P =  Q- \mez \vert \nabla_{x,y} \Phi \vert^2 - gy.$$
where $Q$ is obtained from $B,V,\eta$ by solving a variational problem (see \S 7.2 below for details).

We begin by a useful formula.
\begin{lemm}\label{GB=divV}
Let $I = [0,T]$ and $s_0>1+\frac{d}{2}$. For all $s\geq s_0$ one can find $\mathcal{F}: \xR^+ \to \xR^+$ non decreasing such that 
\begin{equation*} 
 G(\eta)B = - {\rm{div }} \, V + \gamma 
 \end{equation*}
with
\begin{multline*}
\Vert \gamma \Vert_{L^\infty(I, H^{s-\mez} (\xR^d))_{ul} } \\
\leq  
\mathcal{F}\big(\Vert (\eta, \psi)  \Vert_{L^\infty(I,H^{s_0+\mez} (\xR^d)\times H^{s_0+\mez} (\xR^d))_{ul}}\big)  \Big\{1+ \Vert \eta \Vert_{L^\infty(I,H^{s+\mez} (\xR^d))_{ul} }\Big\}\cdot
\end{multline*}
\end{lemm}
\begin{proof}
The estimate of the lemma will be proved first  with fixed $t$ which therefore will be skipped. Let $\theta$ be the variational solution of the problem 
$$\Delta_{x,y}\theta =0 \quad \text{in } \Omega, \quad \theta\arrowvert_{y= \eta(x)}= B. $$
Then $G(\eta)B = (\partial_y \theta - \nabla_x \eta \cdot \nabla_x \theta) \arrowvert_{y=\eta(x)}$. On the other hand since $V_i (x) = \partial_i \Phi(x, \eta(x))$ we have
\begin{equation*}
\begin{aligned}
 \text{div }V &= (\Delta_x \Phi  + \nabla_x \eta \cdot \nabla_x \partial_y \Phi)(x, \eta(x))\\
 &=   (-\partial_y^2 \Phi  + \nabla_x \eta \cdot \nabla_x \partial_y \Phi)(x, \eta(x)) \\
 &= -(\partial_y - \nabla_x\eta \cdot \nabla_x) \partial_y \Phi (x, \eta(x)).
 \end{aligned}
 \end{equation*}
 It follows that
 \begin{equation*}
    G(\eta)B + \text{div } V = \gamma, \quad \gamma = (\partial_y - \nabla_x\eta \cdot \nabla_x)(\theta - \partial_y \Phi)(x, \eta(x)).
   \end{equation*}
Setting $\Theta =  \theta - \partial_y \Phi$ we see that  $\Delta_{x,y}\Theta = 0$ in $\Omega,$ and $\Theta\arrowvert_{y= \eta} =0$. Therefore we may apply Theorem \ref{regell} with $\sigma=s-\mez$ to deduce that 
$$
\Vert \nabla_{x,z}\widetilde{\Theta}\Vert_{X^{s-\mez}_{ul}(z_0,0)} \leq \mathcal{F} (\Vert (\eta, \psi)  \Vert_{ H^{s_0+\mez}_{ul}  \times H^{s_0+\mez}_{ul} }) (1+ \Vert \eta \Vert_{ H^{s+\mez}_{ul} })
$$
where $\widetilde{\Theta} (x,z) =  {\Theta}(x, \rho(x,z))$.
Using \eqref{est:prod1} with $\mu = s-\mez$ we deduce that
 $$
\Vert \nabla_{x,z}\widetilde{\gamma}\Vert_{X^{s-\mez}_{ul}(z_0,0)} \leq \mathcal{F} (\Vert (\eta, \psi)  \Vert_{ H^{s_0+\mez}_{ul}  \times H^{s_0+\mez}_{ul} }) (1+ \Vert \eta \Vert_{ H^{s+\mez}_{ul}}).
$$  
Now using the equation satified by $\Theta$, Lemma \ref{est-alpha} and Lemma \ref{lions} we obtain the desired conclusion.
\end{proof}
 \begin{prop}\label{new:eq}
 Let $s_0 >1+ \frac{d}{2}$. Then for all $s \geq s_0$ we  have
   \begin{align}
   (\partial_t + V \cdot \nabla_x)B &= a-g, \label{eq:B}\\
   (\partial_t + V \cdot \nabla_x)V +a \zeta &=0, \label{eq:V}\\
   (\partial_t + V \cdot \nabla_x)\zeta &= G(\eta)V + \zeta G(\eta)B +R, \label{eq:zeta}
\end{align}
where  the remainder term $R = R(\eta, \psi ,V,B) $ satisfies the  estimate
\begin{equation}\label{est:R}
\Vert R\Vert_{L^\infty(I,H^{s-\mez}_{ul} )} \leq  \mathcal{F}(\Vert (\eta, \psi,V) \Vert_{L^\infty(I, H^{s_0+\mez}  \times H^{s_0+\mez}  \times H^{s_0} )_{ul}})  (1+ \Vert \eta \Vert_{L^\infty(I,H^{s+\mez} )_{ul}}) 
\end{equation}
where $H^\sigma  = H^\sigma (\xR^d)$. 
 \end{prop}
\begin{proof}
According to Proposition 4.3 in~\cite{ABZ3} the only point to be proved is the estimate~\eqref{est:R}. Let us recall how $R$ is defined. Let $\theta_i,\Phi$ be the variationnal solutions of the problems
  \begin{align}
 \Delta_{x,y}\theta_i &=0 \text{ in } \Omega, \quad \theta_i\arrowvert_{y=\eta} = V_i, \quad i= 1, \ldots,d \label{eq:theta}\\
 \Delta_{x,y}\theta_{d+1} &=0 \text{ in } \Omega, \quad \theta_{d+1}\arrowvert_{y=\eta} = B,  \label{eq:theta0}\\
   \Delta_{x,y}\Phi&=0 \text{ in } \Omega, \quad \Phi\arrowvert_{y=\eta} = \psi.  
\end{align}
   Then, (see \cite{ABZ3} Proposition 4.3)
   \begin{align}
 R_i &= (\partial_y - \nabla_x \eta \cdot \nabla_x)U_i\arrowvert_{y=\eta}, \quad U_i =  \partial_i \Phi - \theta_i, \quad i=1 \ldots,d \label{Ri}\\
   R_{d+1} &= (\partial_y - \nabla_x \eta \cdot \nabla_x)U_{d+1}\arrowvert_{y=\eta}, \quad U_{d+1}=  \partial_y \Phi - \theta_{d+1}. \label{Rd+1}
 \end{align}
  First of all for $i=1,\ldots,d$ we have $\Delta_{x,y} U_i = 0$ in $\Omega$ and $U_i\arrowvert_{y=\eta} = 0$ since $\partial_i \Phi \arrowvert_{y=\eta}= V_i$. Denoting by $\widetilde{U}_i$ the image of $U_i$ by the diffeomorphism~\eqref{rho} we see that $\widetilde{U}_i$  satisfies the equation~\eqref{elliptique}  with $F=\psi =0$. It follows from Theorem \ref{regell} wit $\sigma = s-\mez$ that 
  \begin{equation}\label{est:Ui}
\Vert \nabla_{x,z}\widetilde{U}_i \Vert_{X^{s-\mez}_{ul}(J)} \leq  \mathcal{F}(\Vert (\eta, \psi) \Vert_{H^{s_0+\mez}_{ul}\times H^{s_0+\mez}_{ul}}) (1+ \Vert \eta \Vert_{H^{s+\mez}_{ul}})\Vert \nabla_{x,z}\widetilde{U}_i \Vert_{X^{-\mez}_{ul}(J)}.
\end{equation}
 We are left with the  condition~\eqref{-mez-i}, that is,
$$
\big\Vert \nabla_{x,z}\widetilde{U}_i \big\Vert_{X^{-\mez}_{ul}(J)} <+\infty.
$$
Indeed, since $\theta_i$ is the variationnal solution of~\eqref{eq:theta} Corollary~\ref{coroetape1} shows that 
$$
\big\Vert \nabla_{x,z}\widetilde{\theta}_i \big\Vert_{X^{-\mez}_{ul}(J)} 
\leq   \mathcal{F}\big(\Vert \eta \Vert_{H^{s_0+\mez}_{ul} }\big) \Vert V_i\Vert_{H^\mez_{ul} }.
$$
Now $\widetilde{\partial_i \Phi} = (\partial_i - \frac{\partial_i \rho}{\partial_z \rho}\partial_z) \widetilde{\Phi}$. 
It follows that
$$
\big\Vert \nabla_x \widetilde{\partial_i \Phi} \big\Vert_{X^{-\mez}_{ul}(J)}
\leq \big\Vert \widetilde{\partial_i \Phi} \big\Vert_{X^{\mez}_{ul}(J)} 
\leq \big\Vert  \partial_i\widetilde{\Phi}\big\Vert_{X^{ \mez}_{ul}(J)} 
+ \Big\Vert   \frac{\partial_i \rho}{\partial_z \rho}\partial_z\widetilde{ \Phi} \Big\Vert_{X^{ \mez}_{ul}(J)}.
$$
Now we use the following facts:    since $s_0>1+\frac{d}{2}$ one has  $X^{ s_0-\mez}_{ul}(J) \subset X^{ \mez}_{ul}(J)$; moreover $X^{ s_0-\mez}_{ul}(J)$ is an algebra and eventually $\Vert   \frac{\partial_i \rho}{\partial_z \rho}\Vert_{X^{ s_0-1}_{ul}(J) } \leq \mathcal{F}(\Vert \eta \Vert_{H^{s_0+\mez}_{ul} })$. We deduce using Corollary~\ref{coroetape1} that 
$$
\big\Vert \nabla_x \widetilde{\partial_i \Phi} \big\Vert_{X^{-\mez}_{ul}(J)}
\leq \mathcal{F}_1\big(\Vert \eta \Vert_{H^{s_0+\mez}_{ul} }\big)\Vert \psi \Vert_{H^{s_0+\mez}_{ul} }.$$
  To estimate the term $\Vert \partial_z \widetilde{\partial_i \Phi} \Vert_{X^{-\mez}_{ul}(J)}$ we follow the same path using furthermore the equation~\eqref{elliptique} with $F=0$ satisfied by $\widetilde{\Phi}$. We obtain eventually
$$
\big\Vert \nabla_{x,z}\widetilde{U}_i \big\Vert_{X^{-\mez}_{ul}(J)} 
\leq  \mathcal{F}_2\Big(\Vert (\eta,\psi,V) \Vert_{H^{s_0+\mez}_{ul}\times H^{s_0+\mez}_{ul} \times H^{s_0}_{ul}}\Big).
$$
Using \eqref{est:Ui} we obtain
\begin{equation}\label{Ui}
\big\Vert \nabla_{x,z}\widetilde{U}_i \big\Vert_{X^{s-\mez}_{ul}(J)} 
\leq  \mathcal{F}\Big(\Vert (\eta, \psi,V) \Vert_{H^{s_0+\mez}_{ul}\times H^{s_0+\mez}_{ul} \times H^{s_0}_{ul}  }\Big)
\Big\{1+ \Vert \eta \Vert_{H^{s+\mez}_{ul}}\Big\} .
    \end{equation}
Now from ~\eqref{Ri} we have 
\begin{equation}\label{Ritilde}
 R_i := \mathcal{R}_i \arrowvert_{z=0}, \quad \mathcal{R}_i= \Big(\frac{1+ \vert \nabla_x \rho \vert^2}{\partial_z \rho} \partial_z - \nabla_x \rho\cdot \nabla_x \Big) \widetilde{U}_i:= (A\partial_z  + B\cdot \nabla_x) \widetilde{U}_i.
 \end{equation}
 Using \eqref{est:prod1} with $\mu = s-\mez$ and \eqref{Ui} with $s$ and $s_0$ we obtain
\begin{equation}\label{est:Ri}
\big\Vert \mathcal{R}_i\big\Vert_{X^{s-\mez}_{ul}(J)} 
\leq  \mathcal{F}\Big(\Vert (\eta, \psi,V) \Vert_{H^{s_0+\mez}_{ul}\times H^{s_0+\mez}_{ul} \times H^{s_0}_{ul}  }\Big)
\Big\{1+ \Vert \eta \Vert_{H^{s+\mez}_{ul}}\Big\} \cdot
\end{equation}
Now we claim that 
\begin{equation}\label{dzRi}
\big\Vert \partial_z \mathcal{R}_i\big\Vert_{X^{s-\frac{3}{2}}_{ul}(J)}
\leq  \mathcal{F}\Big(\Vert (\eta, \psi,V) \Vert_{H^{s_0+\mez}_{ul}\times H^{s_0+\mez}_{ul} \times H^{s_0}_{ul}  }\Big)
\Big\{1+ \Vert \eta \Vert_{H^{s+\mez}_{ul}}\Big\}\cdot
    \end{equation}
   Indeed we can write
 $$ \partial_z \mathcal{R}_i = (\partial_z A) \partial_z \widetilde{U}_i + (\partial_z B) \cdot \nabla_x \widetilde{U}_i - (\text{div} B) \partial_z \widetilde{U}_i + \nabla_x \cdot(B \partial_z \widetilde{U}_i) + A \partial_z^2 \widetilde{U}_i.$$
The first three terms are bounded using \eqref{est:prod2} with $\mu = s-\frac{3}{2}$ and \eqref{Ui}, 
the fourth is estimated using  \eqref{est:prod1} with $\mu = s-\mez$ and \eqref{Ui}, eventually for 
the last term we use the fact that $\partial_z^2 \widetilde{U}_i = -(\alpha \Delta_x + \beta\cdot \nabla \partial_z - \gamma \partial_z) \widetilde{U}_i$ together with \eqref{est:prod1}, \eqref{est:prod2}, and 
\eqref{Ui}. Finally from \eqref{est:Ri} and \eqref{dzRi}, using Lemma \ref{lions} we obtain \eqref{est:R} for $R_i$.

We use exactly the same argument to estimate $\Vert R_{d+1}\Vert_{H^{s-\mez}_{ul}(\xR^d)}$. This completes the proof of Proposition~\ref{new:eq}
 \end{proof} 
 \subsection{Estimate of the Taylor coefficient}
  \begin{prop}\label{prop:taylor}
 Let $I = [0,T],   s _0>1+ \frac{d}{2}$. For all $s \geq s_0$ there exists  $\mathcal{F} : \xR^+ \to \xR^+$ non decreasing  such that,  with $ H^{\sigma}   = H^{\sigma}(\xR^d)$
 \begin{multline}\label{est:a}
 \Vert a  -g \Vert_{L^\infty(I,H^{s-\mez})_{ul}} 
 \leq \mathcal{F}\big(\Vert (\eta ,\psi,V,B )\Vert_{L^\infty(I, H^{s_0+\mez} \times H^{s_0+\mez}  \times H^{s_0}  \times H^{s_0})_{ul}}\big) \\
 \times \Big \{1 + \Vert \eta \Vert_{L^\infty(I,H^{s+\mez} )_{ul}} + \Vert B \Vert_{L^\infty(I,H^{s} )_{ul}}
 +\Vert V\Vert_{L^\infty(I,H^{s })_{ul}}\Big\}.
 \end{multline}
 \end{prop}
 For convenience we shall set in what follows
 \begin{equation}\label{F0}
 \mathcal{F}_0 =: \mathcal{F}\Big(\Vert (\eta ,\psi,V,B )\Vert_{L^\infty(I, H^{s_0+\mez} \times H^{s_0+\mez}  \times H^{s_0}  \times H^{s_0})_{ul}}\Big)
 \end{equation}
 where $\mathcal{F}: \xR^+ \to \xR^+$ is a non decreasing function which may change from line to line.
  Before giving the proof of this result let us recall how $a$ is defined. As is \cite{ABZ3} the pressure is defined as follows. Le $Q$ be the variationnal solution of the problem
 \begin{equation}\label{def:Q}
 \Delta_{x,y} Q = 0 \quad \text{in } \Omega, \quad Q\arrowvert_{y=\eta} = \mez B^2 + \mez \vert V \vert^2 + g \eta.
 \end{equation}
Then
\begin{equation}\label{def:P}
P = Q - \mez \vert \nabla_{x,y} \Phi \vert^2 -gy.
\end{equation}
 It is shown in~\cite{ABZ3} that $Q = -\partial_t \Phi$. Then 
 \begin{equation}\label{def:a}
 a = - \partial_y P\arrowvert_{y = \eta}.
\end{equation}
We deduce from~\eqref{def:Q}, \eqref{def:P} that $P$ is solution of the problem
\begin{equation}\label{eq:P}
\Delta_{x,y}P = - \vert \nabla_{x,y}^2 \Phi \vert^2, \quad P\arrowvert_{y=\eta} =0.
\end{equation}
Denoting, as usual, by $\widetilde{P}, \widetilde{Q}, \widetilde{\Phi}$ the images of $P,Q,\Phi$ by the diffeomorphism~\eqref {rho}  we have, using the notation~\eqref{lambda},
 $$\widetilde{P}Ê= \widetilde{Q} - \mez(\Lambda_1 \widetilde{\Phi})^2 - \mez \vert \Lambda_2 \widetilde{\Phi}\vert ^2 - g \rho $$
and we see that $\widetilde{P}$ is a solution of the problem in   \, $\xR^d \times  J,$
\begin{equation}\label{eq:Ptilde}
\big(\partial_z^2 + \alpha \Delta_x + \beta\cdot \nabla_x \partial_z - \gamma \partial_z\big) \widetilde{P} 
= -\alpha \sum_{i,j=1}^2\big\vert \Lambda_i \Lambda_j\widetilde{\Phi}\big\vert^2,\quad   \widetilde{P} \arrowvert_{z=0} =0.
\end{equation}
Notice that we have
\begin{equation}\label{VB}
 \Lambda_1 \widetilde{\Phi}\arrowvert_{z=0} =B, \quad  \Lambda_2 \widetilde{\Phi}\arrowvert_{z=0} = V.
\end{equation}
\begin{proof}[Proof of Proposition~\ref{prop:taylor}]
Below  the time is fixed and we will skip it. We want to  apply Theorem~\ref{regell} with $\sigma = s-\mez$, 
so we must estimate the source term and show that the condition~\eqref{-mez-i} is satisfied. 
We claim that (see \eqref{F0})
\begin{equation}\label{-mezP}
\big\Vert \nabla_{x,z} \widetilde{P} \big\Vert_{X^{-\mez}_{ul}(J)} \leq \mathcal{F}_0
\end{equation}
 First of all since $Q$ is a the variationnal solution of~\eqref{def:Q} we have according to Corollary~\ref{rec00}
\begin{equation*}
\big\Vert \nabla_{x,z} \widetilde{Q} \big\Vert_{X^{-\mez}_{ul}} 
\leq \mathcal{F} \big(\Vert \eta \Vert_{H^{s_0+\mez}_{ul} }\big) \big(\big\Vert B^2 \big\Vert_{H^\mez_{ul} } 
+ \big\Vert \vert V \vert^2 \big\Vert_{H^\mez_{ul} } +  \Vert  \eta \Vert_{H^\mez_{ul} }\big).
\end{equation*}
Using the fact that  $H^{s_0}_{ul}(\xR^d)$ is an algebra contained in $H^\mez_{ul}(\xR^d)$ we obtain
 \begin{equation}\label{-mezQ}
\big\Vert \nabla_{x,z} \widetilde{Q} \big\Vert_{X^{-\mez}_{ul} } \leq \mathcal{F}_0. 
\end{equation}
The estimate of $\Vert \nabla_{x,z} \rho\Vert_{X^{-\mez}_{ul}(\xR^d)} $ by the right hand side of~\eqref{-mezP} is straightforward. Now, for $j = 1,2$ since $ X^{s_0-\mez}_{ul}(J)$ is an algebra contained in $X^{ \mez}_{ul}(J)$ we have 
$$
\big\Vert \nabla_x \vert \Lambda_j \widetilde{\Phi}\vert^2 \big\Vert_{X^{-\mez}_{ul}(J)} 
\leq C \big\Vert  \vert \Lambda_j \widetilde{\Phi}\vert^2 \big\Vert_{X^{ \mez}_{ul}(J)} 
\leq C' \big\Vert   \Lambda_j\widetilde{\Phi} \big\Vert^2_{X^{s_0- \mez}_{ul}(J)}
$$
so using Corollary~\ref{coroetape1} with $\sigma = s_0-\mez$ and the estimates on $\rho$ we obtain
\begin{equation}\label{est:lambdaj}
\big\Vert \nabla_x \vert \Lambda_j \widetilde{\Phi}\vert^2 \big\Vert_{X^{-\mez}_{ul}(J)} \leq \mathcal{F}_3 \big( \Vert (\eta ,\psi) \Vert_{H^{s_0+\mez}_{ul} \times H^{s_0+\mez}_{ul} }\big).
\end{equation}
The same kind of arguments show that
\begin{equation}\label{est:lambdaj1}
\big\Vert \partial_z\vert \Lambda_j \widetilde{\Phi}\vert^2 \big\Vert_{X^{-\mez}_{ul}(J)} \leq \mathcal{F}_3 \big( \Vert \eta ,\psi \Vert_{H^{s_0+\mez}_{ul}  \times H^{s_0+\mez}_{ul} }\big).
\end{equation}
Using~\eqref{-mezQ},~\eqref{est:lambdaj}, and~\eqref{est:lambdaj1} we obtain the claim~\eqref{-mezP}.
 Now we estimate the source term $F = -\alpha \sum_{i,j=1}^2\big\vert \Lambda_i \Lambda_j\widetilde{\Phi}\big\vert^2$ in equation~\eqref{eq:Ptilde}. 
Since $s-\mez> \frac{d}{2} $ we can write
\begin{align*}
\Vert F \Vert_{Y^{s-\mez}_{ul}(J)}& \leq \Vert F \Vert_{ L^1(J,H^{s-\mez} )_{ul}} \\
&\leq  C\Vert \alpha \Vert_{ L^\infty(J,H^{s-\mez})_{ul}}   \sum_{i,j=1}^2\Vert \Lambda_i \Lambda_j\widetilde{\Phi}\Vert^2_{ L^2(J,H^{s-\mez} )_{ul}}.
\end{align*}
 Since $(\Lambda_1^2 + \Lambda_2^2)\widetilde{\Phi}=0$ an $\Lambda_1, \Lambda_2$ commute, we have for $j=1,2$
 $$(\Lambda_1^2 + \Lambda_2^2) \Lambda_j \widetilde{\Phi} =0, \quad (\Lambda_1\widetilde{\Phi}, \Lambda_2\widetilde{\Phi})\arrowvert_{z=0} = (B,V)\in H^s_{ul} \times H^s_{ul}.$$
 Since we have (see \eqref{F0})
$$\Vert \Lambda_j \widetilde{\Phi}\Vert_{X^{-\mez}_{ul}(J)} \leq \mathcal{F}_0 $$
 we can apply Theorem~\ref{regell} with $\sigma = s-1$ and conclude that
\begin{equation*}
\Vert \nabla_{x,z}\Lambda_j\widetilde{\Phi}\Vert_{L^2(J,H^{s-\mez})_{ul}} 
 \leq \mathcal{F}_0 \cdot \big(1 + \Vert \eta \Vert_{H^{s +\mez}_{ul}} + \Vert B \Vert_{H^{s  }_{ul}}+\Vert V\Vert_{H^{s }_{ul}}\big).
 \end{equation*}
  Since $\Lambda_1 = \frac{1}{\partial_z \rho} \partial_z, \Lambda_2 = \nabla_x -  \frac{\nabla_x \rho}{\partial_z \rho} \partial_z,$ using \eqref{est:prod1}, the estimates on $\rho$, the above inequality for $s=s_0$ and for $s$ we obtain
  \begin{equation}\label{lambdaij}
   \Vert\Lambda_i \Lambda_j\widetilde{\Phi}\Vert_{L^2(J,H^{s-\mez}(\xR^d) _{ul}} \leq \mathcal{F}_0 \cdot
   \big(1 + \Vert \eta \Vert_{H^{s+\mez}_{ul}} + \Vert B \Vert_{H^{s  }_{ul}}+\Vert V\Vert_{H^{s }_{ul}}\big).
   \end{equation}
    It follows easily that 
 \begin{equation}\label{est:F}
   \Vert F \Vert_{Y^{s-\mez}_{ul}(J)}\leq \mathcal{F}_0\cdot
   \big(1 + \Vert \eta \Vert_{H^{s+\mez}_{ul}} + \Vert B \Vert_{H^{s  }_{ul}}+\Vert V\Vert_{H^{s }_{ul}}\big).
    \end{equation}
Using \eqref{eq:Ptilde},~\eqref{-mezP},~\eqref{est:F},  Theorem~\ref{regell} and $ X^{s-\mez}_{ul}(z_0,0) \subset L^2((z_0,0), H^s)_{ul}$ we obtain, using \eqref{F0}
\begin{equation}\label{est:dzPtilde}
\Vert \nabla_{x,z} \widetilde{P} \Vert_{ L^2((z_0,0), H^s)_{ul}} \leq \mathcal{F}_0\cdot
\big(1 + \Vert \eta \Vert_{H^{s+\mez}_{ul}} + \Vert B \Vert_{H^{s  }_{ul}}+\Vert V\Vert_{H^{s }_{ul}}\big). 
\end{equation}
We claim that 
\begin{equation*}
\big\Vert \partial_z^2 \widetilde{P} \big\Vert_{ L^2((z_0,0), H^{s-1})_{ul}}
\leq \mathcal{F}_0 \cdot
\big(1 + \Vert \eta \Vert_{H^{s+\mez}_{ul}} + \Vert B \Vert_{H^{s  }_{ul}}+\Vert V\Vert_{H^{s }_{ul}}\big). 
\end{equation*}
Indeed this follows from \eqref{eq:Ptilde}, \eqref{lambdaij}, \eqref{est:prod1}, \eqref{est:prod2}, \eqref{est:dzPtilde}. 

Noticing that $a = \frac{1}{\partial_z \rho}\partial_z \widetilde{P}\arrowvert_{z=0}$ and applying Lemma \ref{lions} 
we obtain the conclusion of Proposition~\ref{prop:taylor}.
\end{proof}
\subsection{Paralinearization of the system}
As in~\cite{ABZ3} for $s>1+\frac{d}{2}$ we set
\begin{equation}\label{def:Uzetas}
\left\{
\begin{aligned}
&U_s = \langle D_x \rangle^s V + T_\zeta  \langle D_x \rangle^s B,\\
& \zeta_s =  \langle D_x \rangle^s \zeta.
\end{aligned}
\right.
\end{equation}
and we recall that we have set  (see the statement of Theorem~\ref{paralinDN})  
\begin{equation}\label{lambda1}  
 \lambda(t,x,\xi):= \sqrt{(1+ \vert \nabla_x \eta(t,x) \vert ^2) \vert \xi \vert^2 -(\nabla_x \eta(t,x)\cdot \xi)^2}.
 \end{equation}
 \begin{prop}\label{para:system}
 Let $s_0>1+\frac{d}{2}$. For all $s \geq s_0$ there exists  $\mathcal{F}:\xR^+ \to \xR^+$   non decreasing   such that 
 \begin{align}
 &(\partial_t + T_V\cdot \nabla_x)U_s +T_a \zeta_s = f_1,\label{para1}\\
   &(\partial_t + T_V\cdot \nabla_x))\zeta_s - T_\lambda U_s =f_2,\label{para2}
  \end{align}
   where for each time $t\in [0,T]$
  \begin{equation}\label{est:f}\begin{aligned}
  \Vert (f_1(t), f_2(t)) \Vert_{L^2_{ul} \times H^{-\mez}_{ul} } 
  \leq & \mathcal{F}\Big(\Vert (\eta(t) , \psi(t),V(t),B(t) )\Vert_{H^{s_0+\mez}_{ul}  \times H^{s_0+\mez}_{ul} \times H^{s_0 }_{ul} \times H^{s_0 }_{ul}}\Big) \\
  &\times\Big\{1 + \Vert \eta(t) \Vert_{H^{s+\mez}_{ul}} + \Vert B(t) \Vert_{H^{s  }_{ul}}+\Vert V(t)\Vert_{H^{s }_{ul}}\Big\}.   \end{aligned} 
\end{equation}
\end{prop}
\begin{proof}
We follow the proof of Proposition 4.9 in~\cite{ABZ3}. 
First of all we shall say that a positive quantity $A(t)$ is controlled if it is bounded 
by the right hand side of~\eqref{est:f}. 
Here $t$ will be fixed so we will skip it, taking care that the estimates 
are uniform with respect to $t\in [0,T]$. We also set
$$
\mathcal{L}_0 = \partial_t + T_V\cdot\nabla_x.
$$
  \subsubsection{Paralinearization of the first equation} We begin by proving that
 \begin{equation}\label{para3}
 \mathcal{L}_0  V + T_a \zeta +T_\zeta\mathcal{L}_0 B = h_1
  \end{equation}
where $\Vert h_1 \Vert_{H^s_{ul} }$ is controlled. 
Indeed using~\eqref{eq:B},~\eqref{eq:V} and the fact that $T_\zeta g=0$ we see that  $h_1 = (T_V -V)\cdot \nabla_x V - R(a,\zeta)$.  By Proposition \ref{a-Ta} with $\gamma = s, r=s, \mu = s_0-1$ we see that $\Vert (T_V -V)\cdot \nabla_x V \Vert_{ H^{s }_{ul} } \leq C \Vert V \Vert _{H^{s }_{ul} } \Vert V \Vert _{H^{s_0 }_{ul} }$. On the other hand since $s_0>1+ \frac{d}{2},$ Proposition \ref{paralin} with $ \alpha = s-\mez, \beta = s_0-\mez$  shows that  
$$\Vert R(a, \zeta) \Vert_{ H^{s}_{ul}(\xR^d)} \leq C \Vert a \Vert_{ H^{s-\mez }_{ul}(\xR^d)} \Vert \nabla_x \eta \Vert_{ H^{s_0 -\mez}_{ul}(\xR^d) }.$$
 These estimates together with Proposition \ref{prop:taylor} prove that $h_1$ is controlled.
 \subsubsection{Higher order energy estimates} Now we apply the operator $\langle D_x \rangle^s = (I-\Delta_x)^{\frac{s}{2}}$ to the equation~\eqref{para3} and we commute. We claim that we obtain
\begin{equation}\label{para4}
 \mathcal{L}_0  \langle D_x \rangle^sV + T_a \langle D_x \rangle^s\zeta +T_\zeta\mathcal{L}_0 \langle D_x \rangle^s B = h_2
 \end{equation}
where $\Vert h_2(t) \Vert_{L^2_{ul}(\xR^d)}$ is controlled. Indeed this is a consequence of the following estimates
\begin{align*}
&\Vert [T_V \cdot \nabla_x, \langle D_x \rangle^s ] \Vert_{H^s_{ul} \to L_{ul}^2} \leq C\Vert V \Vert_{W^{1,\infty}} \leq C'\Vert V \Vert_{H^{s_0}_{ul}},\\
 &\Vert [T_a, \langle D_x \rangle^s ] \Vert_{H^{s-\mez}_{ul} \to L^2_{ul}} \leq C \Vert a \Vert_{W^{\mez,\infty} } \leq C'\Vert a \Vert_{H^{s_0-\mez}_{ul} },\\
&\Vert [T_\zeta, \langle D_x \rangle^s ] \Vert_{H^{s-\mez}_{ul} \to L^2_{ul}} \leq C\Vert \zeta \Vert_{W^{\mez,\infty}  }  \leq C'\Vert \zeta \Vert_{H^{s_0-\mez}_{ul}} ,
 \end{align*}
which follow from Theorem~\ref{calc:symb}. Now Lemma \ref{commutation} shows that
  $$\Vert [T_\zeta, \mathcal{L}_0] \langle D_x \rangle^s B\Vert_{L_{ul}^2 } 
\leq C( \Vert \zeta \Vert_{L^\infty }\Vert V \Vert _{W^{1+\eps, \infty} } + \Vert \mathcal{L}_0 \zeta \Vert_{L^\infty }) \Vert B \Vert_{H^s_{ul} }.$$
   Since $s_0>1+\frac{d}{2}$ one can find $\eps>0$ such that  $H^{s_0}_{ul}(\xR^d)$ is continuously embedded in  $W^{1+\eps,\infty}(\xR^d)$. Therefore we obtain
$$\Vert [T_\zeta, \mathcal{L}_0] \langle D_x \rangle^s B\Vert_{L_{ul}^2} \leq  \mathcal{F} (\Vert (\eta, B ,V)\Vert_{H^{s_0+\mez}_{ul}   \times H^{s_0 }_{ul}  \times H^{s_0 }_{ul} }) \Vert B \Vert_{H^s_{ul}} $$
which shows that $\Vert [T_\zeta, \mathcal{L}_0] \langle D_x \rangle^s B\Vert_{L_{ul}^2 } $ is controlled.
Using~\eqref{para4} and \eqref{def:Uzetas} we obtain ~\eqref{para1}.
\subsubsection{Paralinearization of the second equation}
\begin{equation}\label{para5}
  (\partial_t + V \cdot \nabla)\zeta = G(\eta)V + \zeta G(\eta) B +R.
\end{equation}
We first replace $V$ by $T_V$ modulo a controlled term. To do this we use Proposition~\ref{a-Ta} with $\gamma = s - \mez, r=s,\mu = s_0-\frac{3}{2}$ and we obtain
\begin{equation}\label{para6}
 \Vert (V-T_V)\nabla \zeta \Vert_{H^{s-\mez}_{ul} } \leq \Vert VÊ\Vert_{H^s_{ul} } \Vert \eta \Vert_{H^{s_0+\mez}_{ul} }.
 \end{equation}
Next we paralinearize the Dirichlet-Neumann part. To achieve this paralinearization we use the analysis performed in Section 2. Using Theorem~\ref{paralinDN} with $t = s-\mez$ we can write
\begin{equation}\label{para7}
 G(\eta)V + \zeta G(\eta) B = T_\lambda U + \mathcal{R}
 \end{equation}
where 
\begin{equation}
\begin{aligned}
U &= V+ T_\zeta B\\
\mathcal{R} &= [T_\zeta, T_\lambda] B + R(\eta) V + \zeta R(\eta)B  + (\zeta - T_\zeta) T_\lambda B.
\end{aligned}
\end{equation}
and
\begin{multline*}
\Vert R(\eta) V\Vert_{H^{s-\mez}_{ul} }   +  \Vert \zeta R(\eta)B \Vert _{H^{s-\mez}_{ul} } \\
 \leq \mathcal{F}(\Vert (\eta , B ,V )\Vert_{H^{s_0+\mez}_{ul}   \times H^{s_0 }_{ul} \times H^{s_0 }_{ul} })\big(1 + \Vert \eta  \Vert_{H^{s+\mez}_{ul}} + \Vert  B \Vert_{H^{s  }_{ul}}+\Vert V \Vert_{H^{s }_{ul}}\big).
 \end{multline*}
 Using again Proposition~\ref{a-Ta} with $ \gamma = s-\mez, r = s-\mez, \mu = s_0-1$ and Theorem~\ref{calc:symb} $(i)$ we can write
$$
\Vert (\zeta - T_\zeta) T_\lambda B \Vert_{H^{s-\mez}_{ul} }  
\leq C \Vert \eta \Vert_{H^{s+\mez}_{ul} } M^1_0(\lambda) 
\Vert B \Vert_{H^ {s_0} _{ul} } \leq \mathcal{F}(\Vert \eta \Vert_{H^{s+\mez}_{ul} })\Vert B \Vert_{H^ {s_0} _{ul} },
$$
which shows that this term is controlled. Eventually, by Theorem~\ref{calc:symb} $(ii),$ 
the term  $\Vert  [T_\zeta, T_\lambda] B \Vert_{H^{s-\mez}_{ul} }$ is also controlled. 
Therefore we have the equality~\eqref{para7} with 
$ \Vert \mathcal{R}\Vert_{H^{s-\mez}_{ul} }$ controlled. 
It follows from~\eqref{para5},~\eqref{para6} and~\eqref{para7} that 
$$
\mathcal{L}_0 \zeta = T_\lambda U + \mathcal{R}
$$
where $ \Vert \mathcal{R}\Vert_{H^{s-\mez}_{ul} } $ controlled. 
As in the second step by commuting the equation~\eqref{para6} with $\langle D_x \rangle^s$ 
we obtain equation~\eqref{para2}. This completes the proof of Proposition~\ref{para:system}.
 \end{proof}
 \subsection{Symmetrization of the equations}
 As in~\cite{ABZ3} before proving an $L^2$ estimate for our system we begin 
 by performing a symmetrization of the non diagonal part. 
 Recall that $  \mathcal{L}_0  = \partial_t + T_V\cdot \nabla$.
 \begin{prop}\label{psym}
 Introduce the symbols
 $$\gamma = \sqrt{a\lambda}, \quad q = \sqrt{\frac{a}{\lambda}},$$
 where $a$ is the Taylor coefficient and $\lambda$ is recalled in~\eqref{lambda1}. Set $\theta_s = T_q \zeta_s $ and  $\mathcal{L}_0=\partial_t+T_V\cdot\nabla$. Then
 \begin{align}
 \mathcal{L}_0 U_s +T_\gamma \theta_s &=F_1 \label{syst:red1}\\
  \mathcal{L}_0 \theta_s -T_\gamma U_s &= F_2 \label{syst:red2}
\end{align}
where  $F_1,F_2$ satisfy, with $L^2_{ul}= L^2_{ul}(\xR^d), H_{ul}^\sigma = H_{ul}^\sigma(\xR^d),$
\begin{multline*}
\Vert (F_1(t), F_2(t)) \Vert_{L^2_{ul}  \times L^2_{ul} }\leq 
    \mathcal{F}(\Vert (\eta(t), B(t),V(t))\Vert_{H^{s_0+\mez}_{ul}   \times H^{s_0}_{ul} \times H^{s_0}_{ul} })\\
\cdot \big(1 + \Vert \eta(t)  \Vert_{H^{s+\mez}_{ul}  } + \Vert  B(t) \Vert_{H^{s  }_{ul} }+\Vert V(t) \Vert_{H^{s }_{ul}}\big), 
 \end{multline*}
for some non decreasing function $\mathcal{F}: \xR^+ \to \xR^+$ and all $t \in [0,T]$.
\end{prop}
\begin{proof}
We follow \cite{ABZ3}. From~\eqref{para1} and ~\eqref{para2} we have
\begin{align*}
&F_1:= f_1 + (T_\gamma T_q -T_a) \zeta_s\\
&F_2:= T_q f_2 +(T_qT_\lambda - T_\gamma) U_s -[T_q, \mathcal{L}_0] \zeta_s.
\end{align*}
Then the Proposition follows from Lemma \ref{commutation} and from the symbolic calculus.
\end{proof}
We can now state our $L^2$ estimate. Let us set with $H_{ul}^\sigma = H_{ul}^\sigma(\xR^d)$
\begin{equation*}
\left\{
\begin{aligned}
&M_s(0) =   \Vert (\eta(0) ,\psi(0),B(0),V(0))\Vert_{H^{s+\mez}_{ul} \times H^{s+\mez}_{ul} \times H^{s }_{ul} \times H^{s }_{ul} },\\
  & M_s(T) = \sup_{t\in [0,T]}  \Vert (\eta(t) ,\psi(t),B(t),V(t))\Vert_{H^{s+\mez}_{ul} \times H^{s+\mez}_{ul} \times H^{s }_{ul} \times H^{s }_{ul} }. 
\end{aligned}
\right.
\end{equation*}
\begin{prop}\label{est:L2}
 There exists     $\mathcal{F}: \xR^+ \to \xR^+$   non decreasing  such that 
\begin{equation*}
\begin{aligned}
 &(i)  \quad \Vert U_s(t) \Vert_{ L^2_{ul}  } 
+ \Vert \theta_s(t) \Vert_{  L^2_{ul} } \leq \mathcal{F} (M_{s_0}(t)) M_s(t), \quad t\in I=:[0,T],\\ 
 &(ii) \quad  
   \Vert U_s \Vert_{L^\infty(I, L^2 )_{ul}}  + \Vert \theta_s \Vert_{L^\infty(I, L^2 )_{ul}} \leq 
 \mathcal{F} (TM_{s_0}(T))\Big\{ M_s(0) + \sqrt{T}M_s(T)\Big\}.
 \end{aligned}
\end{equation*}
\end{prop}
\begin{proof}
$(i)$ This follows easily  from the definition of $U_s$ and $ \theta_s$ given in \eqref{def:Uzetas} and in Proposition \ref{psym}.

$(ii)$ Let $\chi_k$ be as in ~\eqref{kiq}. Then we have
\begin{equation}\label{syst:ul}
\left\{
\begin{aligned}
\mathcal{L}_0 (\chi_kU_s) +T_\gamma (\chi_k\theta_s) &=G_1 \\
\mathcal{L}_0 (\chi_k\theta_s) -T_\gamma (\chi_k U_s) &= G_2 
\end{aligned}
\right.
\end{equation}
where $G_1,G_2$ are given by
\begin{equation*}
\begin{aligned}
G_1   &= \chi_k F_1  + V\cdot (\nabla \chi_k) U_s + [T_\gamma, \chi_k] \theta_s \\
G_2  &= \chi_k F_2   +  V \cdot (\nabla \chi_k) \theta_s - [T_\gamma, \chi_k] U_s.
\end{aligned}
\end{equation*}
We claim that for all $t\in [0,T]$ we have
\begin{multline}\label{est:Gj}
\Vert (G_1(t), G_2(t)) \Vert_{L^2   \times L^2 } 
\leq \mathcal{F}\big(\Vert (\eta(t), B(t),V(t))\Vert_{H^{s_0+\mez}_{ul}  \times H^{s_0}_{ul} \times H^{s_0}_{ul}}\big)\\
\cdot \big(1 + \Vert \eta(t)  \Vert_{H^{s+\mez}_{ul}  } + \Vert  B(t) \Vert_{H^{s  }_{ul} }+\Vert V(t) \Vert_{H^{s }_{ul}}\big). 
\end{multline}
According to Proposition~\ref{psym} this is true for the terms coming from $\chi_k F_j, j =1,2$. 
Now according to~\eqref{def:Uzetas}  we have (for fixed $t$ which is skipped)
\begin{align*}
 \Vert  V\cdot (\nabla \chi_k) U_s \Vert_{L^2 } &\leq \Vert V \Vert_{L^\infty }\Vert U_s \Vert_{L^2_{ul} }\\
 &\leq \Vert V \Vert_{H^{s_0}_{ul} }(\Vert V \Vert_{H^s_{ul}(\xR^d)} + \Vert \nabla \eta  \Vert_{ H^{s_0 +\mez}_{ul} } \Vert B\Vert_{H^s_{ul} })\\
 &\leq \mathcal{F}
 \big(\Vert (\eta ,V )\Vert_{H^{s_0+\mez}_{ul}  \times H^{s_0 }_{ul} }\big)\big(1  + \Vert   B \Vert_{H^{s  }_{ul} }+\Vert V \Vert_{H^{s }_{ul}}\big).
   \end{align*}
The same estimate holds for $ \Vert  V\cdot (\nabla \chi_k) \theta_s \Vert_{L^2 }$.  Eventually we have
\begin{equation*}
\Vert [T_\gamma, \chi_k] \theta_s \Vert_{L^2 } \leq
\mathcal{F}\big(\Vert (\eta , B ,V )\Vert_{H^{s_0+\mez}_{ul}  \times H^{s_0}_{ul} \times H^{s_0}_{ul} }\big)\\
\cdot \big(1  +  \Vert \eta \Vert_{H^{s+\mez}_{ul}}+ \Vert   B \Vert_{H^{s  }_{ul} }+\Vert V \Vert_{H^{s }_{ul}}\big).
\end{equation*}
This proves our claim.  

Now we compute the quantity 
$$(I) = \frac{d}{dt}\Big\{ \Vert \chi_k U_s(t) \Vert^2_{L^2} + \Vert \chi_k \theta_s(t) \Vert^2_{L^2}\Big\}.$$
Using  the equations \eqref{syst:red1}, \eqref{syst:red2}, the point $(i)$,  the fact that 
\begin{equation*}
\begin{aligned}
&\Vert (T_{V(t}\cdot \nabla)^* + T_{V(t)}\cdot \nabla\Vert_{L^2 \to L^2} \leq C M_{s_0}(t)\\
& \Vert T_{\gamma(t)} - (T_{\gamma(t)})^*\Vert_{L^2 \to L^2} \leq C M_{s_0}(t) 
\end{aligned}
 \end{equation*}
and \eqref{est:Gj} we obtain easily $(ii)$.
\end{proof}
\subsection{Back to the original unknowns}
 Recall that
 \begin{equation*} 
 \begin{aligned}
 U_s &= \langle D_x \rangle^s V + T_{\nabla \eta} \langle D_x \rangle^s B,\\
 \theta_s & = T_{\sqrt{\frac{a}{\lambda}}}\langle D_x \rangle^s \nabla \eta.
 \end{aligned}
\end{equation*}
From the estimate in Proposition~\ref{est:L2} 
we would like to recover estimates of the original unknowns $\psi, \eta, V, B$. 
We follow closely~\cite{ABZ3}. The result is as follows.
\begin{prop}\label{original}
Let $s_0>1+ \frac{d}{2}$. For all $s \geq s_0$ one can find    $\mathcal{F}: \xR^+ \to \xR^+$ non decreasing such that  
$$ M_s(T) \leq \mathcal{F}(M_{s_0}(0) + T M_{s_0}(T))  \big\{ M_s(0)  + T  M_s(T)\big\} .$$
\end{prop}
The Proposition will be implied by the following  Lemmas.
\begin{lemm}\label{weaker}
 There exists    $\mathcal{F} : \xR^+ \to \xR^+$ non decreasing such that with $I=[0,T]$ and $H_{ul}^\sigma = H_{ul}^\sigma(\xR^d)$ we have 
 \begin{multline*} 
   \Vert \eta \Vert_{L^\infty(I, H^s )_{ul}} + \Vert (B, V) \Vert_{L^\infty(I, H^{s-\mez}  \times  H^{s-\mez} )_{ul}}  \\ 
 \leq \mathcal{F}(M_{s_0}(0) + \sqrt{T}M_{s_0}(T))\big\{M_s(0) + \sqrt{T} M_s(T)\big\}.
\end{multline*}
\end{lemm}
\begin{proof}
Set $\mathcal{L}_1 = \partial_t + V \cdot \nabla$. According to Proposition~\ref{new:eq} we have
 $\mathcal{L}_1 B = a-g, \quad \mathcal{L}_1 V = -a \nabla \eta$  and
 from the definition of $V,B$ and the equations we have $\mathcal{L}_1 \eta = B$. 
 Then the lemma follows from Lemma~\ref{transport2} with $\mu =s, \mu = s-\mez$ 
 and Proposition~\ref{prop:taylor}.
\end{proof}
  \begin{lemm}\label{end:proof}
 Let $s_0>1+ \frac{d}{2}$. For $s \geq s_0$ one can find  $\mathcal{F}: \xR^+ \to \xR^+$ non decreasing such that, with $H^\sigma_{ul}= H^\sigma_{ul}(\xR^d), $ we have
 \begin{equation*}
 \begin{aligned}
 &(i)  \quad \Vert \eta \vert_{L^\infty(I, H^{s+\mez} )_{ul}} \leq  \mathcal{F}(M_{s_0}(0)+\sqrt{T}M_{s_0}(T))\big\{M_s(0) + \sqrt{T} M_s(T)\big\},\\ 
 &(ii) \quad \Vert (V,B) \Vert_{L^\infty(I, H^{s} )_{ul}} \leq  \mathcal{F}(M_{s_0}(0)+\sqrt{T}M_{s_0}(T))\big\{M_s(0) + \sqrt{T} M_s(T)\big\},\\
  &(iii) \quad \Vert \psi \Vert_{L^\infty(I, H^{s+\mez} )_{ul}} \leq  \mathcal{F}(M_{s_0}(0)+\sqrt{T}M_{s_0}(T))\big\{M_s(0) + \sqrt{T} M_s(T)\big\}.  
 \end{aligned}
 \end{equation*}
 \end{lemm}
\begin{proof}
$(i)$ By Lemma \ref{weaker} it is sufficient to bound $A =  \Vert \nabla \eta \Vert_{L^\infty(I, H^{s-\mez}_{ul})}$. Recall that $q = \sqrt{\frac{a}{\lambda}},\,  \theta_s = T_q \zeta_s,$ and $  \zeta_s = \langle D \rangle^s \nabla \eta$. By Theorem \ref{calc:symb} $(ii)$ we can write $\zeta_s = T_\frac{1}{q} T_q \zeta_s + R \zeta_s$ where $\Vert R\Vert_{L^\infty(I,H^\mu )_{ul} \to L^\infty(I,H^{\mu+\mez} )_{ul}} \leq C(\Vert a \Vert_{L^\infty(I,H^{s_0-\mez} )_{ul}} + \Vert \eta \Vert_{L^\infty(I,H^{s_0+Ê\mez} )_{ul}})$. Then we have 
$$A = \Vert \langle D \rangle^{-\mez} \zeta_s \Vert_{L^\infty(I, L^2 )_{ul}} \leq  \Vert \langle D \rangle^{-\mez} T_\frac{1}{q}\theta_s \Vert_{L^\infty(I, L^2 )_{ul}} + \Vert \langle D \rangle^{-\mez} R \zeta_s \Vert_{L^\infty(I, L^2 )_{ul}}. $$
Using  Theorem \ref{calc:symb}, $(i)$,   the above estimate on the norm of $R$ with $\mu = -1,$  Lemma  \ref{weaker}Ê and Proposition \ref{prop:taylor} we deduce that
$$ A \leq \mathcal{F}(M_{s_0}(0)+\sqrt{T}M_{s_0}(T)) \big\{ \Vert \theta_s \Vert_{L^\infty(I,L^2 )_{ul}} + \Vert \eta \Vert_{L^\infty(I, H^s )_{ul}}\big\}.$$
Then the conclusion follows from Proposition \ref{est:L2} and Lemma \ref{weaker}.

$(ii)$ Recall that $U = V + T_\zeta B$. The commutator $[\langle D \rangle^s, T_\zeta]$ is of order $s-\mez$ which norm from $L^\infty(I,L^2 )_{ul}$ to  $L^\infty(I, L^2 )_{ul}$ is bounded by $C \Vert \eta_{L^\infty(I,C_*^{\frac{3}{2} })}$ thus by $C' \Vert \eta \Vert_{L^\infty(I, H^{s_0 + \mez} )_{ul}}$. Therefore we deduce from Proposition \ref{est:L2} and Lemma \ref{weaker} that 
\begin{equation}\label{est:U}
  \Vert U \Vert_{L^\infty(I, H^s )_{ul}} \leq \mathcal{F}(M_{s_0}(0)+ \sqrt{T} M_{s_0}(T))  (M_s(0)  + \sqrt{T}  M_s(T)).
  \end{equation}
   Now by Lemma \ref{GB=divV} we have 
 \begin{equation*}
 \begin{aligned}
 \text{div } U &= \text{div }V +T_{\text{div }\zeta} B + T_\zeta \cdot \nabla B = -G(\eta)B +T_{i\zeta \cdot \xi + \text{ div }\zeta}B + \gamma\\
 &= - T_\lambda B + R(\eta) B + T_{i\zeta \cdot \xi + \text{ div }\zeta}B + \gamma = T_eB + T_{\text{div }\zeta }B + \gamma 
 \end{aligned}
 \end{equation*}
 where $e = -\lambda +i \zeta \cdot \xi$. Writing $B = T_{\frac{1}{e}}T_eB + (I-  T_{\frac{1}{e}}T_e)B$ we obtain
 $$B = T_{\frac{1}{e}} \text{div }U - T_{\frac{1}{e}}\gamma + SB.$$
 Then   using \eqref{est:U}, Lemma \ref{GB=divV} we obtain the desired estimate on $B$ and since $V = U- T_\zeta B$ the estimate on $V$ follows as well. 
 
$(iii)$\quad We have $\nabla \psi = V + B \nabla \eta$. Since the $L^\infty(I, H^{s-\mez} )_{ul}$ norm of $(\nabla \eta, V,B)$ has been already estimated, it remains to bound $\Vert \psi \Vert_{L^\infty(I, L^2 )_{ul}}$. Now from \eqref{ww} and \eqref{BV} by a simple computation we see that
$$(\partial_t + V \cdot \nabla)\psi = -g \eta + \mez \vert V \vert^2  + \mez B^2.$$
Then the conclusion follows from Lemma \ref{transport2} with $ \mu=s $ and $(ii)$.
 \end{proof}
 
\section{Contraction and well posedness}\label{sec.6}
\subsection{Contraction}
In this section we shall prove estimates on the difference 
of two solutions of the system described in~\eqref{eq:B},~\eqref{eq:V},~\eqref{eq:zeta} 
which will prove the  uniqueness and also enter in the proof by contraction of the existence. 
Let $(\eta_j, \psi_j, V_j,B_j), j=1,2$ be two solutions of the system
\begin{equation}\label{eq.7.1}
\left\{
\begin{aligned}
&(\partial_t + V_j \cdot \nabla_x)B_j  = a_j-g, \\
&(\partial_t + V_j \cdot \nabla_x)V_j +a_j \zeta_j  =0, \\
&(\partial_t + V_j \cdot \nabla_x)\zeta_j  = G(\eta_j)V_j + \zeta_j G(\eta_j)B_j +R_j, 
\end{aligned}
\right.
\end{equation}
on $[0,T_0]$ such that with $H^\sigma_{ul} = H^\sigma_{ul}(\xR^d)$ we have 
$$M_j = \sup_{t\in [0,T_0]} \Vert (\eta_j(t), \psi_j(t), V_j(t),B_j(t))\Vert_{H^{s+\mez}_{ul} \times H^{s+\mez}_{ul} \times H^{s }_{ul} \times H^{s }_{ul}} <+\infty.$$
We also assume the Taylor sign condition satisfied that is that there exist $c_j>0$ for $j=1,2$ such that  $a_j(t,x) \geq c_j $ for all $ t\in [0,T_0]$.
We set 
\begin{equation*}
\begin{aligned}
  &\eta = \eta_1-\eta_2, \quad \psi = \psi_1-\psi_2, \quad V= V_1 - V_2, \quad B = B_1 - B_2,\\
  & N(T) = \sup_{t\in [0,T]} \Vert (\eta(t),  \psi(t) , V(t) ,B(t) )\Vert_{H^{s-\mez}_{ul} \times H^{s-\mez}_{ul} \times H^{s-1 }_{ul} \times H^{s-1 }_{ul}}.
  \end{aligned}
  \end{equation*}
\begin{theo}\label{th.lipschitz}
Let~$(\eta_j,\psi_j)$,~$j=1,2$, be two solutions of~\eqref{ww} such that 
$$
(\eta_j,\psi_j,V_j,B_j)\in 
C^0([0,T_0],H_{ul}^{s+\mez}\times H_{ul}^{s+\mez}\times H_{ul}^s\times H_{ul}^s),
$$
for some fixed~$T_0>0$,~$d\ge 1$ and~$s>1+\frac{d}{2}$. We also assume 
that the condition \eqref{condition} holds for 
$0\le t\le T_0$ and that there exists a positive constant~$c$ 
such that for all~$0\le t\le T_0$ and for all~$x\in {\mathbf{R}}^d$, we have 
$\ma_j(t,x)\ge c$ for~$j=1,2$,~$t\in [0,T]$.
Set
$$
M_j\defn \sup_{t\in [0,T]} \lA (\eta_j,\psi_j,V_j,B_j)(t)\rA_{H_{ul}^{s+\mez}\times H_{ul}^{s+\mez} \times H_{ul}^s\times H_{ul}^s}, 
$$
$$
\deta\defn \eta_1-\eta_2, \quad \psi=\psi_1-\psi_2,\quad \dV\defn V_1-V_2, \quad \dB=\B_1-\B_2.
$$
Then we have 
\begin{multline}\label{eq.lipschitz}
\|(\eta, \psi, V, B)\| _{L^\infty((0,T), H^{s-\mez}\times H ^{s-\mez}\times H ^{s-1}\times H ^{s-1})_{ul}} \\
\leq \mathcal{K} (M_1, M_2) \|(\eta, \psi, V, B) 
\mid_{t=0}\| _{H_{ul}^{s-\mez}\times H_{ul}^{s-\mez}\times H_{ul}^{s-1}\times H_{ul}^{s-1}}.
\end{multline}
\end{theo} 
Let us recall that 
\begin{equation}\label{BVzeta0}
\left\{
\begin{aligned}
&(\partial_{t}+V_j\cdot\partialx)\B_j=\ma_j-g,\\
&(\partial_t+V_j\cdot\partialx)V_j+\ma_j\zeta_j=0,\\
&(\partial_{t}+V_j\cdot\partialx)\zeta_j=G(\eta_j)V_j+ \zeta_j G(\eta_j)\B_j+\gamma_j,\quad \zeta_j=\nabla \eta_j,
\end{aligned}
\right.
\end{equation}
where~$\gamma_j$ is the remainder term given by \eqref{eq:zeta}. Let
$$
N(T)\defn \sup_{t\in [0,T]} \lA (\eta,\psi,V,B)(t)\rA_{H_{ul}^{s-\mez}\times H_{ul}^{s-\mez}\times H_{ul}^{s-1}\times H_{ul}^{s-1}}.
$$
Our goal is to prove an 
estimate of the form 
\begin{equation}\label{NM1M2}
N(T)\le \mathcal{K}(M_1, M_2) N(0) +T \, \mathcal{K}(M_1,M_2) N (T),
\end{equation}
for some non-decreasing function~$\mathcal{K}$ depending only on~$s$ and~$d$. Then, by choosing~$T$ small enough, this implies~$N(T)\leq 2\mathcal{K}(M_1, M_2) N(0)~$ for~$T_1$ smaller than the minimum of $T_0$ and $1/2 \mathcal{K}(M_1,M_2)$, and iterating the estimate between~$[T_1, 2T_1]$,\dots,~$[T- T_1, T_1]$ implies Theorem~\ref{th.lipschitz}. 

\begin{rema}
Notice that we prove a Lipschitz property in weak norms. 
This is a general fact related to the fact that the flow map 
of a quasi-linear equation is not expected to be Lipschitz in the highest norms 
(this means that one does not expect to control the difference~$(\eta,\psi,V,B)$ in 
$L^\infty([0,T_0],H^{s+\mez}\times H^{s+\mez}\times H^{s}\times H^{s})_{ul}$). 
\end{rema}
The proof of Theorem~\ref{th.lipschitz} follows the same lines as the proof of the similar result~\cite[Theorem 5.1]{ABZ3}. Il follows 4 steps: first we prove a Lipschitz estimate for the Dirichlet-Neumann operator. Then we paralinearize the system satisfied by $(\eta, \psi, V, B)$, symmetrize this system, estimate the good unknowns of the symmetrized system and finally estimate $(\eta, \psi, V, B)$. The Lipschitz estimate of the Dirichlet-Neumann operator is the crucial one and we shall give some details.  Having established  the paradifferential calculus in uniformly local spaces, the other steps are identical {\em mutatis mutandi} as in~\cite{ABZ3} and we shall skip the proofs.
 
 \subsection{Contraction for the Dirichlet-Neumann operator.}
In this section the time being fixed  we will skip it.
\begin{lemm}\label{lip:G}
Assume $s>1+\frac{d}{2}$. Then there exists   $\mathcal{F}:\xR^+ \to \xR^+ $ non decreasing  such that for all  $\eta_1,\eta_2 \in H^{s+\mez}_{ul} $ and all $f\in H^s_{ul} $ we have
\begin{equation}\label{Glip} 
\Vert (G(\eta_1)- G(\eta_2))f \Vert_{H^{s-\frac{3}{2}}_{ul} } \leq \mathcal{F}( \Vert (\eta_1,\eta_2,f)\Vert_{H^{s+\mez}_{ul} \times H^{s+\mez}_{ul} \times H^{s }_{ul}})  \Vert \eta_1-\eta_2 \Vert_{H^{s-\mez}_{ul}}
\end{equation}
where $H_{ul}^\sigma = H_{ul}^\sigma (\xR^d)$.
 \end{lemm}
\begin{proof}
We follow closely~\cite{ABZ3}. 
As in~\eqref{rho}, ~\eqref{alpha} we introduce   $\rho_j$, $\widetilde{u}_j$, $\alpha_j$, $\beta_j$, $\gamma_j$ 
for $j=1,2$. Then if $\widetilde{u}_j \arrowvert_{z=0} = f$ we have
\begin{equation}\label{encoreG}
 G(\eta_j)f = \big(\frac{1+ \vert \nabla_x \rho_j \vert^2}{\partial_z \rho_j}\partial_z \widetilde{u}_j - \nabla_x \rho_j \cdot \nabla_x \widetilde{u}_j \big) \arrowvert_{z=0}.
 \end{equation}
We set  $ \widetilde{u} = \widetilde{u}_1  - \widetilde{u}_2$. Then 
$$ (\partial_z^2 + \alpha_1 \Delta_x + \beta_1\cdot \nabla_x \partial_z - \gamma_1 \partial_z) \widetilde{u} := F$$ 
where 
$$F= \big\{(\alpha_2 - \alpha_1)\Delta_x + (\beta_2- \beta_1)\cdot \nabla_x \partial_z -(\gamma_2 - \gamma_1) \partial_z \big\}\widetilde{u}_2.$$
Since $s>1+\frac{d}{2}$, Lemma~\ref{Houl} with $s_0 = s-2, s_1 = s-2, s_2 = s-1, p=2$ gives (with $J = (z_0,0)$ and $H^\mu_{ul} = H^\mu_{ul}(\xR^d)$)
\begin{equation*}
\begin{aligned}
  \Vert F \Vert_{L^2(J, H^{s-2} )_{ul}} \leq \, 
  &K  \{ \Vert \alpha_2-\alpha_1 \Vert_{L^2(J, H^{s-1})_{ul}} 
  \Vert \Delta \widetilde{u}_2 \Vert_{L^\infty(J, H^{s-2})_{ul}}\\
 & +\Vert \beta_2-\beta_1 \Vert_{L^2(J, H^{s-1})_{ul}} 
 \Vert \nabla \partial_z \widetilde{u}_2 \Vert_{L^\infty(J, H^{s-2} )_{ul}}\\
 &+\Vert \gamma_2-\gamma_1 \Vert_{L^2(J, H^{s-2} )_{ul}}
  \Vert   \partial_z \widetilde{u}_2 \Vert_{L^\infty(J, H^{s-1} )_{ul}}\}.
 \end{aligned}
 \end{equation*}
Using ~\eqref{alpha} and Lemma~\ref{Houl} we can find a non decreasing function $\mathcal{F}: \xR^+ \to \xR^+$ such that 
\begin{align*}
  \Vert \alpha_2-\alpha_1 \Vert_{L^2(J, H^{s-1} )_{ul}} + &\Vert \beta_2-\beta_1 \Vert_{L^2(J, H^{s-1} )_{ul}}+ \Vert \gamma_2-\gamma_1 \Vert_{L^2(J, H^{s-2} )_{ul}}  \\
 & \leq \mathcal{F} ( \Vert (\eta_1, \eta_2) \Vert_{H^{s+\mez}_{ul}  \times H^{s+\mez}_{ul} }) \Vert \eta_1 - \eta_2 \Vert_{H^{s-\mez}_{ul} }.
 \end{align*}
On the other hand by Theorem~\ref{regell} with $\sigma = s-1$ we have
$$
\Vert \nabla_{x,z} \widetilde{u}_2 \Vert_{L^\infty((z_0,0) , H^{s-1})_{ul} } \leq 
\mathcal{F}\big(\Vert \eta_2 \Vert_{H^{s+\mez}_{ul}}\big) \Vert f \Vert_{H^s_{ul}}.
$$ 
Combining these estimates we obtain eventually
$$  \Vert F \Vert_{L^2(J, H^{s-2})_{ul}} \leq 
\mathcal{F} \Big( \Vert (\eta_1, \eta_2,f) \Vert_{H^{s+\mez}_{ul} \times H^{s+\mez}_{ul} \times H^s_{ul}} \Big) \Vert \eta_1 - \eta_2 \Vert_{H^{s-\mez}_{ul} }.$$
Since $\tilde{u}$ vanishes at $z=0$ Theorem~\ref{regell} with $\sigma = s-\frac{3}{2}$ gives
$$
\Vert \nabla_{x,z} \widetilde{u}  \Vert_{C^0((z_0,0) , H^{s-\frac{3}{2}})_{ul}} \leq 
\mathcal{F}\big( \Vert (\eta_1, \eta_2,f) \Vert_{H^{s+\mez}_{ul}  \times H^{s+\mez}_{ul}\times H^s_{ul}}\big) \Vert \eta_1 - \eta_2 \Vert_{H^{s-\mez}_{ul} }.
$$
Using~\eqref{encoreG} and Proposition~\ref{Ho} $(i)$ we obtain~\eqref{Glip}.
\end{proof}

\subsection{Paralinearization of the equations}
Notice that it is enough to estimate 
$\deta,\dB,\dV$. Indeed, since~$V_j =\nabla \psi_j-B_j\nabla \eta_j$, one can estimate the~$L^\infty([0,T],H  ^{s-\frac{3}{2}})_{ul}$-norm of 
$\nabla\psi$
from the identity
\begin{align*}
\nabla \psi = \dV +\dB \nabla \eta_1 +\B_2 \nabla \deta.
\end{align*}
\begin{lemm}[\protect{\cite[Lemma~5.6]{ABZ3}}]\label{lemm:symmind}
The differences~$\zeta,B,V$ satisfy a system of the form
\begin{equation}\label{syst:delta}
\left\{
\begin{aligned}
&(\partial_t +V_1\cdot\partialx)(\dV+\zeta_1 B)+\ma_2 \dzeta =f_1 ,\\
&(\partial_{t}+V_2\cdot\partialx)\dzeta-G(\eta_1)\dV- \zeta_1 G(\eta_1)\dB=f_2,
\end{aligned}
\right.
\end{equation}
for some remainders such that
$$
\lA (f_1,f_2)\rA_{L^{\infty}([0,T],H ^{s-1}\times H ^{s-\tdm})_{ul}} \le \mathcal{K}(M_1,M_2)N(T).
$$
\end{lemm}
\subsection{Estimates for the good unknown}
In this section we introduce the good-unknown of Alinhac in~\cite{AM,ABZ1,Alipara,Ali} and symmetrize the system.
Let~$I = [0,T]$.
\begin{lemm}[\protect{\cite[Lemma~5.7]{ABZ3}}] \label{symetr}
Set
$$
\ell \defn \sqrt{ \lambda_1 \ma_2},\quad \varphi\defn T_{\sqrt{\lambda_1}}(\dV+\zeta_1 B),\quad \vartheta\defn T_{\sqrt{\ma_2}} \dzeta.
$$
Then 
\begin{align}
&(\partial_t +T_{V_1}\cdot\partialx)\varphi + T_{\ell}\vartheta =g_1 ,\label{syst:delta5}\\
&(\partial_{t}+T_{V_2}\cdot\partialx)\vartheta -T_{\ell }\varphi =g_2,\label{syst:delta5bis}
\end{align}
where
$$
\lA (g_1,g_2)\rA_{L^{\infty}(I,H ^{s-\tdm}\times H ^{s-\tdm})_{ul}} \le \mathcal{K}(M_1,M_2)N(T).
$$
\end{lemm}
Once this symmetrization has been performed, simple energy estimates allow to prove 
\begin{lemm}[\protect{\cite[Lemma~5.8]{ABZ3}}]\label{L63.a}
Let
$$
N'(T)\defn  \sup_{t\in I} \big\{\lA \vartheta(t)\rA_{H^{s-\tdm}}
+ \lA \varphi(t)\rA_{H^{s-\tdm}}\big \}.
$$
We have
\begin{equation}\label{esti:N'}
N'(T)\le \mathcal{K}(M_1,M_2)\bigl( N(0)+TN(T)\bigr).
\end{equation}
\end{lemm}
\subsection{Back to the original unknowns}
From the estimates in Lemma~\ref{L63.a}, it is fairly easy to recover estimates for $\eta$.
\begin{lemm}[\protect{\cite[Lemma~5.9]{ABZ3}}]\label{lem.N}
\begin{equation}\label{p64.10}
\| \eta \|_{L^\infty(I; H^{s-\mez})_{ul}}
\le  \mathcal{K}(M_1, M_2) \{N(0)+  T N(T)\}.
\end{equation}
\end{lemm}
We now estimate $(V,B)$.
\begin{prop}[\protect{\cite[Proposition~5.10]{ABZ3}}]\label{lem.NN}
\begin{equation}\label{p64.10bis}
\| (V,B) \|_{L^\infty(I, H^{s-1}\times H^{s-1})_{ul}}
\le  \mathcal{K}(M_1, M_2) \big\{N(0)+  T N(T)\big\}.
\end{equation}
\end{prop}

The proof will require several preliminary lemmas. 
We begin by noticing that it is enough to estimate $B$. Indeed, if
$$
\| B \|_{L^\infty(I, H ^{s-1})_{ul}}
\le  \mathcal{K}(M_1, M_2) \big\{N(0)+  T N(T)\big\}.
$$
then, the estimate of $\varphi$ in~\eqref{esti:N'} above allows to recover an estimate for $V+ \zeta_1 B$ (by applying  $T_{\sqrt{ \lambda_1} ^{-1}}$), which in turn implies the estimate for  $V$.

 Let~$v=\widetilde{\phi}_1-\widetilde{\phi}_2$, where $\widetilde{\phi}_j$ 
 is the harmonic extension in $\tilde\Omega$ of the function  $\psi_j$ and set 
$$ b_2\defn\frac{\partial_z \widetilde{\phi}_2}{\partial_z \rho_2},\quad 
w= v - T_{b_2}\rho.$$ We have 
 \begin{equation}\label{trace:w}
 w\arrowvert_{z=0} = \psi - T_{B_2} \eta.
\end{equation}
We first state the following result.
  \begin{lemm}[\protect{\cite[Lemma~5.11]{ABZ3}}]\label{est:trace}
We have 
\begin{equation}
\| \psi - T_{B_2} \eta \|_{L^\infty(I, H ^{s})_{ul}}
\le  \mathcal{K}(M_1, M_2)\big\{N(0)+  T N(T)\big\}
\end{equation}
\end{lemm}
We next relate $w$, $\rho$ and $B$.
\begin{lemm}[\protect{\cite[Lemma~5.12]{ABZ3}}]We have
$$
B=\Bigl[ \frac{1}{\partial_z\rho_1}\Bigl( \partial_z w -(b_2-T_{b_2})\partial_z\rho
+T_{\partial_z b_2}\rho\Bigr)\Bigr]\Big\arrowvert_{z=0}.
$$
\end{lemm}
\begin{lemm}[\protect{\cite[Lemma~5.13]{ABZ3}}]\label{esti:b2012}
Recall that $ b_2\defn\frac{\partial_z \widetilde{\phi}_2}{\partial_z \rho_2}$. For $k=0,1,2$, we have
$$
\lA \partial_z^k b_2\rA_{C^0([-1,0], L^\infty(I, H^{s-\mez-k}) _{ul})}\le C \lA \psi_2\rA_{H_{ul}^{s+\mez}}.
$$
for some constant $C$ depending only on $\lA \eta_2\rA_{H_{ul}^{s+\mez}}$. 
\end{lemm}

Notice that $\eta$ and hence $\rho$ are estimated in $L^\infty(I; H^{s-\mez})$ (see~\eqref{p64.10}). 
To complete the proof of the Proposition \ref{lem.NN}, it remains only to estimate $\partial_z w\arrowvert_{z=0}$ 
in $L^\infty(I,H_{ul}^{s-1})$.  
 \begin{lemm}[\protect{\cite[Lemma~5.14]{ABZ3}}]\label{reg:w}
 For $t\in [0,T]$ we have   
 \begin{equation}
\| \nabla_{x,z} w  \|_{C^0([-1,0],   H^{s-1})_{ul} }
\le  \mathcal{K}(M_1, M_2) \big\{N(0)+  T N(T)\big\}.
\end{equation}
\end{lemm}
 \subsection{Well posedness}
The proof goes as follows. In a first step we prove the main theorem for very smooth data, using a parabolic  regularization. Then, when the data are rough, we regularize them, thus obtaining a sequence of solutions living on an interval depending on a small parameter $\eps$. In a second step, using the tame estimates proved in Proposition \ref{original}, we show that this sequence  exists on a fixed interval. In the last step, using the results stated 
in section 7, 
we prove that it is a Cauchy sequence and we conclude. Let us notice that most of this work has been already done in \cite{ABZ3} in the case of the classical Sobolev spaces. Therefore we will only sketch here the main points.

\subsection{Parabolic regularization} We assume first that $(\eta_0, \psi_0) \in H^{s }_{ul} \times H^{s}_{ul}$ for $s\geq n_0 + \frac{d}{2}, n_0$ large enough. and we consider for $\eps>0$ the problem 
\begin{equation}\label{wweps}
\left\{
\begin{aligned}
&\partial_t \eta  = G(\eta) \psi + \eps \Delta_x\eta,\\
&\partial_t \psi  = - \mez \vert \nabla_x \psi \vert^2 + \mez \frac{(\nabla_x \eta \cdot \nabla_x \psi + G(\eta)\psi)^2}{1+ \vert \nabla_x \eta \vert^2} - g \eta  + \eps \Delta_x \psi \\
&(\eta, \psi)\arrowvert_{t=0}  = (\eta_0, \psi_0).
\end{aligned}
\right.
\end{equation}
Setting $U = (\eta, \psi)$ we can rewrite this problem as
\begin{equation}\label{parabolique}
 U(t) = e^{\eps t \Delta_x}U_0 + \int_0^t e^{\eps(t-\tau) \Delta_x} \big[ \mathcal{A}(U(\tau))\big] \, d\tau. 
 \end{equation}
We set $I =[0,T]$ and we introduce the space 
$$E_s =L^\infty(I, H^s)_{ul} \cap L^2(I, H^{s+1})_{ul}. $$
According to Lemma \ref{lemmsmooth} we have $\Vert e^{\eps t \Delta_x}U_0 \Vert_{E_s} \leq C_\eps \Vert U_0\Vert_{H^s_{ul}}: = R$. Then using the estimates 
\begin{equation}
  \begin{aligned}
\Vert \mathcal{A}(U) \Vert_{L^2(I, H^s)_{ul}} &\leq \mathcal{F}(\Vert U \Vert_{L^\infty(I,H^s)_{ul}}) \Vert U \Vert_{L^2(I,H^{s+1})_{ul}} \\
\Vert \mathcal{A}(U_1) -  \mathcal{A}(U_2)\Vert_{L^2(I, H^s)_{ul}} &\leq \mathcal{F}(\Vert (U_1,U_2) \Vert_{L^\infty(I,H^s \times H^s)_{ul}}) \Vert U_1 -U_2 \Vert_{L^2(I,H^{s+1})_{ul}} 
 \end{aligned}
  \end{equation}
we can show that, if $T= T_\eps$ is small enough,  the right hand side of \eqref{parabolique} maps the ball of radius $2R$ in $E_s$ into itself and is contracting. By the Banach principle  the equation \eqref{parabolique} has a maximal solution on $[0, T_\eps^*)$.  Moreover if $T_\eps<+\infty $ then
\begin{equation}\label{explosion}
\lim_{t\to T_\eps^*} \Vert(\eta,\psi)(t)\Vert_{H^s_{ul}\times H^s_{ul}} = + \infty.
\end{equation}
Now with this large $s$ we set 
\begin{equation*}
M^\eps_s(T) = \sup_{t\in [0,T]}\Vert (\eta^\eps, \psi^\eps, V^\eps,B^\eps)(t) \Vert_{H^s_{ul}\times H^{s }_{ul} \times H^{s-1}_{ul} \times H^{s-1}_{ul}}.
\end{equation*}
Using the same computations as in \cite{ABZ3} and the method of proof of Proposition \ref{original} (but in an easier way since here $s$ is large) we deduce that one can find $\mathcal{F}: \xR^+ \to \xR^+$ strictly increasing such that 
$$ M^\eps_s(T) \leq \mathcal{F}(M^\eps_s(0) + \sqrt{T} M^\eps_s(T)).$$
Since $M^\eps_s(0) = M_s(0)$ does not depend on $\eps$, this will imply that there exists $T_0>0$ independent of $\eps$ such that $M^\eps_s(T) \leq \mathcal{F}(2M_s(0))$ for $T\in [0,T_0]$. Using this uniform bound on this fixed interval and the arguments of \cite{ABZ3} we can pass to the limit in the equations~\ref{wweps} to obtain a solution $(\eta, \psi)$ of the water wave system.

\subsection{Regularizing the data, {\em{a priori}} estimates}\label{sec.8.2}. Assume 
$ (\eta_0, \psi_0,V_0,B_0)$ belongs to $ H^{s_0 +\mez }_{ul}\times H^{s_0 +\mez }_{ul} \times H^{s_0}_{ul} \times H^{s_0}_{ul}$ where $s_0>1+\frac{d}{2}$. Let $j \in C_0^\infty(\xR^d), j(\xi) =1$ when $\vert\xi \vert \leq 1$. We regularize the data in setting  $f_0^\eps = j(\eps D)f_0 $ if $f_0$ is one of them. Then the regularized data belong to $H^{s}_{ul} $ for $s$ large.  Therefore we can apply Step 1. to get a maximal solution $(\eta_\eps, \psi_\eps,V_\eps,B_\eps)$  of the water wave system,  on an interval $[0,T_\eps^*),$ wich is very regular.  Moreover we know that if $T_\eps^* <+ \infty$ then
\begin{equation}\label{explosion2}
\lim_{T\to T_\eps^*} M^\eps_s(T) = +\infty.
\end{equation}
  We first apply Proposition \ref{original} with $s=s_0$ and we obtain
  $$ M^\eps_{s_0}(T) \leq \mathcal{F}(M^\eps_{s_0}(0) + \sqrt{T} M^\eps_{s_0}(T)).$$
Since there exists $A_0>0,$ independent of $\eps,$ such that $M^\eps_{s_0}(0) \leq A_0,$  for all $\eps>0$ small, we deduce that one can find $T_0 >0$ independent of $\eps$ such that $ M^\eps_{s_0}(T) \leq \mathcal{F}(2A_0)$ for all $T\leq  \text{min} (T_0, T_\eps^*)$.  We apply again Proposition \ref{original} with $s$ large and we get
 $$ M^\eps_s(T) \leq \mathcal{F}(A_0 + \sqrt{T_0}\mathcal{F}(2A_0))(  M^\eps_s(0)  +  \sqrt{T}M^\eps_s(T)).$$ 
 Let $T_1>0$ be such that  $ \sqrt{T_1} \mathcal{F}(A_0 + \sqrt{T_0}\mathcal{F}(2A_0)) \leq \mez$. Then 
 $$M^\eps_s(T) \leq 2\mathcal{F}(A_0 + \sqrt{T_0}\mathcal{F}(2A_0))M^\eps_s(0),  \forall 0<T \leq \min(T_1,T_\eps^*).$$
Using \eqref{explosion2} we deduce  that $T_\eps ^* \geq T_1$ for all $\eps>0$ small. This shows that our solution  $(\eta_\eps, \psi_\eps,V_\eps,B_\eps)$ exists on a fixed interval $[0,T_1]$. Moreover, as seen above,  $M^\eps_{s_0}(T)$ is uniformly bounded on this interval.
 
\subsection{Passing to the limit} 
According to Theorem~\ref{th.lipschitz}, $(\eta_\eps, \psi_ \eps, V_\eps, B_\eps)$ which is, according to Section~\ref{sec.8.2},  bounded in 
$$L^\infty((0,T); H^{s_0+\mez}_{ul}\times H^{s_0+\mez}_{ul}\times H^{s_0}_{ul} \times H^{s_0}_{ul}),$$
is convergent in 
 $$L^\infty((0,T); H^{s_0-\mez}_{ul}\times H^{s_0-\mez}_{ul}\times H^{s_0-1}_{ul} \times H^{s_0-1}_{ul}),
 $$
 and hence also for any $\delta >0$ in 
 $$
 L^\infty((0,T); H^{s_0+\mez -\delta}_{ul}\times H^{s_0+\mez-\delta}_{ul}
 \times H^{s_0-\delta}_{ul} \times H^{s_0-\delta}_{ul}).
 $$
To get the existence of solutions, it remains to pass to the limit in the equations (the uniqueness follows once again from Theorem~\ref{th.lipschitz}). For this step, we rewrite the system~\eqref{ww},~\eqref{BV} as 

\begin{equation}\label{wwbis}
\left\{
\begin{aligned}
\partial_t \eta_\eps &= G(\eta_\eps) \psi_\eps,\\
\partial_t \psi_\eps + V_\eps \cdot \nabla \psi _\eps&= \mez( V_\eps^2 + B_\eps^2) - g \eta_\eps.\\
B_\eps &=  \frac{ \nabla_x \eta_\eps \cdot \nabla_x \psi_\eps + G(\eta_\eps)\psi_\eps} {1+ \vert \nabla_x \eta_\eps \vert^2},\\
V_\eps&=\nabla_x \psi_\eps - B_\eps \nabla_x \eta_\eps.
\end{aligned}
\right.
 \end{equation}
 Choosing $\delta >0$ such that $s- \delta- \mez  > \frac d 2 $ ( so that $H^{s- \delta- \mez}$  is an algebra), we deduce that  
\begin{equation}
\begin{aligned} 
\partial_t \eta_\eps &\rightharpoonup \partial_t \eta \text{ in } \mathcal{D} ' ( (0,T) \times \xR^d)\\
\partial_t \psi_\eps &\rightharpoonup \partial_t \psi \text{ in } \mathcal{D} ' ( (0,T) \times \xR^d)\\
V_ \eps \cdot \nabla \psi_ \eps &\rightarrow V \cdot \nabla \psi \text{ in } L^\infty((0,T); H^{s- \delta- \mez }_{ul} )\\
V^2_\eps + B^2_\eps &\rightarrow V^2 + B^2 \text { in } L^\infty((0,T); H^{s- \delta }_{ul} )\\
\nabla_x \eta_ \eps \cdot \nabla_x \psi_\eps &\rightarrow \nabla_x \eta \cdot \nabla_x \psi \text{ in } L^\infty((0,T); H^{s- \delta- \mez }_{ul} ) \subset L^\infty((0,T); L^2_{ul})\\
|\nabla_x \eta_\eps | ^2&\rightarrow |\nabla_x \eta | ^2 \text{ in } L^\infty((0,T); H^{s- \delta- \mez }_{ul} )\subset L^\infty((0,T); C^0\cap L^\infty ( \xR^d))\\
\end{aligned}
\end{equation}
On the other hand, according to Lemma~\ref{lip:G}, we get 
$$
G( \eta_\eps) \psi_\eps- G( \eta) \psi= G( \eta_\eps) ( \psi_\eps - \psi) + (G( \eta_\eps) - G(\eta)) \psi \rightarrow 0 ,
$$
in
$$
L^\infty((0,T); H^{s- \tdm- \delta } _{ul} )\subset L^\infty((0,T); L^2 _{ul} ),
$$
which allows to pass to he limit in~\eqref{wwbis} and show that the same system of equations is satisfied by $(\eta, \psi, V, B)$ in $\mathcal{D} ' ( (0,T) \times \xR^d)$.
\subsection{Continuity in time}We now prove that $(\eta, \psi, V,B)$ is continuous in time with values in $H^{s_0+\mez- \delta}_{ul}\times H^{s_0+\mez- \delta}_{ul}\times H^{s_0- \delta}_{ul} \times H^{s_0- \delta}_{ul}$. From the equation, and product rules, its time derivative is clearly in 
$$L^\infty((0,T);  H^{s_0-\mez}_{ul}\times H^{s_0-\mez}_{ul}\times H^{s_0- \tdm}_{ul} \times H^{s_0- \tdm}_{ul})
$$ and consequently (interpolating with the a priori estimate), for any $\delta >0$,
\begin{equation}\label{eq.contfaible}( \eta, \psi, V, B) \in C^0((0,T);  H^{s_0+\mez- \delta}_{ul}\times H^{s_0+\mez- \delta }_{ul}\times H^{s_0-\delta}_{ul} \times H^{s_0- \delta}_{ul}).
\end{equation} 

 
\section{The canal}\label{sec.7}
We consider now the case of a canal having  vertical walls near the free surface or the case of  a rectangular basin. 

The propagation of waves whose crests are orthogonal to the walls is one 
of the main motivation for the analysis of 2D waves. It was historically at the heart 
of the analysis of water waves. The study of the propagation of three-dimensional 
water waves for the linearized equations goes back to Boussinesq (see~\cite{Boussinesq}).  
However, there are no existence results for the nonlinear equations in the 
general case where the waves can be reflected on the walls of the canals 
(except the analysis of 3D-periodic travelling waves 
which correspond to the reflexion of a 2D-wave off a vertical wall, see Reeder-Shinbrot~\cite{ReSh}, 
Craig and Nicholls~\cite{ABZ3CN} and Iooss-Plotnikov~\cite{IP}). 

We hence consider a fluid domain which at time~$t$ is of the form
$$
\Omega(t)=\left\{ (x_1,x_2,y)\in M \times \xR \, :\, b(x)< y < \eta(t,x), ~x= (x_1, x_2)\right\},
$$
where $M= (0,1) \times \xR$ in the case of the canal and $M= (0,1) \times (0,L)$ 
in the case of a rectangular basin, 
and $b$ is a fixed continuous function on $M$ describing the bottom. 

Denote by~$\Sigma$ the free surface 
and by~$\Gamma$ the fixed boundary of the canal:
$$
\Sigma(t)=\{ (x_1,x_2,y)\in M \times \xR\,:\, y=\eta(t,x)\},
$$
and we set~$\Gamma=\partial \Omega(t)\setminus \Sigma(t)~$ (which does not depend on time). We have 
$$
\Gamma= \Gamma_1 \cup \Gamma_2,
$$
\begin{equation}\label{eq.bord}
\begin{aligned}
\Gamma_1&=  \{\, (x_1, x_2,y) \in M\times \xR\, ;\,   b(x)=y\,\} \\
\Gamma_2&= \{\, (x_1, x_2,y) \in \partial M \times \xR\,;\,  b(x) <y< \eta (x_1, x')\,\}.
\end{aligned}
\end{equation}

Denote by~$n$ the normal to the boundary~$\Gamma$  and denote by~$\nu$ the normal  to 
the free surface~$\Sigma$.  
The irrotational water-waves system is then the following:  the Eulerian velocity field~$v\colon \Omega \rightarrow \xR^{3}$ 
solves the incompressible Euler equation
\begin{equation}\label{E1}
\partial_{t} v +v\cdot \partialyx v + \partialyx P = - g e_y ,\quad \cnxy v =0 \quad \text{curl}_{x,y} v =0 \quad \text{in }\Omega,
\end{equation}
where~$-ge_y$ is the acceleration of gravity ($g>0$) and 
where the pressure term~$P$ can be recovered from 
the velocity by solving an elliptic equation. 
The problem is then given by three boundary conditions. They are
\begin{equation}\label{E2}
\left\{\begin{aligned}
&v\cdot n=0 &&\text{on }\Gamma, \\
&\partial_{t} \eta = \sqrt{1+|\partialx \eta|^2}\, v \cdot \nu \quad 
&&\text{on }\Sigma, \\
& P=0
&&\text{on }\Sigma.
\end{aligned}\right. 
\end{equation}
We notice that the first condition in~\eqref{E2} expresses the fact 
that the particles in contact with the rigid bottom remain 
in contact with it. Notice that to fully make sense, 
this condition requires some smoothness on~$\Gamma$, 
but in general it has a weak variational meaning (see Section~\ref{sec.4}).
 
Finally we impose the initial condition
\begin{equation}\label{E3}
 (\eta, v)\arrowvert_{t=0} = (\eta_0, v_0),
 \end{equation}
where $v_0$ satisfies
$$\quad \cnxy v_0 =0 \quad \text{curl}_{x,y} v_0 =0 \quad \text{in }\Omega_0, \quad v_0\cdot n =0, \text {on } \Gamma.$$ 
It follows 
that  there exists a function $\phi_0: \Omega_0 \to \xR$ such that
$$
v_0 = \nabla_{x,y} \phi_0 \quad \text{in } \Omega_0,  \quad \text{with } \Delta_{x,y} \phi_0 =0.
$$
We set 
$$\psi_0=\phi_0\arrowvert_{y=\eta_0(x)}$$ and introduce the trace of the velocity field 
 $v_0 =(v_{0,x_1}, v_{0,x_2}, v_{0,y})$  on $\Sigma_0 = \{(x, \eta_0(x))\}$ in setting
$$
v_{0,x_1}\arrowvert_{y=\eta_0} = V_{0,x_1}, \quad v_{0,x_2}\arrowvert_{y=\eta_0} 
= V_{0,x_2},  \quad v_{0,y}\arrowvert_{y=\eta_0}= B_0,  \,\,V_0 =(V_{0,x_1}, V_{0,x_2}).
$$

 Similarly, to a solution $v$ of \eqref{E1}-\eqref{E2} we associate $\phi, \psi $ and $(V,B)=v\arrowvert_{y=\eta}$ 
as above.

The stability of the waves is dictated by 
the Taylor sign condition, which is the assumption that there exists a positive 
constant~$c$ such that
\begin{equation}\label{taylor}
\ma(t,x) := -  (\partial_y P)(t,x,\eta(t,x)) \ge c >0.
\end{equation}
\subsection{A simple observation}\label{se.ob}
  We begin with a elementary calculation showing that, at least for regular enough solutions, as soon as the Taylor sign condition~\eqref{taylor}  is satisfied, in the case of vertical walls, it is {\em necessary} that at the points where  the free surface and the boundary of the canal meet ($\Sigma(t) \cap \Gamma$), the scalar product between the two normals (to the free surface and to the boundary of the canal) vanishes :~$\nu\cdot n=0$ on~$\Sigma\cap \Gamma$, which means that the free surface~$\Sigma$ necessarily makes a right-angle with the rigid walls (see~Figure~\ref{fig22}).

\begin{figure}[!h]
\centerline{
\begin{tikzpicture}[samples=100]
\filldraw[fill=blue!20!white, draw=gray] (-3,-1) -- (-3,-2) arc (180:270:1) -- (-2,-3) -- (0,-3) arc (-90:0:1) -- (1,-1)  -- 
(1,-0.7) to [out=180,in=50] node [above] {$\Sigma$} (-2,-0.2) to [out=230,in=0] (-3,-1.2) -- (-3,-1) ;
\draw [gray] (-3,-1) -- (-3,-0.5) ;
\draw [gray] (1,-1) -- (1,-0.5) ;
\node at (-3,-2) [left] {$\Gamma$};
\draw [->] (-3,-1) -- (-3.5,-1)  node[left] {$n$} ;
\draw (-3,-1.4) -- (-2.8,-1.4) -- (-2.8,-1.17) ;
\draw (1,-0.9) -- (0.8,-0.9) -- (0.8,-0.685) ;
\end{tikzpicture}
}
\caption{Two-dimensional section of the fluid domain, exhibiting the right-angles at the interface~$\Sigma\cap \Gamma$}\label{fig22}
\end{figure}
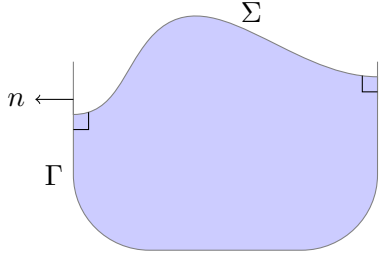

\begin{prop}\label{prop:ad}
Let~$(\eta, v)$ 
be a classical solution of System~\eqref{E1}, \eqref{E2} such  that the Taylor coefficient~$a$ is continuous and non-vanishing and $\eta(t,x)\ge b(x)+h$  \text{for some positive constant} $h$. Then the angle between the free surface,~$\Sigma(t)$ and the boundary of the canal~$\Gamma$  is a right angle: 
$$
\forall t\in [0,T], \forall x \in \Sigma(t) \cap \Gamma,\quad n \cdot \nu (t,x) =0,
$$
which is equivalent to 
\begin{equation}\label{eq.right}
 \partial_{x_1} \eta (t, x_1, x_2) \mid_{x_1= 0,1} =0.
 \end{equation} 
\end{prop}

 \begin{proof}  We give the proof in the case of a canal, the proof for the rectangular basin 
 is similar. Since $\eta_0(x)\ge b(x)+h$   at a point   $m_0$ where $\Sigma(t)$ and $\Gamma$ meet we have $m_0=(\eps, x^0_2,y^0)$  where $\eps = 0$ or $1.$   Let $m = (\eps, x_2,y)$ be a point on $\Gamma$ near $m_0.$ At $m$  the normal $n$ to $\Gamma$ is  $n = (\pm 1,0,0)$.   Taking the scalar product of the equation \eqref{E1} with $n $   we obtain, since $e_y \cdot n =0,$
 \begin{equation}\label{nablaP}
 \big( \nabla_{x,y} P\big) \cdot n    = - (\partial_t v)\cdot n - ((v \cdot \nabla_{x,y})v)\cdot n \quad \text{ at }  m. 
  \end{equation}
   Denote by $(v_{x_1}, v_{x_2}, v_y)$ the three components of the velocity field $v.$
 The first condition in  \eqref{E2} implies that $(v \cdot n)(m)  = \pm v_{x_1}(t,\eps,x_2,y) = 0.$ It follows that $ (\partial_t v )\cdot n   =  \partial_t (v \cdot n)  = 0$ at $m$. Moreover on $\Gamma $ near $m_0$ we have
 \begin{align*}
 \big[((v \cdot \nabla_{x,y})v)\cdot n \big](t,\eps,x_2,y)&= \pm \big[((v \cdot \nabla_{x,y})v_{x_1}\big](t,\eps,x_2,y)\\
 & = \pm \big[(v_{x_2} \partial_{x_2} + v_y \partial_y)v_{x_1}\big](t,\eps,x_2,y)\\
 &= \pm \big[(v_{x_2} \partial_{x_2} + v_y \partial_y)\big] (v_{x_1}(t, \eps,x_2,y))=0.
 \end{align*}
It follows from \eqref{nablaP} that 
\begin{equation}\label{nablaP:n}
 \big(\nabla_{x,y} P\big) \cdot n = 0  \quad \text {at } m.
\end{equation}
Now by the third condition in \eqref{E2} we have $P=0$ on $\Sigma$  and by \eqref{taylor} and our hypothesis on the Taylor coefficient we have $\nabla_{x,y}P \neq 0$ on $\Sigma$.
 It follows that $\nabla_{x,y}P$ is proportional to the normal $\nu$ at $\Sigma$ and by continuity at $\Sigma \cap \Gamma.$ We deduce from \eqref{nablaP:n} that $\nu\cdot n = 0$ at $m_0$. 
   \end{proof}
Once this {\em right angle} property is ensured, it is easy to show that some additional compatibility conditions have also to be fulfilled. Namely, for $ f = B_0, V_{0, x_2}, \partial_{x_1}V_{0, x_1}$, using \eqref{E2},  as soon as the function $\phi$ is smooth enough so that all terms below are defined, we have with $m=(\eps,x_2)$ ($\eps =0$ or $1$).
\begin{equation}\label{eq.comp}\left\{
\begin{aligned}\partial_{x_1}  {\psi} _0 (m)&= \partial_{x_1}   \phi_0 (m, \eta_0(m)) + \partial_{y}   \phi_0 (m, \eta_0(m)) \partial_{x_1}  \eta_0 (m) =0,\\ 
\partial_{x_1}  {B} _0 (m)  &= \partial_{x_1} 
\partial_y \phi_0 (m, \eta_0(m)) + \partial^2_y \phi_0 (m, \eta_0(m))\partial_{x_1}  \eta_0 (m) =0,\\
   \partial_{x_1} {V} _{0,x_2} (m)  &= \partial_{x_1}  \partial_{x_2} \phi_0 (m, \eta_0( m))  
 + \partial_{y} \partial_{x_2} \phi_0(m, \eta_0(m))\partial_{x_1}  \eta_0 (m)  =0,\\ 
  \partial^2_{x_1}  {V} _{0,x_1}(m) 
 &=  \partial_{x_1} ^3 \phi_0(m, \eta_0(m)) + 2 \partial_y \partial_{x_1} ^2 \phi_0 (m, \eta_0(m)) \partial_{x_1} \eta_0(m)\\ 
 &+ \partial_y^2 \partial_{x_1} \phi_0 (m, \eta_0(m) (\partial_{x_1} \eta_0(m))^2 + \partial_y \partial_{x_1} \phi_0 (m, \eta_0(m))   \partial^2_{x_1} \eta_0(m)\\ 
  &= 0,
 \end{aligned}\right.
 \end{equation}
where in the last equality, we used that 
$\partial_{x_1}^3 \phi = - ( \partial_{x_2}^2 + \partial_y ^2)\partial_{x_1}\phi,$  since $\phi$ is harmonic.
\subsection{The result}
As before, we denote by $x_1$ (resp. $x_2$) the variable in $(0,1)$ (resp. in $\xR$). To state our results we need to introduce the uniformly local Sobolev spaces in the $x_2$ direction (these spaces are introduced by Kato in \cite{Kato}). 
Let  $ 1 = \sum_{k\in \xZ} \chi(x_2- k)$ be a partition of unity and define for any $s\in \xR$, 
$$
H^s_{ul} ((0,1) \times \xR) 
= \Big\{ u \in H^s_{loc} ((0,1) \times \xR): \sup_k \| \chi(x_2 - k ) u \|_{H^s( (0,1) \times \xR)} < + \infty\Big\}.
$$
These are Banach spaces when endowed with the norm 
$$
\Vert u \Vert_{H^s_{ul}}=  \sup_k \| \chi(x_2 - k ) u \|_{H^s( (0,1) \times \xR)}.
$$
In Section~\ref{se.ob} we showed that in order to get smooth solutions, a set of compatibility conditions~\eqref{eq.right}, \eqref{eq.comp} {\em have} to be assumed. Here we prove that these conditions are not only {\em necessary}, but they are {\em sufficient}.
\begin{theo} \label{th.Canal}
Set $M= (0,1) \times \xR.$ Let~$s\in (2,3)$, $s\neq \frac 5 2$,  and
$$ \mathcal{H}^s(M )=  H_{ul}^{s+ \mez}( M )\times H_{ul}^{s+\mez}( M ) \times H_{ul}^s (M )\times H_{ul}^s(M ).
$$
Consider   $(\eta_0,\psi_0, V_0, B_0) \in\mathcal{H}^s(M )$  and  assume that, with $\eps = 0,1$
 
$(H_ 1)$  \quad $V_{0, x_1}( \eps, x_2) =0 \text{ and } \partial_{x_1}f(\eps, x_2)=0 \text{ when  } f = \eta_0,\psi_0, B_0, V_{0,x_2}$. Furthermore, $ \partial_{x_1}^2 V_{0,x_1}(\eps, x_2)=0$ if $s>5/2$.
  
$(H_2)$  \quad The Taylor sign condition, $a_0 (x) \geq c>0$ \text{is satisfied at time} $t=0$.
 
$(H_3)$ \quad  $\eta_0(x)\ge b(x)+h$  \text{for some positive constant} $h$.
 
Then there exists  a time $T>0$ and a unique solution $(\eta, v= \nabla_{x,y}\phi)$ of the system \eqref{E1}, \eqref{E2}, \eqref{E3} such that 
\begin{enumerate}[i)]
\item $(\eta,\phi \arrowvert_\Sigma, V, B) \in C([0,T); \mathcal{H}^s(M ))$,
\item \label{item:4} the Taylor sign condition is satisfied at time~$t$ and~$\eta(t) \ge b+h/2$.
\end{enumerate}
\end{theo}
In the case of a rectangular basin we have the following result.
\begin{theo} \label{th.Swimming}
Set $M = (0,1)\times (0,L).$ Let~$s\in (2,3)$, $s\neq \frac 5 2$,  and
 $$
 \mathcal{H}^s(M) 
 =  H^{s+ \mez}( M)\times H^{s+\mez}( M) \times H^s (M )\times H^s(M). 
$$
Consider initial data $(\eta_0,\psi_0, V_0, B_0) \in\mathcal{H}^s( M),$   such that 

$(C_1)$  \quad  $V_{0, x_1}( \eps,x_2) =0 \text{ and } \partial_{x_1}f(\eps, x_2)=0 \text{ when  } f = \eta_0,\psi_0, B_0, V_{0,x_2}$. 
Furthermore, $ \partial_{x_1}^2 V_{0,x_1}(\eps,x_2)=0$ if $s>5/2$. Here 
$\eps = 0$ or $1.$

 $(C_2)$   \quad $V_{0, x_1}( x_1, \delta) =0 \text{ and } \partial_{x_2}f(x_1, \delta)=0 \text{ when  } f = \eta_0,\psi_0, B_0, V_{0,x_2}$. Furthermore, $ \partial_{x_2}^2 V_{0,x_1}(x_1, \delta)=0$ if $s>5/2$. Here $\delta = 0$ or $L.$  
 
 $(C_3)$ \quad  The Taylor sign condition, $a_0 (x) \geq c>0$ is satisfied at time~$t=0,$  
  
 $(C_4)$ \quad$\eta_0(x)\ge b(x)+h$ for some positive constant~$h$.

Then there exists $T>0$ and a unique solution $(\eta, v = \nabla_{x,y} \phi)$ of \eqref{E1}--\eqref{E3}  such that 
 \begin{enumerate}
\item $(\eta,\phi \arrowvert_{\Sigma}, V, B) \in C([0,T); \mathcal{H}^s( M ))$,
\item the Taylor sign condition is satisfied at time~$t$ and~$\eta(t) \ge b+h/2$.
\end{enumerate}
\end{theo}
\begin{rema}
$(i)$ Our results exclude  the case $s= \frac 5 2$ for technical reasons. It would be possible (but unnecessarily complicated) to include this case.
  
$(ii)$ In the case of a flat bottom (say $b(x)=-1$) we do not need assumption $(H_2)$ (and $(C_3))$  which is in this case always satisfied 
as proved by Wu~(\cite{Wu09,WuJAMS}), see also \cite{LannesJAMS}. 
Also, this condition is satisfied under a smallness assumption. 

$(iii)$ Condition  $(H_1),$ when $f = \eta_0,$ says that at $t=0$ the fluid has to be  orthogonal to the fixed vertical walls.   

\end{rema}

 \subsection{Proof of the result}
Following Boussinesq (see~\cite[page 37]{Boussinesq}) the strategy of proof is to perform a symmetrization process (following the process which is illustrated on Figure~\ref{fig3i} below). 
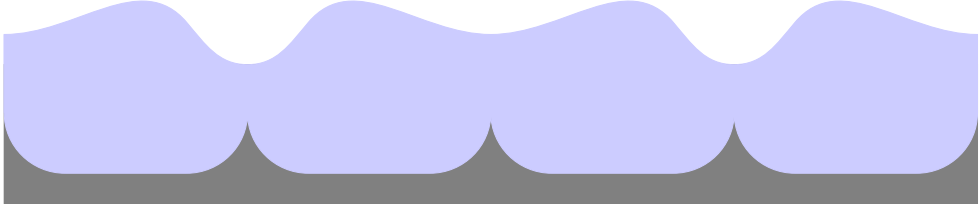
\begin{figure}[ht]
\centerline{
\begin{tikzpicture}[scale=0.8,samples=200]
\filldraw[gray] (-7,-1.2) -- (-7,-3.5) -- (9,-3.5) -- (9,-1.2)  -- cycle;
\filldraw[fill=blue!20!white, draw=blue!20!white] 
(1,-0.7) to [out=180,in=50] (-2,-0.5) to [out=230,in=0] (-3,-1.2) 
to [in=310,out=180] (-4,-0.5) to [out=130,in=0] (-7,-0.7) -- (-7,-2) 
arc (180:270:1) (-6,-3) -- (-4,-3) arc (-90:0:1) -- (-3,-2) arc (180:270:1) -- (-2,-3) -- (0,-3) arc (-90:0:1) -- (1,-2) 
arc (180:270:1) -- (2,-3) -- (4,-3) arc (-90:0:1) -- (5,-2)  arc (180:270:1) -- (6,-3) -- (8,-3) arc (-90:0:1) 
-- (9,-2) -- (9,-0.7) to [out=180,in=50] (6,-0.5) to [out=230,in=0] (5,-1.2)to [in=310,out=180] (4,-0.5) to [out=130,in=0] (1,-0.7) ;
\draw [blue!20!white] (9,-2) -- (9,-0.7) ;
\draw [blue!20!white] (-7,-2) -- (-7,-0.7) ;
\end{tikzpicture}
}
\caption{Two-dimensional section of the extended fluid domain.}  \label{fig3i}
\end{figure}

Once this symmetrization process is performed, we will apply our result~\cite[Theorem 2.3]{ABZ4}   to conclude.
 \subsubsection{The periodization process}\label{sec.4.1}
 Without additional assumptions, the reflection procedure should  yield in general a Lipschitz singularity. However, here the possible singularities are weaker according to the physical hypothesis $(H_1).$
   
 For a function $v$ defined on $(0, +\infty)$, define 
$\even{v}$ and $\odd{ v}$ to be the even and odd extensions of $v$ to $(-\infty, + \infty)$ defined by 
\begin{equation}\label{eq.und}
\begin{aligned}
\even{v} ( y) &= \begin{cases} & v( -y), \text{ if } y<0\\
& v(y) \text{ if }y \geq 0.
\end{cases}\\
\odd{v} ( y) &= \begin{cases} & -v( -y), \text{ if } y<0\\
& v(y) \text{ if }y \geq 0.
\end{cases}
\end{aligned}
\end{equation} We have the following result
\begin{prop}\label{prop.7.4}We have 
\begin{enumerate}
\item Assume that $0 \leq s<\frac{3}{2}$. Then the map 
$ v\mapsto \even{v} $ is continuous from $H^s( 0, + \infty)$ to $H^s( \mathbf{R})$.
 \item Assume that $  \frac{3}{2} <s < \frac 7 2. $  Then the map 
$v\mapsto \even{v}$ is continuous from the space $\{v\in H^s(0,+\infty): v'(0) =0\}$  to $H^s(\xR)$.
\item Assume that $0 \leq s < \mez $. Then the map 
$ v\mapsto \odd{v} $ is continuous from $H^s( 0, + \infty)$ to $H^s( \mathbf{R})$.
  \item Assume that $\mez \leq s < \frac{5}{2} $. Then the map 
$ v\mapsto \odd{v} $ is continuous from    the space $\{v \in H^s(0,+\infty): v(0)=0\}$  to $H^s(\xR)$.
    \item Assume that $ \frac 5 2< s \leq 4$.  Then the map 
$v\mapsto \odd{v}$ is continuous from the space $\{ v \in H^s(0,+\infty): v(0)= v''(0) =0\}$  to $H^s(\xR)$.
\end{enumerate}
\end{prop}
\begin{proof}
Let  $I =(0,+\infty).$  Then $C_0^\infty(\overline{I})$ is dense in $H^s(I)$ for all $s \in \xR$.

$(1)$ The case $ s=0$ is trivial since 
$\Vert  \even{v} \Vert^2_{L^2(\xR)} = 2 \Vert   v  \Vert^2_{L^2(I)}.$ 
Consider now the case $0<s<1.$ Then the square of the $H^s(\xR)$-norm of $\even{v}$ is equivalent to
$$
\Vert  \even{v} \Vert^2_{L^2(\xR)} 
+ A ,\quad A:= \iint_{\xR \times \xR}
\frac{ \vert \even{v}(x) -\even{v}(y)\vert^2}{\vert x-y\vert^{1+ 2s}}\, dxdy.
$$
Then we can write
$$
A= 2\iint_{I \times I }\frac{ \vert {v}(x) - {v}(y)\vert^2}{\vert x-y\vert^{1+ 2s}}\, dxdy +2 \iint_{I  \times I }\frac{ \vert  {v}(x) - {v}(y)\vert^2}{\vert x+y\vert^{1+ 2s}}\, dxdy:= A_1 +A_2.
$$
We have 
$$
A_1 \leq 2\Vert v \Vert^2_{H^s(I)}, \quad A_2 \leq 2 \Vert v \Vert^2_{H^s(I)}
$$
since $\frac{1}{(x+y)^{1+2s}} \leq \frac{1}{\vert x-y \vert^{1+2s}}.$ The case $s=1$ being  straightforward consider the case $1< s < \frac{3}{2}.$ Set $\sigma = s-1\in (0, \mez).$ 
Then
$$
\Vert \even{v} \Vert^2_{H^s(\xR)} = \Vert \even{v} \Vert^2_{L^2(\xR)}  + \Vert \even{ \partial_xv} \Vert^2_{H^\sigma(\xR)}.
$$
Since $0< \sigma< 1$ we have 
\begin{align*}
&\Vert \even{ \partial_xv} \Vert^2_{H^\sigma(\xR)} \leq C(A_0 + A_1 +A_2)\\
&A_0 = \Vert \even{ \partial_xv} \Vert^2_{L^2(\xR)}
\leq C_1  \Vert {v'} \Vert^2_{L^2(I)}\leq C_1  \Vert {v } \Vert^2_{H^s(I)}\\
& A_1 =   \iint_{I  \times I } \frac{ \vert {v'}(x) - {v'}(y)\vert^2}{\vert x-y\vert^{1+ 2\sigma}}\, dxdy 
\leq C_2\Vert v' \Vert_{H^\sigma(I)} \leq C_2\Vert v  \Vert_{H^s(I)} \\
&A_2 = \iint_{I  \times I } \frac{ \vert {v'}(x) +{v'}(y)\vert^2}{\vert x+y\vert^{1+ 2\sigma}}\, dxdy.
\end{align*} 
Eventually we have
$$
A_2 \leq C_3 \int_I \frac{ \vert v'(x)\vert^2}{\vert x \vert^{2 \sigma}}\, dx 
\leq C_4\Vert v' \Vert_{H^\sigma(I)} \leq C_4\Vert v  \Vert_{H^s(I)}
$$
by Theorem 11.2 in \cite{Lions}, since $0< \sigma <\mez$. 
This completes the proof of $(1)$.

$(2)$ 
If $\frac{3}{2} < s < 2$ let $\sigma = s-1 \in (\mez , 1);$ arguing as above we see that
 $$\Vert \even{v} \Vert^2_{H^s(\xR)} \leq C \Big( \Vert {v} \Vert^2_{H^s(I)} + \int_I \frac{\vert v'(x)\vert^2}{\vert x \vert^{2 \sigma}}\, dx dy \Big).$$
Now since $v' \in H^\sigma(I)$ and $v'(0) =0$ 
we can apply Theorem 11.3 in \cite{Lions} which ensures 
that the integral in the right hand side can be estimated by $C \Vert {v} \Vert^2_{H^s(I)} $.
The case $s=2$ being straightforward let  $2<s \leq \frac{7}{2}.$ Then  
$$
\Vert \even{v} \Vert^2_{H^s(\xR)}  \leq C \big(\Vert \even{v} \Vert^2_{L^2(\xR)}  + \Vert \partial_x^2\even{v} \Vert^2_{H^{s-2}(\xR)}\big).
$$
Since $0< s-2 < \frac{3}{2}$ and  
$v'(0) = 0$ we may apply the same argument as in the case $(1)$ to ensure that $\Vert \partial_x^2\even{v} \Vert^2_{H^{s-2}(\xR)}  \leq C \Vert v \Vert^2_{H^s(I)}. $ 

The cases $(3)$ to $(5)$  are proved by exactly the same arguments.
\end{proof}

To state the reflection procedure in higher dimension 
we need to introduce the uniformly local Sobolev spaces in $\xR^n, n \geq 2$.
 
Let  $ 1 = \sum_{k\in \xZ} \chi(x - k)$ be a partition of unity in $\xR^n$ and define for any $s\in \xR$, 
$$
H^s_{ul} (\xR^n) = \{ u \in H^s_{loc} (\xR^n): \sup_k \| \chi(\cdot  - k ) u \|_{H^s(\xR^n)} < + \infty\}.
$$
These are Banach spaces when endowed with the norm 
$$
\Vert u \Vert_{H^s_{ul}}=  \sup_k \| \chi(\cdot - k ) u \|_{H^s(\xR^n)}.
$$
Now, if $v$ is  a function  on $M=(0,1) \times \xR^d$ 
we define the even (resp. odd) periodic extensions on $\xT\times \xR^d$, $\even{v}$ (resp. $\odd{v}$), by 
\begin{equation}\label{eq.per}
\begin{aligned}
\even{v} ( x_1, x') &= \begin{cases}  v( -x_1, x'), &\text{ if } -1<x_1<0,\\
 v(x_1, x'), & \text{ if }0 \leq x_1 < 1,\\
 v(x_1-2k, x').&\text{ if } x_1 - 2k \in (-1,1), \, k\in \xZ.
\end{cases}\\
\odd{v} ( x_1, x') &= \begin{cases}  -v( -x_1, x'), &\text{ if } -1<x_1<0,\\
 v(x_1, x'), &\text{ if }0 \leq x_1 < 1,\\
v(x_1- 2k, x').&\text{ if } x_1 - 2k \in (-1,1), \, k\in \xZ.
\end{cases}
\end{aligned}
\end{equation}
\begin{coro}\label{cor.sym}
Let $M = (0,1)\times \xR.$

$(1)$ Assume $0 \leq s < \frac{3}{2}$  then the map 
$ v\mapsto \even{v} $ is continuous from $H_{ul}^s(M)$ to $H_{ul}^s( \xR^2)$.

$(2)$ Assume  $ \frac{3}{2} < s <\frac{7}{2}. $ Let 
$$
E_s= \{ u \in  H_{ul}^s(M) : \partial_{x_1}u (\eps, x_2) = 0, \eps = 0,1, \forall x_2 \in \xR\}.
$$ 
Then the map $v\mapsto \even{v}$ is continuous from $E_s$ to $H_{ul}^s( \xR^2)$.
 
$(3)$ Assume  $ 0 \leq s < \mez$  then the map 
$ v\mapsto \odd{v} $ is continuous from $H_{ul}^s(M)$ to $H_{ul}^s( \xR^2)$.

$(4)$ Assume $ \mez< s < \frac{5}{2}.$ Let 
$$
F_s= \{u \in H_{ul}^s(M): u(\eps,x_2) = 0,  \eps = 0,1, \forall x_2\in \xR\}.
$$
Then the map 
$v\mapsto \odd{v}$ is continuous from $F_s$ to $H_{ul}^s( \xR^2)$.

$(5)$ Assume  $\frac 5 2 < s\leq 4.$  Let 
$$
G_s= \{ u \in H_{ul}^s(M): u(\eps,x_2) =   \partial_{x_1}^2 u(\eps,x_2)  =0 , \eps = 0,1, \forall x_2 \in \xR\}.
$$   
Then the map 
$v\mapsto \odd{v}$ is continuous from $G_s$ to $H_{ul}^s( \xR^2)$.
\end{coro}
\begin{proof}
Since $\even{(D_{x_2}^\alpha v)} = D_{x_2}^\alpha (\even{v})$, $\odd{(D_{x_2}^\alpha v)} = D_{x_2}^\alpha (\odd{v})$ the result is clearly a one dimensional result and it is enough to prove it for the one dimensional case, in which case it is a direct consequence of Proposition~\ref{prop.7.4} and a localization argument.
\end{proof}
 Consider now an initial data ($ \eta_0, \psi_0= \phi_0 \arrowvert_{ \Sigma_0}, V_0,B_0$) satisfying the assumptions in Theorem~\ref{th.Canal} 
 and define
 $$ \widetilde{\eta} _0  = \even{\eta}_0,  \quad \widetilde{\psi}_0 = \even{\psi}_0,  \quad \widetilde{V}_{0,x_1}= \odd{V}_{0, x_1}, \quad \widetilde{V}_{0, x_2}= \even{V}_{0, x_2},  \quad   \widetilde{B}_0 = \even{B}_0 $$
 on $ \xT\times  \xR$.
 
 Recall (see Theorem \ref{th.Canal}) that, with $M = (0,1) \times \xR,$ we have set
$$ \mathcal{H}^s(M)=  H_{ul}^{s+ \mez}(M)\times L_{ul}^2(M) \times H_{ul}^s (M)\times H_{ul}^s(M), $$
 and introduce
$$  \mathcal{H}^s(\xR^2 )=  H_{ul}^{s+ \mez}(\xR^2 )\times L_{ul}^2(\xR^2 ) \times H_{ul}^s ( \xR^2)\times H_{ul}^s(\xR^2).$$ 
Then we have the following lemma. 
\begin{lemm}
Let $2<s<3, s \neq \frac{5}{2} $ and  $ (\eta_0,\psi_0, V_0, B_0) \in\mathcal{H}^s(M)$  satisfying the hypothesis $(H_1)$ in Theorem \ref{th.Canal}. Then   $ (\widetilde{\eta}_0, \widetilde{\psi} _0, \widetilde{V}_0, \widetilde{B}_0) \in   \mathcal{H}^s(\xR^2 ) $  and are $2$- periodic with respect to the $x_1$ variable.
  \end{lemm}
\begin{proof}
This follows immediately from the hypothesis $(H_1)$ and Corollary \ref{cor.sym}.
\end{proof}

In the case of a rectangular basin, 
performing both reflection and periodizations 
with respect to the $x_1$ and the $x_2$ variables leads similarly to extensions  
$$
(\widetilde{\eta}_0, \widetilde{\psi} _0, \widetilde{V}_0, \widetilde{B}_0) \in H_{ul}^{s+ \mez}( \xR^2) \times L^2_{ul} (\xR^2)\times H_{ul}^s (\xR^2)\times H_{ul}^s (\xR^2)
$$
which are $2$- periodic with respect to the $x_1$ 
variable and $2L$ periodic with respect to the $x_2$ variable ).

\subsubsection{Conclusion}
We are now in position to apply Theorem~\ref{theo:princ}.
We consider first the case of the canal. Starting from $(\eta_0, \psi_0, V_0, B_0)$, 
we define $ (\widetilde{\eta}_0, \widetilde{\psi} _0, \widetilde{V}_0, \widetilde{B}_0) $ 
their periodized extensions following the process in Section~\ref{sec.4.1}.
Let $(\widetilde{\eta}, \widetilde{v} )$ be the solution of the free surface 
water waves system given by Theorem~\ref{theo:princ}. 
Since the initial data $ (\widetilde{\eta}_0, \widetilde{\psi} _0, \widetilde{V}'_0, \widetilde{B}_0)$ 
are even while $\widetilde{V}_{0,x_1}$ is odd, 
our uniqueness result guaranties that the solution 
satisfies the same symmetry property 
(because if we consider our solution, the function obtained 
by symmetrization is also a solution with same initial data). 
The same argument shows that as 
the initial data are $2$-periodic with 
respect to the variable $x_1$, so is the solution. 
As a consequence if we define $v,\eta, P$ 
as the trace of $\widetilde{v}, \widetilde{\eta}, \widetilde{P} $ on $(0,1) \times \xR$, 
we get that they satisfy trivially the system free boundary Euler equation
\begin{equation}\label{syst.canal}
\begin{aligned}
\partial_{t} v +v\cdot \partialyx v + \partialyx P &= - g e_y ,\quad \cnxy v =0 \quad\text{in }\Omega,\\
\partial_{t} \eta &= \sqrt{1+|\partialx \eta|^2}\, v \cdot \nu \quad \text{on }\Sigma, \\
P&=0 \quad \text{on }\Sigma,
\end{aligned}
\end{equation}
and to conclude on the existence point in Theorem~\ref{th.Canal}, 
it only remains to check that the "solid wall condition" 
\begin{equation}\label{eq.solid}
v\cdot n =0, \text{ on } \Gamma= \Gamma_1 \cup \Gamma_2 \end{equation}
is satisfied.
On $\Gamma_1$ it is a straightforward consequence 
of the condition $\widetilde{v} \cdot \widetilde{\Gamma}=0$, 
while on $\Gamma_2$ it is simply consequence of the fact 
that the component  of the velocity field along $x_1$, $\widetilde{v}_{x_1}$ 
is odd and $2$-periodic. To prove the uniqueness part in Theorem~\ref{th.Canal}, 
starting from a a solution of ~\eqref{syst.canal},~\eqref{eq.solid}, 
on the time interval $[-T, T]$, if we define the 
function $\widetilde{v}, \widetilde{\eta}$ at each 
time $t$ following the same procedure, we end up 
with a solution of ~\eqref{E1},~\eqref{E2} in the 
domain $\{ (t,x,y); t \in (-T, T), (x,y) \in \widetilde{\Omega} (t)\} $, 
at the same level of regularity. Indeed, the jump formula gives 
$$ \partial_{t} \widetilde{v} +\widetilde{v}\cdot \partialyx \widetilde{v} + \partialyx \widetilde{P} = - g e_y + [ v_{x_1}\cdot \partial_{x_1} v]\otimes \delta_{\Gamma_2}= -ge_y, $$
where in the last equality we used that 
the component of the velocity field along $x_1$ 
vanishes on $\Gamma_2$. The uniqueness part 
in Theorem~\ref{th.Canal} consequently follows from 
the uniqueness part in Theorem~\ref{theo:princ}. The case of a rectangular basin is similar.
\section{Technical results}\label{sec.3}
 \subsection{Invariance}  
  The following result shows that the definition of the uniformly local 
  Sobolev spaces does not depend on the choice of the function $\chi$ satisfying \eqref{kiq}.
 \begin{lemm}\label{invariance}
Let $E$ be a normed space of functions from $\xR^d$ to $\xC$ such that  
$$ \forall \, \theta \in W^{\infty,\infty}(\xR^d) \quad  \exists C>0:  \Vert \theta u\Vert_E \leq C \Vert u\Vert_E  \quad   \forall u\in E $$
where $C$ depends only on a finite number of semi-norms of $\theta$ in $W^{\infty,\infty}(\xR^d)$.
 Then  for any   $\widetilde{\chi} \in C_0^\infty(\xR^d)$ there exists $C'>0$ such that
\begin{equation}\label{chitilde}
 \sup_{k\in \xZ^d} \Vert \widetilde{\chi}_k u \Vert_{E} \leq C' \sup_{q\in \xZ^d} \Vert \chi_q u \Vert_{E} 
 \end{equation}
 where $\widetilde{\chi}_k(x) =\widetilde{\chi}(x-k)$.
\end{lemm}
\begin{proof}
Let   $\underline{\chi}\in C_0^\infty(\xR^d)$ be equal to one on the support of $\chi $. We write with $N = d+1$
$$
\widetilde{\chi}_k \chi_q u
= \langle k-q\rangle^{-N} \left[\frac{\langle k-q\rangle^{ N} }{\langle x-q\rangle^{ N}}\widetilde{\chi}_k\right] \left[ \langle x-q\rangle^{ N}\underline{\chi}_q\right] \chi_q u.
$$
Since the two functions inside the brackets belong to $W^{\infty,\infty}(\xR^d)$ with semi-norms independent of $k,q,$ using the assumption in the lemma we deduce that
$$
\Vert \widetilde{\chi}_k u\Vert_E \leq   \sum_{q\in \xZ^d} \Vert \widetilde{\chi}_k \chi_q u\Vert_E  \leq C\sum_{ q\in \xZ^d}\langle k-q\rangle^{-N} \sup_{q\in \xZ^d} \Vert  {\chi}_qu\Vert_E,
$$
which completes the proof.
\end{proof}
 
\begin{lemm}\label{ulN}
Let $\mu \in \xR$ and $N \geq d+1$. Then there exists $C>0$ such that
\begin{equation}\label{N>d+1}
\sup_{x\in \xR^d} \Vert \langle x-\cdot \rangle^{-N} u \Vert_{H^\mu(\xR^d)} \leq  C \Vert  u \Vert_{H_{ul}^\mu(\xR^d)}
\end{equation}
for all $u\in H^{\mu}_{ul}(\xR^d)$.
\end{lemm}
\begin{proof}
Indeed we have
$$\Vert \langle x-\cdot \rangle^{-N} u \Vert_{H^\mu } \leq \sum_{q \in \xZ^d}\Vert \langle x-\cdot \rangle^{-N} \chi_qu \Vert_{H^\mu}
$$
and we write
  $$
  \langle x-y\rangle^{-N} \chi_q (y)u (y)  =  \frac{1}{\langle x-q \rangle^N} \frac{\langle x-q \rangle^N}{ \langle x-y\rangle^{ N}}\widetilde{\chi}_q(y)  \chi_q (y)u (y)$$
   where $\widetilde{\chi}\in C_0^\infty(\xR^d), \widetilde{\chi} = 1$ on the support of $\chi$. 
  This implies that
  $$
  \sum_{q \in \xZ^d}\Vert \langle x-\cdot \rangle^{-N} \chi_qu \Vert_{H^\mu}
  \leq C_N \sum_{q\in \xZ^d} \frac{1}{\langle x-q \rangle^N} 
   \Vert  u \Vert_{H_{ul}^\mu} \leq C'_N \Vert  u \Vert_{H_{ul}^\mu},
  $$
  since the function $y \mapsto \frac{\langle x-q \rangle^N}{ \langle x-y\rangle^{ N}}\widetilde{\chi}_q(y)$ belongs to $W^{\infty,\infty}(\xR^d)$ with semi-norms uniformly bounded 
  (independently of $x$ and $q$).
  \end{proof}
 
\subsection {Product laws}
\begin{prop}\label{Ho}
$(i)$  Let $\sigma_j \in \xR, j=1,2$ be such that $\sigma_1+\sigma_2>0$ and  $u_j \in H^{\sigma_j}_{ul}(\xR^d), j=1,2$. Then $u_1 u_2 \in H^{\sigma_0}_{ul}(\xR^d)$ for $\sigma_0 \leq \sigma_j$ and $\sigma_0< \sigma_1 +\sigma_2 -\frac{d}{2}$. Moreover we have
$$\Vert u_1 u_2 \Vert_{H^{\sigma_0}_{ul}(\xR^d)} \leq C\Vert u_1 \Vert_{H^{\sigma_1}_{ul}(\xR^d)} \Vert  u_2 \Vert_{H^{\sigma_2}_{ul}(\xR^d)}.$$
 $(ii)$  Let $s\geq 0$ and $u_j \in H^s_{ul}(\xR^d)\cap L^\infty(\xR^d), j=1,2$. Then $u_1 u_2 \in H^s_{ul}(\xR^d)$ and 
$$\Vert u_1 u_2 \Vert_{H^{s }_{ul}(\xR^d)} \leq C\big( \Vert u_1\Vert_{L^\infty(\xR^d)}\Vert u_2\Vert_{H^s_{ul}(\xR^d)} + \Vert u_2\Vert_{L^\infty(\xR^d)}\Vert u_1\Vert_{H^s_{ul}(\xR^d)} \big).$$

 $(iii)$ Let $F\in C^\infty(\xR^N, \xC)$ be such that $F(0) = 0$. Let $s> \frac{d}{2}$. If $U\in \big(H^s_{ul}(\xR^d)\big)^N$ then $F(U) \in H^s_{ul}(\xR^d)$ and 
$$\Vert F(U) \Vert_{H^s_{ul}(\xR^d)} \leq  G(\Vert U\Vert_{ (L^\infty(\xR^d))^N})\Vert U \Vert_{(H^s_{ul}(\xR^d))^N} $$
for an increasing function $G: \xR^+\to\xR^+$.
\end{prop}
 \begin{proof}
 The proofs are straightforward extensions of the proofs in the classical Sobolev spaces case. Indeed let us show $(i)$ for instance. Let $\chi_q$ be defined in~\eqref{kiq} and $\widetilde{\chi} \in C^\infty_0(\xR^d)$ be equal to one on the support of $\chi$. Then from the classical case we can write
\begin{align*}
\Vert \chi_q u_1u_2\Vert_{H^{s_0}} &= \Vert \chi_q u_1\widetilde{\chi}_q u_2\Vert_{H^{s_0}}
\leq C\Vert \chi_q u_1\Vert_{H^{s_1}}\Vert \widetilde{\chi}_q u_2\Vert_{H^{s_2}}\\ 
&\leq C\Vert  u_1\Vert_{H^{s_1}_{ul}}\Vert u_2\Vert_{H^{s_2}_{ul}}.
\end{align*}
The proofs of $(ii)$ and $(iii)$ are similar.
\end{proof}
The following spaces will be used in the sequel
\begin{defi}
Let $ p\in [1,+\infty],$ $J =(z_0,0), z_0<0$ and $\sigma \in \xR$. 
\begin{enumerate}
\item The space $L^p(J, H^\sigma(\xR^d))_{ul}$ is defined as the space of measurable functions 
$u$ from $ \xR^d_x \times J_z$ to $ \xC$ such that
$$
\Vert u \Vert_{L^p(J, H^\sigma(\xR^d))_{ul}}:= \sup_{q\in \xZ^d}
\Vert \chi_q u \Vert_{L^p(J, H^\sigma(\xR^d))} <+\infty.
$$  
\item We set
  \begin{equation*}\label{X,Y}
\begin{aligned}
  X^\sigma_{ul}(J) &= L^\infty(J, H^\sigma(\xR^d))_{ul}\cap L^2(J, H^{\sigma+\mez}(\xR^d))_{ul}\\
  Y^\sigma_{ul} (J) & = L^1(J, H^\sigma(\xR^d))_{ul} +  L^2(J, H^{\sigma - \mez}(\xR^d))_{ul}
  \end{aligned}
\end{equation*}
endowed with their natural norms.
\item 
We define the spaces $X^\sigma(J), Y^\sigma(J) $  by the same formulas without the subscript $ul$.
\end{enumerate}
\end{defi}
Notice that $L^\infty(J, H^\sigma(\xR^d))_{ul} = L^\infty(J, H^\sigma_{ul}(\xR^d))$. 
  
\begin{lemm}\label{Houl}
Let $\sigma_0,\sigma_1,\sigma_2$ be real numbers such that $\sigma_1+\sigma_2>0,  \sigma_0\leq \sigma_j, j=1,2, \sigma_0 < \sigma_1+ \sigma_2-\frac{d}{2} $ and $2\leq p\leq +\infty$.  Then
$$
\Vert uv \Vert_{L^p(J, H^{\sigma_0}(\xR^d))_{ul}} \leq C \Vert u  \Vert_{L^\infty(J, H^{\sigma_1}(\xR^d))_{ul}}\Vert  v \Vert_{L^p(J, H^{\sigma_2}(\xR^d))_{ul}}
$$ 
whenever the right hand side is finite.

The same inequality holds for the spaces without the subscript $ul$.
\end{lemm}
\begin{proof}
This follows immediately from Proposition~\ref{Ho} $(i) $ and~\eqref{chitilde}.
\end{proof}

\begin{lemm}\label{algebre}
If $\sigma> \frac{d}{2}$ 
the spaces $X^\sigma_{ul}(J)$  and $ X^\sigma (J)$ are algebras.
\end{lemm}
 \begin{lemm}\label{est:produit}
 Let $s_0> 1+ \frac{d}{2}, \mu>0$ and $J= (-1,0)$. Then we have 
 \begin{align}
 \label{est:prod1}  \quad &\Vert fg\Vert_{X^\mu_{ul}} \leq C\bigl( \Vert f\Vert_{L^\infty(J, H^{s_0-1})_{ul}}\Vert g \Vert_{ X^\mu_{ul}} + \Vert g\Vert_{L^\infty(J, H^{s_0-1} )_{ul}}\Vert f \Vert_{X^\mu_{ul}} \bigr), \\
 \label{est:prod2}  \quad &\Vert fg\Vert_{X^\mu_{ul}} \leq C\bigl( \Vert f\Vert_{L^\infty(J, H^{s_0-1} )_{ul}}\Vert g \Vert_{ X^\mu_{ul}} + \Vert g\Vert_{L^\infty(J, H^{s_0-\frac{3}{2}} )_{ul}}\Vert f \Vert_{X^{\mu+\mez}_{ul}} \bigr). 
 \end{align}
 Let $F\in C_b^\infty(\xC^N, \xC)$ be such that $F(0)=0$. Then there exists a non decreasing function $\mathcal{F}:\xR^+ \to \xR^+$ such that for $\mu>\frac{d}{2}$  we have 
 \begin{equation}\label{est:prod3}
   \Vert F(U)\Vert_{X^\mu_{ul}} \leq \mathcal{F}\bigl( \Vert U\Vert_{L^\infty(J, H^{s_0-1})_{ul}} \bigr)\Vert  U\Vert_{X^\mu_{ul}}.
\end{equation}
 \end{lemm}
\begin{proof}
The first and the third estimates follow easily from $(ii), (iii)$ in Proposition \ref{Ho}. To prove  the second one  we start from the inequality (see \cite{ABZ3}, Corollary 2.12)
 $$ \Vert \chi_kfg\Vert_{H^t } \leq C\bigl( \Vert \chi_k f\Vert_{ L^\infty}\Vert \widetilde{\chi}_kg \Vert_{ H^t} + \Vert \widetilde{\chi}_kg\Vert_{C_*^{-\mez}}\Vert \chi_k f \Vert_{H^{t+\mez} } \bigr), \quad t>0$$
  where $\widetilde{\chi} \in C_0^\infty(\xR^d)$ is equal to one on the support of $\chi $. Then we use the  continuous embeddings: $H^{s_0-1} \subset L^\infty,  H^{s_0- \frac{3}{2}} \subset C_*^{-\mez} $ and the above inequality for $t = \mu, t = \mu+\mez$.
 \end{proof}
 \subsection{Continuity of the pseudo-differential operators}
We have the following result  which reflects the pseudo-local character of the pseudo-differential operators.
Recall that $S^m_{1,0}$ is the set of symbols $p\in C^\infty(\xR^d \times \xR^d)$ such that
$$  \vert D_\xi^\alpha D_x^\beta p(x,\xi)\vert \leq C_{\alpha,\beta}(1+\vert \xi \vert)^{m-\vert \alpha \vert}\quad \forall \alpha, \beta \in \xN^d, \forall (x,\xi) \in \xR^d \times \xR^d .$$
\begin{prop}\label{pseudo}
Let $P$ be a pseudo-differential operator whose symbol belongs to the class~$S^m_{1,0}$. 
Then for every $s \in \xR$ there exists a constant $C>0$ such that 
$$
\Vert Pu \Vert_{H^s_{ul}(\xR^d)} \leq C \Vert u \Vert_{H^{s+m}_{ul}(\xR^d)},
$$
for every $u\in H^{s+m}_{ul}(\xR^d), $ where $C$ depends only on   semi-norms of the symbol in $S^m_{1,0}$.
\end{prop}
\begin{proof}
Write
\begin{equation}\label{pseudo:ul}
\chi_k Pu = \sum_{\vert k-q\vert \leq2} \chi_k P\chi_q u 
+  \sum_{\vert k-q\vert \geq 3} \chi_k P\chi_q u =: A+ \sum_{\vert k-q\vert \geq 3}B_{k,q}.
\end{equation}
The first sum is finite depending only on the dimension. To bound it in $H^s(\xR^d)$ we use the usual continuity of pseudo-differential operators. For the second one let $n_0\in \xN, n_0 \geq s$. We shall prove that
\begin{equation}\label{Bkq}
\Vert D_x^\alpha B_{k,q} \Vert_{L^2(\xR^d)} \leq \frac{C_d}{\langle k-q \rangle^{d+1}} \Vert u \Vert_{H^{s+m}_{ul} }, \quad \vert \alpha \vert \leq n_0 
\end{equation}
which will complete the proof of Proposition~\ref{pseudo}. 

Notice that, due to the presence of $\chi_k,$ we have $\Vert D_x^\alpha B_{k,q} \Vert_{L^2 } \leq C \Vert D_x^\alpha B_{k,q} \Vert_{L^\infty }$. We have
$$ D_x^\alpha B_{k,q}(x)= \big \langle D_x^\alpha K(x,\cdot), \chi_qu \big\rangle $$
with
$$K(x,y) = (2\pi)^{-d}\int_{\xR^d} e^{i(x-y)\cdot \xi} p(x,\xi) d\xi \, \chi_k(x) \widetilde{\chi}_q(y)$$
where $\widetilde{\chi} \in C_0^\infty(\xR^d), \widetilde{\chi}   =1$ on the support of $\chi$. 

Now on the support of $\chi_k(x) \widetilde{\chi}_q(y)$ we have $\vert x-y\vert \geq \delta \vert k-q\vert, \delta>0$. 
Integrating by parts $N$ times (with large $N$ depending on $d,n_0$) 
with the vector field $L = \sum_{j=1}^d \frac{x_j-y_j}{\vert x-y \vert^2}\partial_{\xi_j}$ 
we see that for all $\beta \in \xN^n$ we have
$$\vert D_x^\beta K(x,y)\vert \leq \frac{C_{d, \beta}}{\langle k-q \rangle^{d+1}} \vert \widetilde{\chi}_q(y) \vert, \quad \forall (x,y) \in \xR^d \times \xR^d.$$
It follows that 
\begin{align*}
\vert D_x^\alpha B_{k,q} (x) \vert &\leq   \Vert D_x^\alpha K(x,\cdot)\Vert_{H^{-(s+m)} }\Vert \chi_q u\Vert_{ H^{ s+m } } \\
& \leq  \frac{C_{d, \beta}}{\langle k-q \rangle^{d+1}}\Vert \chi_q u\Vert_{ H^{ s+m } } 
\end{align*}
which proves~\eqref{Bkq} and hence concludes the proof.
\end{proof}
In a particular case the proof above gives the  following more precise result.
\begin{prop}\label{pseudoh}
Let $m \in \xR$, $h(\xi)= \widetilde{h}\big (\frac{\xi}{\vert \xi \vert}\big)\vert \xi \vert^m \psi(\xi)$ 
where $\widetilde{h} \in C^\infty(\xS^{d-1})$ and $\psi \in C ^\infty(\xR^d)$ is such that $\psi(\xi) = 1 $ 
if $\vert \xi \vert \geq 1$, $\psi(\xi) = 0$  if $\vert \xi \vert \leq \mez$. 
Then for every $\mu \in \xR$ there exists a constant $C$ such that
$$\Vert h(D_x) u \Vert_{H^\mu_{ul}(\xR^d)} \leq C \Vert \widetilde{h} \Vert_{H^{d+1}(\xS^{d-1})} \Vert  u \Vert_{H^{\mu+m}_{ul}(\xR^d)} $$
for all $u \in  H^{\mu+m}_{ul}(\xR^d)  $. 
\end{prop}

 We shall use the following result when $p(\xi) = \langle \xi \rangle$ and $p(\xi)  = \vert \xi \vert^2$.
 \begin{lemm}\label{lemmsmooth}
Let $d\geq 1, r>0, m\in\xR$. Let $p\in S^r_{1,0}(\xR^d), a\in S^m_{1,0}(\xR^d)$ two symbols with constant coefficients. We assume that one can find $c_0>0$ such that for all $\xi \in \xR^d$ we have $p(\xi)  \geq c_0 \vert \xi \vert^r$. Then for all $\sigma\in \xR$ and every interval $I = [0,T],$ one can find a positive constant $C$ such that, with $H^s   = H^s (\xR^d)$
\begin{equation}\label{smoothing}
 \Vert e^{-tp(D)}a(D)u\Vert_{L^\infty(I, H^\s)_{ul}} + \Vert e^{-tp(D)}a(D)u\Vert_{L^2(I, H^{\s+\frac{r}{2}})_{ul}} \leq C \Vert u \Vert_{H^{\s+m}_{ul}} 
 \end{equation}
 for all $u\in H^{\s+m}_{ul}$.
\end{lemm}
\begin{proof}
The estimate of the first term in \eqref{smoothing} follows from Proposition \ref{pseudo} since $e^{-tp(D)}a(D)$ is a pseudo-differential operator of order $m$ whose symbol has semi-norms in $S^m_{1,0}$ bounded by  constants depending only on $T$. Let us look at the second term. Set 
$$I_q = \Vert \chi_q e^{-tp(D)}a(D)u\Vert_{L^2(I, H^{\s+\frac{r}{2}})}.$$
One can write 
\begin{equation}\label{Aq-Bq}
\left\{
\begin{aligned}
 I_q &= A_q + B_q\\
A_q&= \sum_{\vert k-q\vert \leq 2} \Vert \chi_q e^{-tp(D)}a(D)\chi_ku\Vert_{L^2(I, H^{\s+\frac{r}{2}})},\\
B_q &= \sum_{\vert k-q\vert \geq 3} \Vert \chi_q e^{-tp(D)}a(D)\chi_ku\Vert_{L^2(I, H^{\s+\frac{r}{2}})}.
\end{aligned}
\right.
 \end{equation}
Since the number of terms in the sum defining $A_q$ is bounded by a fixed constant (depending only on $d$) using a classical computation we can write
\begin{equation}
\begin{aligned}\label{est:Aq}
 A_q  & \leq C_1 \sup_{k \in \xZ^d} \Vert   e^{-tp(D)}a(D)\chi_ku\Vert_{L^2(I, H^{\s+\frac{r}{2}})}\\
  A_q &\leq C_2 \sup_{k \in \xZ^d}\Vert a(D)\chi_ku\Vert_{ H^{\s} } \leq C_3 \sup_{k \in \xZ^d}\Vert \chi_ku\Vert_{ H^{\s+m}} \leq C_3\Vert u\Vert_{ H_{ul}^{\s+m}}. 
 \end{aligned}
  \end{equation}
  Let us look at the term $B_q$. Let $N_0 $ be an integer such that $N_0 \geq \sigma + \frac{r}{2}$. Then 
  $B_q$ is bounded by a finite sum of terms of the form 
  $$  \sum_{\vert k-q\vert \geq 3} \Vert (D^\alpha\chi_q) (D^\beta e^{-tp(D)}a(D))\chi_ku\Vert_{L^2(I,L^2)} $$
  with $\vert \alpha\vert + \vert \beta \vert \leq N_0$. Due to the presence of the function $ D^\alpha\chi_q$, $B_q$ is therefore bounded by a finite number of terms of the form 
   $$  \sum_{\vert k-q\vert \geq 3} \Vert (D^\alpha\chi_q) (D^\beta e^{-tp(D)}a(D))\chi_ku\Vert_{L^2(I,L^\infty).} $$
Now we can write
\begin{equation}\label{def:F}
 F(t,x):= (D^\alpha\chi_q) (D^\beta e^{-tp(D)}a(D))\chi_ku(x) = \langle K(t,x,\cdot),  (\chi_k u)(\cdot) \rangle
 \end{equation}
with
$$K(t,x,y) = (2\pi)^{-d} (D^\alpha\chi_q)(x)  \widetilde{\chi}_k(y)\int e^{i(x-y)\cdot \xi} q(t,\xi)\,  d\xi$$
where  $\widetilde{\chi} \in C_0^\infty(\xR^d)$ is equal to one on the support of $\chi $ and $q(t,\xi) = \xi^\beta e^{-tp(\xi)}a(\xi)$.  It follows that for fixed $(t,x) $ we have 
\begin{equation}\label{est:F-i}
 \vert F(t, x) \vert  \leq \Vert K(t, x,\cdot) \Vert_{H^{-(\sigma +m)}} \Vert \chi_k u \Vert_{ H^{ \sigma +m}}.
 \end{equation}
Let $N_1 \in \xN$ be fixed such that $N_1 \geq -(\sigma +m)$. We shall show that for every $N \in \xN$ one can find $C_N = C_N(T)>0$ such that for every $(t,x)\in I\times \xR^d$ we have 
\begin{equation}\label{est:K}
 \Vert K(t, x,\cdot) \Vert_{H^{N_1}} \leq \frac{C_N}{\langle k-q\rangle^N}\vert  (D^\alpha\chi_q)(x) \vert.
 \end{equation}
 
Indeed for $\vert \mu \vert \leq N_1, D_y^\mu K(x,y)$ is a finite linear combination of terms of the form
$$
J(t,x,y):= (D_x^\alpha\chi_q)(x)  (D_y^\nu\widetilde{\chi}_k)(y)\int e^{i(x-y)\cdot \xi} \xi^\lambda q(t,\xi) d\xi
$$
where $\vert \nu \vert + \vert \lambda \vert = \vert \mu \vert$.
 
We notice  that for all $\gamma \in \xN^d$ we have
  \begin{equation}\label{est:symbq}
 \vert D_\xi ^\gamma (\xi^\lambda q(t,\xi))\vert \le C_\gamma(T) \langle \xi \rangle^{N_0+ N_1 + m - \vert \gamma \vert}.
 \end{equation}

Now let $N\in \xN $ be  such that $N \geq \max(d+1, N_0 + N_1+m + d+1) $ and $ \gamma \in \xN^d$ with $ \vert \gamma \vert = N$. Then 
$$(x-y)^\gamma J(t,x,y) =  (D^\alpha\chi_q)(x)  (D_y^\nu\widetilde{\chi}_k)(y)\int e^{i(x-y)\cdot \xi} (-D_\xi)^\gamma(\xi^\lambda  q(t,\xi)) \, d\xi.$$
It follows from \eqref{est:symbq} that
 $$\vert (x-y)^\gamma J(t,x,y)\vert  \leq C_1(T) \vert (D^\alpha\chi_q)(x)\vert \vert  (D_y^\nu\widetilde{\chi}_k)(y)\vert. $$
Now since $\vert k-q \vert \geq 3,$ on the support of $(D^\alpha\chi_q)(x)  (D_y^\nu\widetilde{\chi}_k)(y)$ we have $\vert x-y \vert \geq \frac{1}{3} \vert k-q \vert$. It follows that 
$$\vert J(t,x,y) \vert \leq \frac{C_2(T)}{ \langle k-q \rangle^N} \vert (D^\alpha\chi_q)(x)\vert \vert  (D_y^\nu\widetilde{\chi}_k)(y)\vert $$
 which proves \eqref{est:K}. According to \eqref{est:F-i} and \eqref{est:K} we obtain
 $$\Vert F\Vert_{L^2(I,L^\infty)} \leq  \frac{C_3(T)}{ \langle k-q \rangle^N} \Vert \chi_k u \Vert_{ H^{ \sigma +m}} \leq \frac{C_3(T)}{ \langle k-q \rangle^N} \Vert  u \Vert_{ H_{ul}^{ \sigma +m}}$$
 which implies that   $B_q \leq C_4(T ) \Vert   u \Vert_{ H_{ul}^{ \sigma +m}}$.  Combined with \eqref{Aq-Bq} and \eqref{est:Aq} this proves  the  estimate of the second term in \eqref{smoothing}.
\end{proof}
  
\begin{coro}\label{e^z} 
Let $m\in \xR$ and $a\in S^m_{1,0}(\xR^d)$. 
Then for every $\sigma \in \xR$ there exists $C>0$ such that 
\begin{equation*} 
\Vert e^{\delta z \langle D_x\rangle}  a(D_x) u \Vert_{X_{ul}^{\sigma- \mez}(-1,0)} \leq C \Vert u \Vert_{H_{ul}^{\sigma+m-\mez}(\xR^d)}
\end{equation*}
for every $\delta>0$ and every $u\in H^{\sigma + m- \mez}_{ul}(\xR^d)$.
\end{coro}

\subsection{An interpolation Lemma.}
We shall use the following   interpolation lemma for which we refer to \cite{Lions} Th\' eor\`eme 3.1.

\begin{lemm}\label{lions}
Let $J = (-1,0)$ and $t\in \xR$.  Let  $f \in L_z^2(J, H^{t+ \mez}(\xR^d))$ be such that $\partial_z f \in L_z^2(J, H^{t-\mez}(\xR^d))$. Then $f \in C_z^0([-1,0], H^{t}(\xR^d))$ and there exists an absolute constant $C>0$ such that
$$
\sup_{z\in [-1,0]} \Vert f(z,\cdot)\Vert_{H^{t}(\xR^d)} \leq C
\Vert f \Vert _{L_z^2(J, H^{t+ \mez}(\xR^d))} +C\Vert \partial_zf \Vert_{L_z^2(J, H^{t- \mez}(\xR^d))}.
$$
 \end{lemm}

\subsection{Para-differential operators}

\subsubsection{Symbolic calculus}
In this section we quote some results 
which concern the symbolic calculus for para-differential operators in the framework of the uniformly local Sobolev spaces. Of course, here, the theory for the classical Sobolev spaces will be  assumed to be known (see \cite{MePise}).

The following technical lemma will be used in the sequel.
 \begin{lemm}\label{techpara}
Let $\chi \in C_0^\infty(\xR^d)$ and $\widetilde{\chi} \in C_0^\infty(\xR^d)$ be equal to one on the support of $\chi$. Let $\psi, \theta \in \mathcal{S}(\xR^d)$. For every $m, \sigma  \in \xR$ one can find a constant $C>0$ such that
\begin{equation}\label{est:tech1}
 \sum_{j\geq -1} \Vert \chi_k \psi(2^{-j}D)((1-\widetilde{\chi}_k)u) \theta(2^{-j}D)v \Vert_{H^m(\xR^d)} \leq C \Vert u \Vert_{H^\sigma_{ul}(\xR^d)}  \Vert v\Vert_{L^\infty(\xR^d)}.
 \end{equation}
For every $m, \sigma, t \in \xR$ one can find a constant $C>0$ such that
  \begin{equation}\label{est:tech2}
 \sum_{j\geq -1} \Vert \chi_k \psi(2^{-j}D)((1-\widetilde{\chi}_k)u) \theta(2^{-j}D)v \Vert_{H^m(\xR^d)} \leq C \Vert u \Vert_{H^\sigma_{ul}(\xR^d)}  \Vert v\Vert_{H^t_{ul}(\xR^d)}.
 \end{equation}
 \end{lemm}
\begin{proof}
We may assume $m\in \xN$. Let us call $A_{k,j}$ the term inside the sum in the left hand side of~\eqref{est:tech1}. Due to   $\chi_k,$ the term  $A_{k,j}$ is a bounded by finite sum of terms of the form 
$$A_{k,j,\alpha} := 2^{jm}\Vert (D^{\alpha_1}\chi_k) \psi_{\alpha_2}(2^{-j}D)((1-\widetilde{\chi}_k)u) \widetilde{\chi}_k \theta_{\alpha_3}(2^{-j}D)v \Vert_{L^\infty }$$
where $\vert \alpha_1 \vert+\vert \alpha_2\vert+\vert \alpha_3 \vert \leq m $ and $\psi_{\alpha_2} = x^{\alpha_2}\psi, \theta_{\alpha_3} = x^{\alpha_3}\theta$.
We are going to show that for large $N\in \xN$ we have
\begin{align*}
&(i) \quad \Vert (D^{\alpha_1}\chi_k) \psi_{\alpha_2}(2^{-j}D)((1-\widetilde{\chi}_k)u) \Vert_{L^\infty } \leq C_N 2^{jM_1(d,\sigma)} 2^{-jN}  \Vert u \Vert_{H^\sigma_{ul} }\\
&(ii) \quad \Vert \widetilde{\chi}_k \theta_{\alpha_3}(2^{-j}D)v \Vert_{L^\infty } \leq C\Vert v\Vert_{L^\infty }\\
&(iii) \quad \Vert \widetilde{\chi}_k \theta_{\alpha_3}(2^{-j}D)v \Vert_{L^\infty } \leq  C2^{jM_2(d,t)}\Vert v\Vert_{H^t_{ul} }
\end{align*}
where, as indicated, $M_j$ are fixed constants depending only on $d,\sigma,t$.
 Then the lemma will follow from these estimates.
 
 To prove $(i)$  we write
 \begin{align*}
 &(D^{\alpha_1} \chi_k(x))  \psi_{\alpha_2}(2^{-j}D)((1-\widetilde{\chi}_k)u)(x) =2^{jd}2^{-jN}\\
 &\times  \big \langle (2^j \vert x-\cdot \vert)^N \widehat{\psi_{\alpha_2}}(2^j(x-\cdot )), \frac{\langle x-\cdot  \rangle^N}{\vert x-\cdot  \vert^N} (D^{\alpha_1}\chi_k(x)) ((1-\widetilde{\chi}_k(\cdot )) \langle x-\cdot  \rangle^{-N} u \big \rangle. 
 \end{align*}
 The function $y \mapsto \frac{\langle x-y \rangle^N}{\vert x-y \vert^N} (D^{\alpha_1}\chi_k(x)) ((1-\widetilde{\chi}_k(y))$ belongs to $W^{\infty, \infty} $ with semi-norms uniformly bounded in $x$. Using the duality $H^{-\sigma}-H^{\sigma}$ we deduce that
\begin{equation*}
\Vert (D^{\alpha_1}\chi_k) \psi_{\alpha_2}(2^{-j}D)((1-\widetilde{\chi}_k)u) \Vert_{L^\infty } 
\leq C_N 2^{jM_1(d,\sigma)} 2^{-jN}  \Vert\langle x-\cdot  \rangle^{-N} u \Vert_{H^\sigma},
\end{equation*}
and we conclude using Lemma \ref{ulN}.
 
 The estimate $(ii)$ is easy. To prove $(iii)$ we take $\check{\chi} \in C_0^\infty(\xR^d)$ equal to one on the support of $\widetilde{\chi}$ and we write
 $$\Vert \widetilde{\chi}_k \theta_{\alpha_3}(2^{-j}D)v \Vert_{L^\infty } \leq \Vert \widetilde{\chi}_k \theta_{\alpha_3}(2^{-j}D)\check{\chi}_kv \Vert_{L^\infty } + \Vert \widetilde{\chi}_k \theta_{\alpha_3}(2^{-j}D)(1-\check{\chi})v \Vert_{L^\infty }.$$
 The second term is bounded exactly by the same method as $(i)$. For the first one we write
 $$\widetilde{\chi}_k \theta_{\alpha_3}(2^{-j}D)\check{\chi}_kv(x) = 2^{jd} \widetilde{\chi}_k(x) \big \langle \widehat{ \theta_{\alpha_3}}(2^j(x-\cdot)), \check{\chi}_k(\cdot)v(\cdot) \big\rangle $$
 and we use the $H^{-t}-H^t$ duality.
\end{proof}
\begin{rema}
$(i)$ Notice that the sames estimates in \eqref{est:tech1}, \eqref{est:tech2} hold if in the left hand side one $2^{-j}$ is replaced by $2^{-j-j_0}$ where $j_0\in \xZ$ is fixed.

$(ii)$ Notice also that in the above proof we have proved that for all real numbers $m,\sigma $, all $N \in \xN$ and all $\psi \in C_0^\infty(\xR^d)$ one can find a positive constant $C_{N,m,\sigma}$ such that 
\begin{equation}\label{est:deltaj}
\Vert \chi_k \psi(2^{-j}D)(1-\widetilde{\chi}_k) u \Vert_{H^m(\xR^d)} \leq C_{N,m,\sigma} 2^{-jN} \Vert u \Vert_{H^\sigma_{ul}(\xR^d)} 
\end{equation}
for every $j\in \xN$ and every $k\in \xZ^d$.

\end{rema}

We introduce now the para-differential calculus.

\begin{defi}\label{def:symb}
Given $m\in \xR, \rho \in [0,1], \Gamma_\rho^m(\xR^d)$ denotes the space of locally bounded functions on $\xR^d \times \xR^d\setminus\{0\}$ which are $C^\infty$ with respect to $\xi,$   such that for all $\alpha \in \xN^d$ the function $x\mapsto \partial_\xi^\alpha a(x,\xi)$ belongs to $W^{\rho,\infty}(\xR^d)$ and there exists a constant $C_\alpha >0$ such that
$$\Vert \partial_\xi^\alpha a (\cdot,\xi) \Vert_{W^{\rho,\infty}(\xR^d)} \leq C_\alpha (1+ \vert \xi \vert)^{m-\vert \alpha \vert}, \quad \forall \vert \xi \vert \geq \mez.$$
For such $a$ we set 
\begin{equation}\label{seminorm}
M^m_\rho(a)= \sup_{\vert \alpha \vert \leq 2d+2}\sup_{\vert \xi \vert \geq \mez} \Vert (1+ \vert \xi \vert)^{ \vert \alpha \vert -m}\partial_\xi^\alpha a (\cdot,\xi)\Vert_{W^{\rho,\infty}(\xR^d)}.
\end{equation}
 
Then $\dot{\Gamma}^m_\rho(\xR^d)$ denotes the subspace of $\Gamma^m_\rho(\xR^d)$ which consists of symbols $a(x,\xi)$ which are homogeneous of degree $m$ with respect to $\xi$.
\end{defi}
Given a symbol $a $ we denote by $T_a$ the associated para-differential operator which is given by the formula
$$\widehat{T_a u}(\xi) = (2 \pi)^{-d}\int_{\xR^d} \theta(\xi-\eta, \eta) \hat{a}(\xi-\eta, \eta) \psi(\eta) \hat{u}(\eta) d \eta $$
where $\hat{a}(\zeta,\eta) = \int_{\xR^d}e^{-ix\cdot \zeta} a(x,\eta) dx$ is the Fourier transform of $a$ with respect to the first variable, $\psi, \theta$ are two fixed $C^\infty$ functions on $\xR^d$ such that for $ 0< \eps_1 < \eps_2$ small enough 
\begin{align}
\psi(\eta) &= 1 \text{ if }\,  \vert \eta \vert \geq 1, \quad \psi(\eta)  = 0 \text{ if } \, \vert \eta \vert \leq \mez \label{cond:psi}\\
\theta(\zeta, \eta) &=1 \text{ if }\,  \vert \zeta \vert \leq \eps_1 \vert \eta \vert, \quad \theta(\zeta, \eta)   =0 \text{ if } \, \vert \zeta \vert \geq \eps_2 \vert \eta \vert \label{cond:theta}.
\end{align}
Notice that if the symbol $a$ is independent of $\xi$   the associated operator $T_a$ is called a paraproduct.
 \begin{theo}\label{calc:symb}
Let $m,m' \in \xR, \rho \in [0, 1]$.

$(i)$  If $a \in \Gamma^m_0(\xR^d),$ then for all $\mu \in \xR$  $T_a$ is continuous from $H_{ul}^\mu(\xR^d)$ to  $H_{ul}^{\mu-m}(\xR^d)$ with norm bounded by $C M_0^m(a)$.

$(ii)$ If $a \in \Gamma^m_\rho(\xR^d)$, $b\in \Gamma^{m'}_\rho(\xR^d)$ then, for all $\mu \in \xR$, $T_a T_b - T_{ab}$ is continuous from $H_{ul}^{\mu }(\xR^d)$ to $H_{ul}^{\mu-m-m' +\rho}(\xR^d)$ with norm bounded by 
$$
C(M^m_\rho(a)M^{m'}_0(b) +M^m_0(a)M^{m'}_\rho(b)).
$$ 

$(iii)$ Let $a \in \Gamma^m_\rho(\xR^d)$. Denote by $(T_a)^*$ the adjoint operator of $T_a$ and by $\overline{a}$ the complex conjugate of $a$. Then for all $\mu \in \xR$ $(T_a)^* -T_{\overline{a}}$ is continuous from  $H_{ul}^{\mu }(\xR^d)$ to  $H_{ul}^{\mu-m+ \rho }(\xR^d)$ with norm bounded by $CM_\rho^m(a)$.
\end{theo}
\begin{proof}
All these points are proved along the same lines. We shall only prove the first one and for simplicity we shall consider symbols in $\dot{\Gamma}^m_\rho(\xR^d)$. We begin by the case where $a$ is a bounded function. Then we write
$$\chi_k T_a u = \chi_k T_a (\widetilde{\chi}_k u) + \chi_kT_a ((1- \widetilde{\chi}_k)u)$$
where $\widetilde{\chi}\in C_0^\infty(\xR^d), \widetilde{\chi} =1$  on the support of $\chi$.
By the classical theory we have
$$ \Vert \chi_k T_a (\widetilde{\chi}_k u) \Vert_{H^\mu } \leq C \Vert a \Vert_{L^\infty } \Vert  \widetilde{\chi}_k u  \Vert_{H^\mu } \leq C \Vert a \Vert_{L^\infty }\Vert u \Vert_{H_{ul}^\mu  }.$$
Now we write
$$ \chi_kT_a ((1- \widetilde{\chi}_k)u)= \sum_{j} \chi_k\{\psi(2^{-j}D)a\} \{ \varphi(2^{-j}D)((1- \widetilde{\chi}_k)u)\}.$$
and  the desired estimate follows immediately from the first inequality in Lemma~\ref{techpara}.

We now assume $a(x,\xi) = b(x) h(\xi)$ where $h(\xi) = \vert \xi \vert^m\widetilde{h}\big (\frac{\xi}{\vert \xi \vert}\big)$ with $\widetilde{h} \in C^\infty(\xS^{d-1})$. Then directly from the definition we have $T_a = T_b \psi(D_x) h(D_x) $ and our estimate in $(i)$ follows from the first step and from the estimate proved in Proposition \ref{pseudoh}
$$\Vert h(D) v \Vert_{H^\mu } \leq C\Vert \widetilde{h}\Vert_{H^{d+1}(\xS^{d-1})}\Vert u \Vert_{H^{ \mu+ m } }.$$

In the last step we introduce $(\widetilde{h}_\nu)_{\nu \in \xN^*}$ an orthonormal basis of $L^2(\xS^{d-1})$ consisting of eigenfunctions of the (self adjoint) Laplace Beltrami operator $\Delta_\omega = \Delta_{\xS^{d-1}} $ on  $L^2(\xS^{d-1})$ i.e. $\Delta_\omega \widetilde{h}_\nu = \lambda^2_\nu \widetilde{h}_\nu$. By the Weyl formula we know that $\lambda_\nu \sim c\nu^{\frac{1}{d}}$. Setting $h_\nu = \vert \xi \vert^m \widetilde{h} (\omega )$, 
$\omega = \frac{\xi}{\vert \xi \vert}$ when $ \xi \neq 0,$ we can write
$$ a(x,\xi) = \sum_{\nu \in \xN^*}b_\nu(x) h_\nu(\xi) \quad \text{where} \quad b_\nu(x) = \int_{\xS^{d-1}} a(x,\omega) \overline{\widetilde{h}_\nu(\omega)} d\omega.$$
Since
$$\lambda_\nu^{2d +2}b_\nu(x) = \int_{\xS^{d-1}} \Delta_\omega^{d+1}a(x,\omega) \overline{\widetilde{h}_\nu(\omega)} d\omega $$
 we deduce that 
\begin{equation} 
\Vert b_\nu \Vert_{L^\infty(\xR^d)} \leq C \lambda_\nu^{-(2d+2)} M^m_0(a). 
\end{equation}
Moreover there exists a positive constant $K$ such that for all $\nu \geq 1$
\begin{equation}\label{hnu}
\Vert \widetilde{h}_\nu \Vert_{H^{d+1}(\xS^{d-1})} \leq K \lambda_\nu^{ d+1}.
\end{equation}
 Now using the  steps above and Proposition~\ref{pseudoh} we obtain
 \begin{align*}
  \Vert T_a u \Vert_{H^\mu_{ul} } &\leq \sum_{\nu \geq 1} \Vert T_{b_\nu}\psi(D_x)h_\nu(D_x) u \Vert_{H^\mu_{ul} }\\
  &\leq C \sum_{\nu \geq 1}\Vert b_\nu \Vert_{L^\infty(\xR^d)}\Vert \widetilde{h}_\nu\Vert_{H^{d+1}(\xS^{d-1})} \Vert u \Vert_{H^{\mu + m}_{ul} }\\
  &\leq M^m_0(a) \Vert u \Vert_{H^{\mu + m}_{ul} } \sum_{ \nu \geq 1}\lambda_\nu^{-(d+1)} 
   \end{align*}
   and $\lambda_\nu^{-(d+1)} \sim c \nu^{-(1+ \frac{1}{d})}$.
\end{proof}
\subsubsection{Paraproducts}
We have the following result of paralinearization of a product.
 \begin{prop}\label{paralin}
 Given two functions $a\in H ^\alpha_{ul}(\xR^d), u\in H^\beta_{ul}(\xR^d)$ with $\alpha + \beta >0$ we can write
 $$au = T_au +T_ua +R(a,u)$$
 with 
 \begin{equation}\label{R(a,u)}
  \Vert R(a,u) \Vert_{H_{ul}^{\alpha + \beta - \frac{d}{2}}(\xR^d)} \leq C\Vert a\Vert_{ H^\alpha_{ul}(\xR^d)} \Vert u\Vert_{ H^\beta_{ul}(\xR^d)}.
  \end{equation}
\end{prop}
\begin{proof}
We have 
$$R(a,u) = \sum_{j\geq -1} \sum_{\vert k-j\vert \leq 1} \varphi(2^{-j}D)a \cdot\varphi(2^{-k}D)u.$$
We take $\chi \in C_0^\infty(\xR^d)$ satisfying \eqref{kiq},  $\widetilde{\chi} \in C_0^\infty(\xR^d)$ equal to one on the support of $\chi$ and we write $a= \widetilde{\chi}_k a + (1-\widetilde{\chi}_k)a, u=  \widetilde{\chi}_k u + (1-\widetilde{\chi}_k)u$. It follows that
$$\chi_kR(a,u) = \chi_kR(\widetilde{\chi}_ka,\widetilde{\chi}_ku) + \chi_kS_k(a,u).$$
The term $ \chi_kR(\widetilde{\chi}_ka,\widetilde{\chi}_ku)$ is estimated by the right hand side of \eqref{R(a,u)} using Theorem 2.11 in \cite{ABZ3}. The remainder $\chi_kS_k(a,u)$ is estimated using \eqref{est:tech2}.
\end{proof}

\begin{prop}\label{a-Ta}
Let $\gamma, r, \mu$ be real numbers such that 
$$r+\mu>0, \quad \gamma \leq r, \quad \gamma<r+\mu - \frac{d}{2}.$$
 There exists a constant $C>0$ such that
 \begin{align}\label{est:a-Ta}
 \Vert (a-T_a)u \Vert_{H^\gamma_{ul}(\xR^d)} \leq C \Vert a \Vert_{H^r_{ul}(\xR^d)}\Vert  u \Vert_{H^\mu_{ul}(\xR^d)} 
 \end{align}
whenever the right hand side is finite.
\end{prop}
\begin{proof}
We write
\begin{align}
\chi_k(a-T_a)u &= \chi_k(\widetilde{\chi}_k a - T_{ \widetilde{\chi}_k a})  \widetilde{\chi}_ku +R_{1,k}u +R_{2,k}u\\
R_{1,k}u &=  \chi_k(\widetilde{\chi}_k a - T_{ \widetilde{\chi}_k a}) (1-\widetilde{\chi}_k)u\\
R_{2,k} &= - \chi_k T_{(1- \widetilde{\chi}_k)a} u = - \chi_k \sum_{j} S_j((1- \widetilde{\chi}_k)a)\Delta_j(u)
\end{align}
where $\widetilde{\chi} \in C_0^\infty(\xR^d)$ is equal to 
one on the support of $\chi$.
According to Proposition 2.12 in~\cite{ABZ3} we have
\begin{equation}\label{a-Ta1}
\Vert  \chi_k(\widetilde{\chi}_k a - T_{ \widetilde{\chi}_k a})  \widetilde{\chi}_ku 
\Vert_{H^{\gamma} } \leq \Vert a \Vert_{H^r_{ul} }\Vert  u \Vert_{H^\mu_{ul} }. 
\end{equation}
Now 
\begin{align*}
R_{1,k} &= \chi_k T_{(1-\widetilde{\chi}_ku)} \widetilde{\chi}_k a  
+ \chi_k \mathcal{R}((1-\widetilde{\chi}_ku,\widetilde{\chi}_k a)\\
&=  \chi_k \sum_{j} S_j((1- \widetilde{\chi}_k)u)\Delta_j(\widetilde{\chi}_k a ) 
+\chi_k  \sum_{\vert i-j \vert \leq 1}\Delta_i((1- \widetilde{\chi}_k)u)\Delta_j(\widetilde{\chi}_k a).
\end{align*}
Therefore we can apply~\eqref{est:tech2} 
in Lemma~\ref{techpara} to $R_{1,k}$ and $R_{2,k}$ 
to conclude that the estimate ~\eqref{est:a-Ta} holds for these terms.
\end{proof}
 
\subsection{On transport equations}
We will be using the following result about solutions of vector fields.
  \begin{lemm}\label{transport2}
Let $I =[0,T]$, $s_0> 1+ \frac{d}{2} $ and $  \mu>0$. Then there exists  $\mathcal{F}: \xR^+ \to \xR^+$ non decreasing  such that for    $V_j\in  L^\infty(I, H^{s_0} (\xR^d))_{ul}\cap L^\infty(I, H^\mu (\xR^d))_{ul}$  $ j =1, \ldots,d,$  $f\in L^1(I, H^\mu (\xR^d))_{ul},$  $ u_0 \in H^\mu_{ul}(\xR^d) $ and any solution $u \in L^\infty(I, H^{s_0} (\xR^d))_{ul}$   of the problem
$$(\partial_t + V\cdot \nabla)u = f, \quad u\arrowvert_{t=0} =u_0 $$
 we have, 
  \begin{multline*}
\Vert u \Vert_{L^\infty(I, H^\mu )_{ul}}  \leq \mathcal{F}\big( T\Vert V\Vert_{L^\infty(I, H^{s_0} )_{ul}} \big) \Big \{ \Vert u_0 \Vert_{ H^\mu_{ul} } +  \Vert f  \Vert_{L^1(I, H^\mu )_{ul}}   \\
+ \sup_{k\in \xZ^d} \Big(\int_0^T\Vert u(\sigma)\Vert_{H^{s_0}_{ul}} \Vert  \widetilde{\chi}_k V(\sigma) \Vert_{H^\mu }\,d\sigma \Big) \Big\} 
\end{multline*}
where $ \widetilde{\chi} \in C_0^\infty(\xR^d)$ is equal to one on the support of $\chi$.
\end{lemm}
\begin{proof}
Set $V_k =\widetilde{\chi}_k V$.  We have 
\begin{equation}\label{transport:0}
  (\partial_t + T_{V_k}\cdot \nabla)(\chi_ku) = \chi_kf + V_k \cdot(\nabla \chi_k)\widetilde{\chi}_k u + (T_{V_k} - V_k) \cdot \nabla(\chi_ku): = g_k. 
  \end{equation}
Now computing the quantity $\frac{d}{dt} \Vert \chi_k u(t) \Vert^2_{L^2},$ using the above equation, the fact that $\Vert (T_V \cdot \nabla + (T_V \cdot \nabla)^*\Vert_{L^2 \to L^2} \leq C \Vert V(t) \Vert_{W^{1,\infty}}$ and the Gronwall inequality we obtain 
\begin{equation}\label{transport:1}
\Vert  \chi_ku(t)\Vert_{L^2 } \leq \mathcal{F}\big(\Vert V \Vert_{L^1(I, W^{1,\infty})}\big)\Big\{ \Vert \chi_k u_0\Vert_{L^2} + \int_0^t \Vert g_k(\sigma)\Vert_{L^2}\, d\sigma\Big\}.
\end{equation}
Now we can write
$$
(\partial_t + T_{V_k}\cdot \nabla)\langle D \rangle^\mu(\chi_ku) = \langle D \rangle^\mu g_k + [T_{V_k}, \langle D \rangle^\mu] \cdot \nabla (\chi_k u).
$$
By the symbolic calculus (see Theorem \ref{calc:symb}, $(ii)$) we have 
$$ \Vert  [T_{V_k}, \langle D \rangle^\mu] \cdot \nabla (\chi_k u)(t)\Vert_{L^2} \leq C \Vert V(t) \Vert_{W^{1,\infty}}\Vert \chi_ku(t) \Vert_{H^ \mu}.$$
Therefore using \eqref{transport:1} and Gronwall inequality we obtain
 \begin{equation}\label{transport:2}
\Vert  \chi_ku(t)\Vert_{H^\mu} \leq \mathcal{F}\big(\Vert V \Vert_{L^1(I, W^{1,\infty})}\big)\Big\{ \Vert  \chi_k u_0\Vert_{H^\mu} + \int_0^t \Vert g_k(\sigma)\Vert_{H^\mu}\, d\sigma\Big\}.
\end{equation}
Coming back to the definition of $g$ given in \eqref{transport:0} we have
$$\Vert V_k\cdot (\nabla \chi_k) \widetilde{\chi}_k u(t) \Vert_{H^\mu} \leq C (\Vert V(t) \Vert_{L^\infty}\Vert \widetilde {\chi}_k u(t)\Vert_{H^\mu} + \Vert u(t) \Vert_{L^\infty}\Vert V_k(t)\Vert_{H^\mu}).$$
 
On the other hand we have 
$$(V_k-T_{V_k})\cdot \nabla (\chi_k u) = T_{\nabla (\chi_ku)} \cdot V_k + R(V_k, \nabla (\chi_k u)).$$
By Theorem \ref{calc:symb} $(i)$ and an easy computation we see that 
\begin{equation*}
  \Vert  T_{\nabla (\chi_ku)(t)} \cdot V_k(t) \Vert_{H^\mu} + \Vert  R(V_k, \nabla (\chi_k u)) \leq C \Vert u(t) \Vert_{W^{1,\infty}} \Vert V_k(t) \Vert_{H^\mu}. 
\end{equation*}
Using \eqref{transport:2}, the Gronwall inequality, the embedding of $H^{s_0}_{ul}$ in $W^{1,\infty}$   and the above estimates we obtain the desired conclusion.
 \end{proof}
\subsection{Commutation with a vector field}

\begin{lemm}\label{commutation}
Let $I =[0,T], V\in C^0(I, W^{1+\eps}(\xR^d))$ for some $\eps>0$ 
and consider a symbol $p = p(t,x,\xi)$ which is homogeneous of order $m$. 
Then there exists a positive constant K (independent of $p,V$) such 
that for any $t\in I$ and any $u\in C^0(I,H^m_{ul}(\xR^d))$ we have
\begin{equation}\label{est:commutation}
\Vert [T_p,\partial_t +T_V\cdot \nabla_x]u(t, \cdot) \Vert_{L^2_{ul}(\xR^d)} 
\leq K C(p,V)\Vert u(t,\cdot) \Vert_{H^m_{ul}(\xR^d)}
\end{equation}
where
$$
C(p,V)\defn M_0^m(p) \Vert V \Vert_{C^0(I,W^{1+ \eps}(\xR^d))} 
+ M^m_0(\partial_t p + V \cdot \nabla_x p) \Vert V \Vert_{L^\infty(I\times \xR^d)}.
$$
\end{lemm}
\begin{proof}
We proceed as in the proof of Theorem~\ref{calc:symb} 
and we begin by the case where $m=0 $ and $p$ is a function. 
We denote by $\mathcal{R}$ the set of continuous 
operators $R(t)$ from $L_{ul}^2(\xR^d)$ to $L^2(\xR^d)$ 
such that $\sup_{t\in I} \Vert R(t)u(t) \Vert_{L^2(\xR^d)}$ 
is bounded by the right hand side of~\eqref {est:commutation}. 
We write
$$ \chi_k [T_p ,\partial_t +T_V\cdot \nabla_x] 
= \chi_k [T_p ,\partial_t +T_V\cdot \nabla_x]  \widetilde{\chi}_k 
+ [\chi_k T_p ,\partial_t +T_V\cdot \nabla_x] (1-\widetilde{\chi}_k) $$
where $\tilde \chi \in C_0^\infty(\xR^d)$ is equal to one on the support of $\chi$.
By Lemma 2.17 in~\cite{ABZ3} the first  operator in the 
right hand side of the above equality  belongs to $\mathcal{R}$. 
Let us look at  the second term. It is equal to 
$$
-\chi_k T_{\partial_t p}(1- \widetilde{\chi}_k) 
+ \chi_k T_p T_V \cdot \nabla_x  (1- \widetilde{\chi}_k) 
- \chi_k T_V \cdot \nabla_x T_p (1- \widetilde{\chi}_k) = : A  +B + C.
$$
We can write
$$
A  = -\chi_k T_{\partial_t p + V\cdot \nabla_x p}(1- \widetilde{\chi}_k)
+\chi_k T_{ V\cdot \nabla_x p }(1- \widetilde{\chi}_k) := A_1 +A_2.
$$
By Theorem~\ref{calc:symb} $(i)$ the term $A_1$ belongs to $\mathcal{R}$. 
Now 
\begin{equation*}
\begin{aligned}
A_2u &= \chi_k T_{\cnx (pV) - p \cnx V  }(1- \widetilde{\chi}_k)u \\
& = \sum_{j\geq-1} \psi(2^{-j}D)(\cnx(pV) - p \cnx V) \chi_k\varphi(2^{-j}D)((1- \widetilde{\chi}_k)u ). 
\end{aligned}
\end{equation*}
Since 
$$\Vert  \psi(2^{-j}D)(\cnx(pV) - p \cnx V)\Vert_{L^\infty }
\leq C 2^j \Vert p \Vert _{L^\infty }\Vert V \Vert_{W^{1,\infty} }$$ 
we deduce from Remark~\ref{est:deltaj} that $A_2 \in \mathcal{R}$. 

Let $\underline{\chi} \in C_0^\infty(\xR^d)$ such that $ \widetilde{\chi} =1$ 
on the support of $\underline{\chi}$ and $\underline{\chi} =1$ on the support of $\chi$. 
We write 
$$
B = \chi_k T_p \underline{\chi}_kT_V \cdot \nabla_x  (1- \widetilde{\chi}_k)  
+ \chi_k T_p (1-\underline{\chi}_k) T_V \cdot \nabla_x  (1- \widetilde{\chi}_k): = B_1 + B_2.
$$
By Theorem~\ref{calc:symb} $(i)$ we have
\begin{equation*}
\begin{aligned}
  \Vert B_1u \Vert_{L^2 }  
& \leq C \Vert p \Vert_{L^\infty } \big\Vert  \underline{\chi}_kT_V \cdot \nabla_x  (1- \widetilde{\chi}_k)u \big\Vert_{L^2 }\\
 &\leq C\Vert p \Vert_{L^\infty } \sum_{j\geq -1}  \big\Vert (\psi(2^{-j}D) V) 2^j \underline{\chi}_k\varphi_1(2^{-j}D)(1- \widetilde{\chi}_k)u \big\Vert_{L^2 }\\
 & \leq C\Vert p \Vert_{L^\infty } \Vert V \Vert_{L^\infty }\sum_{j\geq -1}   2^j\big\Vert  \underline{\chi}_k\varphi_1(2^{-j}D)(1- \widetilde{\chi}_k)u \big\Vert_{L^2 }
\end{aligned}
\end{equation*}
and Remark~\ref{est:deltaj} shows that $B_1 \in \mathcal{R}$. 
 
Now by~\eqref{est:tech1} and Theorem~\ref{calc:symb} we can write 
\begin{equation*}
\begin{aligned}
\Vert B_2u \Vert_{L^2 } &\leq  C \Vert p \Vert_{L^\infty }\big\Vert T_V \cdot \nabla_x(1-\widetilde{\chi}_k)u \big\Vert_{H^{-1}_{ul} } \\
&\leq  C \Vert p \Vert_{L^\infty  }\Vert  V \Vert_{L^\infty } \Vert u \Vert_{L^2_{ul} }
\end{aligned}
\end{equation*}
so $B_2 \in \mathcal{R}$. 
The term $C$ is estimated exactly by the same way, 
introducing a cut-off $ \underline{\chi}_k$ after the operator 
$T_V \cdot \nabla_x$.  Thus $C \in \mathcal{R}$.

The case where $p = a(x)h(\xi)$ and then were $p$ is a general homogeneous symbol of order $m$ is handled as in the proof of Theorem \ref{calc:symb}. 
\end{proof}

\section{Appendix}
 Let $\alpha \in ]0,+\infty[, , \alpha \neq 1$ and $ S(t) = e^{-it\vert D_x \vert^\alpha}$. Our aim is to prove the following result. 
\begin{prop}\label{perte}
Let $s,\sigma \in \xR. $ Assume that there exists $t_0\neq 0$ such that $S(t_0)$ is continuous from $C^\sigma_*(\xR^d)$ to $C^s_*(\xR^d).$ Then $s \leq \sigma -   \frac{d \alpha}{2}$.
 \end{prop}
 \begin{proof}
 Without loss of generality we can assume that $t_0 =-1$.   Our hypothesis reads
 \begin{equation}\label{hypo}
   \exists  C>0 : \Vert S(-1) u \Vert_{C^s_*(\xR^d)} \leq C \Vert  u \Vert_{C^\sigma_*(\xR^d)}, \quad \forall u\in C^\sigma_*(\xR^d).
   \end{equation}
Now if $u\in L^\infty(\xR^d)$ we set $\widehat{\Delta_j u}(\xi) = \varphi(2^{-j}\xi)\widehat{u}(\xi),$  where $\varphi\in C_0^\infty(\xR^d),$ with    $\supp\varphi \subset \{\xi: \mez \leq \vert \xi \vert \leq 2\}$. Then for fixed $j\in \xN$ we have $\Delta_j u \in C^\sigma_*(\xR^d)$ and 
 $$\Vert  \Delta_j u \Vert_{C^\sigma_*(\xR^d)} \leq C 2^{j\sigma}\Vert \Delta_j u\Vert_{L^\infty(\xR^d)} \leq C' 2^{j\sigma}\Vert u\Vert_{L^\infty(\xR^d)} .$$
This follows from the fact that $\Vert  \Delta_j u \Vert_{C^\sigma_*(\xR^d)}= \sup_{k\in \xN}2^{k\sigma}\Vert \Delta_k \Delta_j u\Vert_{L^\infty(\xR^d)} $ and the fact that $\Delta_k \Delta_j = 0$ if $\vert j-k\vert \geq 2$.  
Since $\Delta_j$ commutes with $S(-1)$,  we see    that
$$ 2^{js}\Vert S(-1) \Delta_j \Delta_j u\Vert_{L^\infty(\xR^d)} \leq    \Vert S(-1)\Delta_j u \Vert_{C^s_*(\xR^d)}.$$
 It follows from \eqref{hypo} applied to $\Delta_ju$ with $u\in L^\infty(\xR^d)$ that   one can find a positive constant $C$ such that
 \begin{equation}\label{hypo2}
   2^{js}\Vert S(-1) \Delta_j \Delta_j u\Vert_{L^\infty(\xR^d)} \leq C2^{j\sigma}\Vert  u\Vert_{L^\infty(\xR^d)} \quad \forall u\in L^\infty(\xR^d), \quad \forall j\in \xN.
   \end{equation}

 Let us set $T =  S(-1) \Delta _j  \Delta _j$. Then 
 $$T  u(x) = (2\pi)^{-d}\iint e^{i  [(x-y)\cdot \xi  + \vert \xi \vert^\alpha ]} \varphi^2(2^{-j}\xi)   u (y)\,dy\,d\xi.$$
 We shall set $h =2^{-j}$ and take $j$ large enough. Then setting $\eta = h \xi$ we obtain
 $$T  u(x) = \int_{\xR^d} K_h(x-y)   u(y) \, dy$$
 where 
 $$ K_h(z) = (2\pi h)^{-d}\int_{\xR^d} e^{ \frac{i}{h}(z\cdot \eta + h^{1-\alpha} \vert \eta \vert^\mez)} \varphi^2(\eta)\, d\eta.$$
 We shall use the following well known lemma.
 \begin{lemm}
 Let $K\in C^0(\xR^d\times\xR^d)$ be such that $\sup_{x \in \xR^d}\int \vert K(x,y)\vert\, dy <+\infty$.
 Then the operator $T$ defined by $Tu(x) = \int K(x,y) u(y) \, dy$ is continuous from $L^\infty(\xR^d) $ to $L^\infty(\xR^d)$ and $\Vert T \Vert_{L^\infty \to L^\infty} = \sup_{x \in \xR^d}\int \vert K(x,y)\vert\, dy$.
\end{lemm}
 It follows from this lemma that in our case we have
 $$\Vert T \Vert_{L^\infty \to L^\infty} = \int_{\xR^d}\vert K_h(z)\vert \, dz.$$
 Setting $z =h^{1-\alpha} s$ and $\widetilde{K}_h(s) = K_h(h^{1-\alpha} s)$ we find that 
 \begin{equation}\label{normT}
 \Vert T \Vert_{L^\infty \to L^\infty} = h^{d(1-\alpha)}\int_{\xR^d}\vert\widetilde{K}_h(s)\vert \, ds
 \end{equation}
 with
 \begin{equation}\label{Ktilde}
 \widetilde{K}_h(s) = (2\pi h)^{-d}\int_{\xR^d}\ e^{ {i} h^{-\alpha} \phi_(s,\eta)} \varphi^2(\eta)\, d\eta, \quad \phi (s,\eta) =  s\cdot \eta +  \vert \eta \vert^\alpha. 
 \end{equation}
 Recall that $\supp \varphi \subset \{\eta: \mez \leq \vert \eta \vert \leq 2 \}$. We have $\frac{\partial \phi}{\partial \eta} = s + \alpha \frac{\eta}{\vert \eta \vert^{2-\alpha}}$.
 
Case 1:   $\vert s \vert \leq \mez \frac{\alpha}{2^{\vert 1-\alpha \vert }}$. 
Here, on the support of $\varphi,$ we have
$$
\la \frac{\partial \phi}{\partial \eta}(s,\eta)\ra \geq  \frac{\alpha}{\vert \eta \vert^{1-\alpha}} - \vert s \vert \geq  \mez \frac{\alpha}{2^{\vert 1-\alpha \vert }}.
$$
Therefore integrating by parts in the right hand side of \eqref{Ktilde} 
using the vector field $L = \frac{h^\alpha}{i} \frac{1}{\vert \partial_\eta \phi \vert^2}\sum_{k=1}^d \frac{\partial\phi}{\partial \eta_k} \frac{\partial}{\partial \eta_k}$ we obtain
\begin{equation}\label{est2}
\vert \widetilde{K}_h(s) \vert \leq C_N h^N, \quad \forall N\in \xN.
\end{equation}

Case 2:   $\vert s \vert \geq   2^{1+\vert \alpha-1 \vert}\alpha $. On the support of $\varphi$ we have
$$
\la \frac{\partial \phi}{\partial \eta}(s,\eta)\ra \geq \vert s \vert 
-  \frac{\alpha}{\vert \eta \vert^{1-\alpha}} \geq   2^{ \vert \alpha-1 \vert}\alpha.
   $$
Then using the same vector field as in the first case and noticing that $\partial_\eta^\alpha \phi$ 
is independent of $s$ when $\vert \alpha \vert \geq 2$ we obtain
\begin{equation}\label{est2-b}
\vert \widetilde{K}_h(s) \vert \leq C_N \vert s \vert ^{-N}h^N, \quad \forall N\in \xN.
\end{equation}

Case 3: $\mez \frac{\alpha}{2^{\vert 1-\alpha \vert }} \leq \vert s \vert \leq 2^{1+\vert \alpha-1 \vert}\alpha$. 
Here the function $\phi$ has a critical point given by $   \frac{\eta}{\vert \eta \vert^{2-\alpha}} = -\frac{s}{\alpha}$. 
It follows that $\frac{1}{\vert \eta \vert^{1-\alpha}} = \frac{\vert s \vert}{\alpha},$ 
which implies that $\eta_c = c_\alpha  s  \vert s \vert^{\frac{2-\alpha}{\alpha-1}} $. 
Moreover we have 
$$
\frac{\partial^2\phi}{\partial \eta_j \partial \eta_k} = \alpha  \vert \eta \vert^{\alpha-2}  m_{jk} 
\quad m_{jk}=   \delta_{jk} - (\alpha-2) \omega_{j }\omega_{k}, \quad \omega= \frac{\eta }{\vert \eta \vert }.
$$
 Since $\det(m_{jk}) = c_0 \neq 0$ we obtain $ \Big(\la\det\Big(\frac{\partial^2\phi}{\partial \eta_j \partial \eta_k}(s, \eta_c)  \Big)\ra\Big)^\mez = c_{\alpha,d} \vert s \vert ^{\frac{(\alpha-2)d}{2(\alpha-1)}}$. The stationnary phase formula implies that there exists $C_d>0$ such that
 \begin{equation}\label{est3}
  \widetilde{K}_h(s) =  C_{\alpha,d} h ^{-d} h^{ \frac{\alpha d}{2}}\Big\{\frac{e^{i\phi(s, \eta_c)}}{\vert s \vert ^{\frac{(\alpha-2)d}{2(\alpha-1)}}}\varphi ^2\big(\eta_c\big) + \mathcal{O}(h^\alpha)\Big\}.
  \end{equation}
  Using \eqref{normT}, \eqref{est2}, \eqref{est2-b}, \eqref{est3} we can conclude that 
  $$ \Vert T \Vert_{L^\infty \to L^\infty} \geq  Ch^{d(1-\alpha)} h^{-d} h^{\frac{d\alpha}{2}} -C_N h^N\geq C' h^{-\frac{d\alpha}{2}} .$$
  Recalling that $h = 2^{-j}$ we obtain
  $$\Vert T \Vert_{L^\infty \to L^\infty} \geq C  2^{j\frac{ d\alpha}{2}}.$$
  Thus for any $\eps>0$ one can find $u_0 \in L^\infty(\xR^d), $ non identically zero, such that 
  $$\Vert Tu_0\Vert_{L^\infty(\xR^d)} \geq  (C 2^{j\frac{ d\alpha}{2}} - \eps) \Vert u_0 \Vert_{L^\infty(\xR^d)}.$$
  Taking $\eps$ small and using \eqref{hypo2} with $u_0$ we obtain
  $$C''_d 2^{j(s+ \frac{d\alpha }{2})}\Vert u_0 \Vert_{L^\infty(\xR^d)} \leq C2^{j\sigma}\Vert u_0 \Vert_{L^\infty(\xR^d)} $$
  for all $j \geq j_0$ large enough which proves the Proposition.
\end{proof}


\addcontentsline{toc}{section}{Bibliography}







\begin{thebibliography}{10}
\small

\bibitem{ABZ1}
Thomas Alazard, Nicolas Burq and Claude Zuily.
\newblock On the water-wave equations with surface tension.
\newblock {\em Duke Math. J.}, 158(3):413--499, 2011.

\bibitem{ABZ2}
Thomas Alazard, Nicolas Burq and Claude Zuily.
\newblock Strichartz estimates for water waves.
\newblock {\em Ann. Sci. {\'E}c. Norm. Sup{\'e}r. (4)},   t.44, 855--903, 2011.

\bibitem{ABZ3}
Thomas Alazard, Nicolas Burq and Claude Zuily.
\newblock On the Cauchy problem for gravity water waves.
\newblock arXiv:1212.0626.

\bibitem{ABZ4}
Thomas Alazard, Nicolas Burq and Claude Zuily.
\newblock Strichartz estimates for gravity water waves.
\newblock {\em  }


\bibitem{AM}
Thomas Alazard and Guy M{\'e}tivier.
\newblock Paralinearization of the {D}irichlet to {N}eumann operator, and
  regularity of three-dimensional water waves.
\newblock {\em Comm. Partial Differential Equations}, 34(10-12):1632--1704,
  2009.
\bibitem{Alipara}
Serge Alinhac.
\newblock Paracomposition et op\'erateurs paradiff\'erentiels.
\newblock {\em Comm. Partial Differential Equations}, 11(1):87--121, 1986.
\bibitem{Ali}
Serge Alinhac.
\newblock Existence d'ondes de rar\'efaction pour des syst\`emes
  quasi-lin\'eaires hyperboliques multidimensionnels.

\bibitem{ASL}
Borys Alvarez-Samaniego and David Lannes.
\newblock Large time existence for 3{D} water-waves and asymptotics.
\newblock {\em Invent. Math.}, 171(3):485--541, 2008.

\bibitem{AmMa}
David~M. Ambrose and Nader Masmoudi.
\newblock The zero surface tension limit of two-dimensional water waves.
\newblock {\em Comm. Pure Appl. Math.}, 58(10):1287--1315, 2005.

\bibitem{BL}
Claude Bardos and David Lannes.
\newblock Mathematics for 2d interfaces.
\newblock Panorama et Synth{\`e}ses, to appear.

\bibitem{BeBrOl}
T.~Brooke Benjamin and Peter~J. Olver.
\newblock Hamiltonian structure, symmetries and conservation laws for water
  waves.
\newblock {\em J. Fluid Mech.}, 125:137--185, 1982.

\bibitem{Bony}
Jean-Michel Bony.
\newblock Calcul symbolique et propagation des singularit\'es pour les
  \'equations aux d\'eriv\'ees partielles non lin\'eaires.
\newblock {\em Ann. Sci. \'Ecole Norm. Sup. (4)}, 14(2):209--246, 1981.
\bibitem{Boussinesq}
Joseph~Boussinesq.
\newblock Sur une importante simplification de la th\'eorie des ondes que
  produisent, \`a la surface d'un liquide, l'emersion d'un solide ou
  l'impulsion d'un coup de vent.
\newblock {\em Ann. Sci. \'Ecole Norm. Sup. (3)}, 27:9--42, 1910.
\bibitem{CCFGGS}
Angel Castro, Diego C{\'o}rdoba, Charles Fefferman, Francisco Gancedo, and
  Javier G{\'o}mez-Serrano.
\newblock Splash singularity for water waves.
\newblock Preprint 2011.

\bibitem{CoCoGa}
Antonio Cordoba, Diego C{\'o}rdoba, and Francisco Gancedo.
\newblock The {R}ayleigh-{T}aylor condition for the evolution of irrotational
  fluid interfaces.
\newblock {\em Proc. Natl. Acad. Sci. USA}, 106(27):10955--10959, 2009.

\bibitem{CL}
Angel Castro and David Lannes. 
\newblock Well-posedness and shallow-water stability for a new Hamiltonian formulation of the water waves equations with vorticity. 
\newblock  arXiv:1402.0464.


 \bibitem{Craig1985}
Walter Craig.
\newblock {An existence theory for water waves and the Boussinesq and
  Korteweg-deVries scaling limits}.
\newblock {\em Communications in Partial Differential Equations},
  10(8):787--1003, 1985.
 \bibitem{ABZ3CN}
Walter~Craig and David~P. Nicholls.
\newblock Travelling two and three dimensional capillary gravity water waves.
\newblock {\em SIAM J. Math. Anal.}, 32(2):323--359 (electronic), 2000.

 
\bibitem{CrSuSu}
W.~Craig, C.~Sulem, and P.-L. Sulem.
\newblock Nonlinear modulation of gravity waves: a rigorous approach.
\newblock {\em Nonlinearity}, 5(2):497--522, 1992.


\bibitem{DaPr}
A.-L. Dalibard and C. Prange. 
\newblock{
Well-posedness of the Stokes-Coriolis system in the half-space over a rough surface. }
\newblock{arXiv:1304.6651}

\bibitem{Favre}
H. Favre.
\newblock {Etude th{\'e}orique et exp{\'e}rimentale des ondes de translation dans les canaux d{\'e}couverts}
\newblock  Dunod, Paris, 1935.

\bibitem{FS}
Charles Fefferman and Elias Stein.
\newblock {$H^p$ spaces of several variables}.
\newblock {\em Acta Mathematica}, 129(1):137--193, 1972.

\bibitem{GVMa}
David G{\'e}rard-Varet and Nader Masmoudi.
\newblock{ Relevance of the slip condition for fluid flows near an irregular boundary,}
\newblock{ \em  Comm. Math. Phys.} 295 (2010), no. 1, 99Ð137.

\bibitem{GMS}
Pierre Germain, Nader Masmoudi, and Jalal Shatah.
\newblock Global solutions for the gravity water waves equation in dimension 3.
\newblock {\em Ann. of Math. (2)}, 175(2):691--754, 2012.

\bibitem{GRI}
Pierre Grisvard.
\newblock Elliptic problems in non smooth domains.
\newblock Pitman, 1985.

\bibitem{HO}
Lars H\" ormander.
\newblock The Analysis of Linear Partial Differential Operators III.
\newblock Springer Verlag 1985.



\bibitem{IP}
G{\'e}rard Iooss and Pavel~I. Plotnikov.
\newblock Small divisor problem in the theory of three-dimensional water
  gravity waves.
\newblock {\em Mem. Amer. Math. Soc.}, 200(940):viii+128, 2009.


\bibitem{Kato}
Tosio Kato.
\newblock The {C}auchy problem for quasi-linear symmetric hyperbolic systems.45

\newblock {\em Arch. Rational Mech. Anal.}, 58(3):181--205, 1975.


\bibitem{LannesLivre}
David Lannes.
\newblock The water waves problem. Mathematical analysis and asymptotics. 
\newblock Mathematical Surveys and Monographs, 188. {\em American Mathematical Society, Providence, RI,} 2013. xx+321 pp.

\bibitem{LannesJAMS}
David Lannes.
\newblock Well-posedness of the water-waves equations.
\newblock {\em J. Amer. Math. Soc.}, 18(3):605--654 (electronic), 2005.


\bibitem{LannesKelvin}
David Lannes.
\newblock A stability criterion for two-fluid interfaces and applications.
\newblock {\em Arch. Ration. Mech. Anal.}, 208(2):481--567, 2013.


\bibitem{LindbladAnnals}
Hans Lindblad.
\newblock Well-posedness for the motion of an incompressible liquid with free
  surface boundary.
\newblock {\em Ann. of Math. (2)}, 162(1):109--194, 2005.

\bibitem{Lions}
Jacques-Louis Lions, Enrico Magenes.
\newblock  Probl\`emes aux limites non homog\`enes
\newblock {Vol 1.Dunod}, 1968.





\bibitem{MePise}
Guy Metivier 
  \newblock {\em Para-differential calculus and applications to the {C}auchy
  problem for nonlinear systems}, volume~5 of {\em Centro di Ricerca Matematica
  Ennio De Giorgi (CRM) Series}.
\newblock Edizioni della Normale, Pisa, 2008.
\bibitem{ReSh}
John Reeder and Marvin Shinbrot.
\newblock Three-dimensional, nonlinear wave interaction in water of constant
  depth.
\newblock {\em Nonlinear Anal.}, 5(3):303--323, 1981.


\bibitem{WuJAMS}
Sijue Wu.
\newblock Well-posedness in {S}obolev spaces of the full water wave problem in
  3-{D}.
\newblock {\em J. Amer. Math. Soc.}, 12(2):445--495, 1999.


\bibitem{Wu09}
Sijue Wu.
\newblock Almost global well-posedness of the 2-{D} full water wave problem.
\newblock {\em Invent. Math.}, 177(1):45--135, 2009.

\bibitem{Wu10}
Sijue Wu.
\newblock Global wellposedness of the 3-{D} full water wave problem.
\newblock {\em Invent. Math.}, 184(1):125--220, 2011.


\bibitem{Zakharov1968}
Vladimir~E. Zakharov.
\newblock Stability of periodic waves of finite amplitude on the surface of a
  deep fluid.
\newblock {\em Journal of Applied Mechanics and Technical Physics},
  9(2):190--194, 1968.

\end{thebibliography}
\end{document}